\documentclass[10pt]{article}

\usepackage[margin=1.3in]{geometry}

\usepackage[T1]{fontenc}
\usepackage[utf8]{inputenc}
\usepackage[USenglish, UKenglish, english]{babel}
\usepackage{amssymb}
\usepackage{amsthm}
\usepackage{amsmath}
\usepackage[mathscr]{euscript}
\usepackage{enumitem}
\usepackage{graphicx}
\usepackage{MnSymbol}
 \let\mathscr\relax
\usepackage[scr]{rsfso}
\usepackage{tikz-cd}
\usepackage{tikz}
\usetikzlibrary{matrix,arrows,decorations.pathmorphing}
\usepackage{hyperref}
\usepackage{marginnote}
\usetikzlibrary{arrows.meta}

\usepackage{libertine}

\theoremstyle{definition}
\newtheorem{defin}{Definition}[section]
\theoremstyle{definition}

\theoremstyle{plain}
\newtheorem{theo}[defin]{Theorem}
\theoremstyle{plain}
\newtheorem{prop}[defin]{Proposition}
\theoremstyle{plain}
\newtheorem{lem}[defin]{Lemma}
\theoremstyle{plain}
\newtheorem{cor}[defin]{Corollary}
\theoremstyle{definition}
\newtheorem{rmk}[defin]{Remark}
\theoremstyle{definition}

\theoremstyle{definition}

\theoremstyle{plain}

\theoremstyle{definition}
\newtheorem{notation}[defin]{Notation}
\theoremstyle{definition}
\newtheorem{hyp}[defin]{Hypothesis}

\theoremstyle{definition}
\newtheorem{cond}[defin]{Condition}

\theoremstyle{definition}
\newtheorem{ass}[defin]{Assumption}

\theoremstyle{definition}
\newtheorem*{defin*}{Definition}
\theoremstyle{definition}
\newtheorem*{ex*}{Example}
\theoremstyle{plain}
\newtheorem*{theo*}{Theorem}
\theoremstyle{plain}
\newtheorem*{prop*}{Proposition}
\theoremstyle{plain}
\newtheorem*{lem*}{Lemma}
\theoremstyle{plain}
\newtheorem*{cor*}{Corollary}
\theoremstyle{definition}
\newtheorem*{rmk*}{Remark}
\theoremstyle{definition}
\newtheorem*{exe*}{Exercise}

\theoremstyle{plain}
\newtheorem{theoA}{Theorem}

\theoremstyle{plain}

\theoremstyle{plain}
\newtheorem{condA}[theoA]{Condition}

\theoremstyle{plain}

\numberwithin{equation}{section}

\usepackage{parskip}

\makeatletter
\def\thm@space@setup{%
  \thm@preskip=\parskip \thm@postskip=0pt
}
\makeatother

\setlist[enumerate]{label=(\roman*)}

\def\R{{\mathbf{R}}} 

\def\Bl{{\rm Bl}}

\def\bl{{\rm bl}}

\def\irr{{\rm Irr}}

\def\ker{{\rm Ker}}

\def\aut{{\rm Aut}}

\def\out{{\rm Out}}

\def\cl{{\mathfrak{Cl}}}

\def\GL{{\rm GL}}

\def\n{{\mathbf{N}}} 

\def\c{{\mathbf{C}}} 
 
\def\z{{\mathbf{Z}}} 

\def\O{{\mathbf{O}}} 

\def\E{{\mathcal{E}}} 

\def\G{{\mathbf{G}}} 

\def\H{{\mathbf{H}}} 

\def\K{{\mathbf{K}}} 

\def\L{{\mathbf{L}}} 

\def\M{{\mathbf{M}}} 

\def\T{{\mathbf{T}}} 

\def\S{{\mathbf{S}}} 

\def\B{{\mathbf{B}}} 

\def\P{{\mathbf{P}}} 

\def\Q{{\mathbf{Q}}} 

\DeclareMathOperator{\HC}{\mathcal{H}\HCkern \mathcal{C}}
\newcommand{\HCkern}{%
  \mkern-1.0mu
  \mathchoice{}{}{\mkern0.2mu}{\mkern0.5mu}%
}

\makeatletter
\newcommand{\uset}[3][0ex]{%
  \mathrel{\mathop{#3}\limits_{
    \vbox to#1{\kern-7\ex@
    \hbox{$\scriptstyle#2$}\vss}}}}
\makeatother

\newcommand{\wt}[1]{\widetilde{#1}} 

\newcommand{\wh}[1]{\widehat{#1}}

\newcommand{\ws}[1][1.5]{
  \mathrel{\scalebox{#1}[1]{$\sim$}}
}

\newcommand{\iso}[1]{\ws_{#1}}

\newcommand{\isoc}[1]{\ws_{#1}^c}

\newcommand{\cusp}[1]{{\rm Cusp}_e({#1})}

\usepackage{sectsty}
\allsectionsfont{\centering}

\makeatletter
\def\blfootnote{\gdef\@thefnmark{}\@footnotetext}
\makeatother

\title{
{\huge\bf 
Inductive local-global conditions and generalized Harish-Chandra theory}\\
\author{\Large Damiano Rossi}
\date{}
\blfootnote{\emph{$2010$ Mathematical Subject Classification:} $20$C$15$, $20$C$20$, $20$C$33$.
\\
\emph{Key words and phrases:} Finite groups of Lie type, generalized Harish-Chandra theory, $e$-Harish-Chandra series.
\\
The content of this paper is part of the author's doctoral thesis written in the framework of the research training group \textit{GRK2240: Algebro-geometric Methods in Algebra, Arithmetic and Topology} funded by the DFG. This work is also supported by the EPSRC grant EP/T$004592/1$. The author would like to thank Britta Sp\"ath for proposing a fascinating research topic and for providing countless suggestions, Julian Brough for multiple discussions on the extendibility of characters of $e$-split Levi subgroups and Gunter Malle for helpful comments on generalized Harish-Chandra theory and for a thorough reading of an earlier version of this paper.
}
}

\begin{document}

\renewcommand{\thetheoA}{\Alph{theoA}}

\renewcommand{\thecorA}{\Alph{corA}}

\renewcommand{\thecondA}{\Alph{condA}}

\selectlanguage{english}

\maketitle

\begin{abstract}

We work towards a version of generalized Harish-Chandra theory compatible with Clifford theory and with the action of automorphisms on irreducible characters. This provides a fundamental tool to verify the inductive conditions for the so-called local-global conjectures in representation theory of finite groups in the crucial case of groups of Lie type in non-defining characteristic. In particular, as shown by the author in an earlier paper, this has a strong impact on the verification of the inductive condition for Dade's Conjecture. As a by-product, we also show how to extend the parametrization of generalized Harish-Chandra series given by Broué--Malle--Michel to the non-unipotent case by assuming maximal extendibility. 
\end{abstract}

\tableofcontents

\section{Introduction}


Generalized Harish-Chandra theory is a powerful tool that lies at the heart of many results in modular representation theory of finite groups of Lie type in non-defining characteristic. First introduced by Fong--Srinivasan \cite{Fon-Sri86} for classical groups and then fully developed by Broué--Malle--Michel \cite{Bro-Mal-Mic93} in the unipotent case, this theory extends the classical Harish-Chandra theory formulated by Harish-Chandra \cite{Har-Cha70} and further studied by Howlett--Lehrer \cite{How-Leh80} by replacing Harish-Chandra induction with Deligne--Lusztig induction. The aim of this paper is to extend generalized Harish-Chandra theory, and in particular the parametrization of generalized Harish-Chandra series given in \cite[Theorem 3.2]{Bro-Mal-Mic93}, to the non-unipotent case. In doing so, we also introduce a new compatibility with Clifford theory and the action of automorphisms. For quasi-simple groups of Lie type in non-defining characteristic, this suggests an explanation for the validity of the so-called inductive conditions for the long-standing local-global conjectures in representation theory of finite groups. Together with the main results of \cite[Section 4]{Ros-Generalized_HC_theory_for_Dade}, our results provide a uniform parametrization of characters of groups of Lie type under suitable assumptions. 

More precisely, let $\G$ be a connected reductive group defined over an algebraic closure $\mathbb{F}$ of a field of characteristic $p$ and $F:\G\to \G$ a Frobenius endomorphism endowing $\G$ with an $\mathbb{F}_q$-structure for some power $q$ of $p$. Fix a prime $\ell$ different from $p$ and denote by $e$ the multiplicative order of $q$ modulo $\ell$ or $q$ modulo $4$ if $\ell=2$. All modular representation theoretic notions are considered with respect to the prime $\ell$. For an $e$-cuspidal pair $(\L,\lambda)$ of $\G$, we denote by $\E(\G^F,(\L,\lambda))$ the $e$-Harish-Chandra series associated to $(\L,\lambda)$ and by $W_\G(\L,\lambda)^F:=\n_\G(\L)^F_\lambda/\L^F$ the corresponding relative Weyl group. The parametrization of $e$-Harish-Chandra series associated to unipotent $e$-cuspidal pairs $(\L,\lambda)$ given in \cite[Theorem 3.2]{Bro-Mal-Mic93} shows the existence of a bijection
\begin{equation}
\label{eq:e-HC series and relative Weyl group}
I^\G_{(\L,\lambda)}:\irr\left(W_\G(\L,\lambda)^F\right)\to \E(\G^F,(\L,\lambda)).
\end{equation}
Our first result extends this parametrization to non-unipotent $e$-Harish-Chandra series in groups with connected centre. Moreover, we show that these bijections can be chosen to satisfy certain additional properties. In what follows, we denote by $\aut_\mathbb{F}(\G^F)$ the set of those automorphisms of $\G^F$ which are obtained by restriction from bijective endomorphisms of $\G$ commuting with $F$ (see Section \ref{sec:Automorphisms}).

\begin{theoA}
\label{thm:Main e-Harish-Chandra parametrization for connected center}
If $\hspace{1pt}\G$ has connected centre and $[\G,\G]$ is simple not of type $\mathbf{E}_6$, $\mathbf{E}_7$ or $\mathbf{E}_8$, then there exist a collection of bijections
\[I^{\K}_{(\L,\lambda)}:\irr\left(W_{\K}\left(\L,\lambda\right)^F\right)\to\E\left(\K^F,\left(\L,\lambda\right)\right)\]
where $\K$ runs over the set of $F$-stable Levi subgroups of $\G$ and $(\L,\lambda)$ over the set of $e$-cuspidal pairs of $\K$, such that:
\begin{enumerate}
\item $I^{\K}_{(\L,\lambda)}$ is $\aut_\mathbb{F}(\G^F)_{\K,(\L,\lambda)}$-equivariant;
\item $I^{\K}_{(\L,\lambda)}(\eta)(1)_\ell=\left|\K^F:\n_{\K}(\L,\lambda)^F\right|_\ell\cdot\lambda(1)_\ell\cdot \eta(1)_\ell$ for every $\eta\in\irr(W_{\K}(\L,\lambda))^F$; and
\item if $z\in\z(\K^{*F^*})$ corresponds to characters $\wh{z}_{\L}\in\irr(\L^F/[\L,\L]^F)$ and $\wh{z}_{\K}\in\irr(\K^F/[\K,\K]^F)$ (see Section \ref{sec:Regular embeddings}), then $\lambda\cdot\wh{z}_{\L}$ is $e$-cuspidal, $W_{\K}(\L,\lambda)^F=W_{\K}(\L,\lambda\cdot\wh{z}_{\L})^F$ and
\[I^{\K}_{(\L,\lambda)}\left(\eta\right)\cdot \wh{z}_{\K}=I^{\K}_{(\L,\lambda\cdot \wh{z}_{\L})}\left(\eta\right)\]
for every $\eta\in\irr(W_{\K}(\L,\lambda))^F$.
\end{enumerate}
\end{theoA}

In Theorem \ref{thm:Compatibility with DL-induction} we consider similar bijections $\mathcal{I}^\K_{(\L,\lambda)}$ and obtain \emph{$e$-Harish-Chandra theory} (as defined in \cite[Definition 2.9]{Kes-Mal13}) above any $e$-cuspidal pair in groups with connected centre. Notice that the restrictions on the type of $\G$ are mainly due to the fact that the Mackey formula is not known to hold in full generality. In addition for types $\mathbf{E}_6$, $\mathbf{E}_7$ and $\mathbf{E}_8$ it is not known whether there exists a Jordan decomposition map which commutes with Deligne--Lusztig induction. If these two properties were to be established for types ${\bf{E}}_6$, ${\bf{E}}_7$ or ${\bf{E}}_8$, then Theorem \ref{thm:Main e-Harish-Chandra parametrization for connected center} would also hold in the excluded cases.

Before proceeding further, we make a remark inspired by \cite{Mal07} and \cite{Mal14}. If $\lambda$ has an extension $\wh{\lambda}$ to $\n_\G(\L)^F_\lambda$, then Gallagher's theorem and the Clifford correspondence implies that there is a bijection
\begin{align*}
\irr\left(W_\G(\L,\lambda)^F\right)&\to\irr\left(\n_\G(\L)^F\enspace\middle|\enspace\lambda\right)
\\
\eta&\mapsto \left(\wh{\lambda}\eta\right)^{\n_\G(\L)^F}.
\end{align*}
In this case, \eqref{eq:e-HC series and relative Weyl group} is equivalent to the following bijection
\begin{equation}
\label{eq:e-HC series, parametrization with normalizer}
\Omega^\G_{(\L,\lambda)}:\E(\G^F,(\L,\lambda))\to\irr\left(\n_\G(\L)^F\enspace\middle|\enspace\lambda\right).
\end{equation}
Observe that the extendibility of the character $\lambda$ to its stabilizer $\n_\G(\L)^F_\lambda$ is expected to hold in general and is known in a plethora of cases.

Working with the formulation given in \eqref{eq:e-HC series, parametrization with normalizer} and using Theorem \ref{thm:Main e-Harish-Chandra parametrization for connected center}, we obtain a parametrization of $e$-Harish-Chandra series for groups with non-connected centre by assuming maximal extendibility for certain characters of $e$-split Levi subgroups. Recall that if $i:\G\to\wt{\G}$ is a regular embedding and $\L$ is a Levi subgroup of $\G$, then $\wt{\L}:=\L\z(\wt{\G})$ is a Levi subgroup of $\wt{\G}$. Moreover, for any connected reductive group $\H$ with Frobenius endomorphism $F$, we denote by $\cusp{\H^F}$ the set of (irreducible) $e$-cuspidal characters of $\H^F$.

\begin{theoA}
\label{cor:Main bijections for nonconnected centre}
Let $\G$ be simple, simply connected not of type $\mathbf{E}_6$, $\mathbf{E}_7$ or $\mathbf{E}_8$ and consider $\ell\in\Gamma(\G,F)$ with $\ell\geq 5$. Let $\K$ be an $F$-stable Levi subgroup of $\G$, $(\L,\lambda)$ an $e$-cuspidal pair of $\K$ and suppose there exists an $(\aut_\mathbb{F}(\G^F)_{\K,\L}\ltimes \irr(\wt{\G}^F/\G^F))$-equivariant extension map for $\cusp{\wt{\L}^F}$ with respect to $\wt{\L}^F\unlhd \n_{\wt{\K}}(\L)^F$ (see the discussion following Definition \ref{def:Maximal extendibility}). Then, there exists an $\aut_\mathbb{F}(\G^F)_{\K,(\L,\lambda)}$-equivariant bijection
\[\Omega_{(\L,\lambda)}^\K:\E\left(\K^F,(\L,\lambda)\right)\to\irr\left(\n_{\K}(\L)^F\enspace\middle|\enspace\lambda\right)\]
that preserves the $\ell$-defect of characters.
\end{theoA}

The extendibility condition considered in Theorem \ref{cor:Main bijections for nonconnected centre} should be compared with condition ${\rm{B}(d)}$ of \cite[Definition 2.2]{Cab-Spa19}. In particular, this has been shown to hold for groups of type $\mathbf{A}$ and $\mathbf{C}$ (see \cite{Bro-Spa20} and \cite{Bro22} respectively) and is expected to hold in all cases. Due to these results, in Section \ref{sec:Stabilizers, extendibility and consequences} we obtain consequences for groups of type $\mathbf{A}$ and $\mathbf{C}$ (see Corollary \ref{cor:Bijections for nonconnected centre, type A and C}).


Most importantly, by working with the formulation given in \eqref{eq:e-HC series, parametrization with normalizer} we are able to compare the Clifford theory of corresponding characters via equivalence relations on character triples as defined in \cite{Nav-Spa14I} and \cite{Spa17}. This idea has been introduced in \cite[Condition D]{Ros-Generalized_HC_theory_for_Dade} and provides an adaptation of generalized Harish-Chandra theory to the framework of the inductive conditions for the so-called Local--Global conjectures in representation theory of finite groups.

\begin{condA}
\label{cond:Main iEBC}
Let $\G$, $F$, $\ell$, $q$ and $e$ be as above. For every $e$-cuspidal pair $(\L,\lambda)$ of $\hspace{1pt}\G$ there exists a defect preserving $\aut_\mathbb{F}(\G^F)_{(\L,\lambda)}$-equivariant bijection
\[\Omega^\G_{(\L,\lambda)}:\E\left(\G^F,(\L,\lambda)\right)\to\irr\left(\n_\G(\L)^F\enspace\middle|\enspace \lambda\right)\]
such that
\[\left(X_\vartheta,\G^F,\vartheta\right)\iso{\G^F}\left(\n_{X_\vartheta}(\L),\n_{\G^F}(\L),\Omega^\G_{(\L,\lambda)}(\vartheta)\right)\]
in the sense of \cite[Definition 3.6]{Spa17} for every $\vartheta\in\E\left(\G^F,(\L,\lambda)\right)$ and where $X:=\G^F\rtimes \aut_\mathbb{F}(\G^F)$.
\end{condA}

The bijections described in Condition \ref{cond:Main iEBC} play a central role in the verification of the inductive conditions for the local-global conjectures. In particular, in \cite[Theorem E]{Ros-Generalized_HC_theory_for_Dade} the author shows that Condition \ref{cond:Main iEBC} implies the inductive condition for Dade's Conjecture for quasi-simple groups of Lie type in good non-defining characteristic.

In section \ref{sec:Criteria}, we prove a criterion for Condition \ref{cond:Main iEBC} (see Theorem \ref{thm:Criterion}) and show that its validation reduces to the verification of certain requirements related to the extendibility of characters of $e$-split Levi subgroups. These requirements (see Definition \ref{def:Star conditions}) are analogous to the one considered in \cite[Definition 2.2]{Cab-Spa19} and have already been studied when verifying the inductive McKay condition (see \cite{Spa12}, \cite{Cab-Spa13}, \cite{Mal-Spa16}, \cite{Cab-Spa17I}, \cite{Cab-Spa17II}, \cite{Cab-Spa19}, \cite{Spa21}) as well as the inductive condition for the Alperin--McKay and the Alperin Weight conjectures (see \cite{Spa13II}, \cite{Mal14}, \cite{Sch14}, \cite{Cab-Spa15}, \cite{Kos-Spa16I}, \cite{Kos-Spa16II}, \cite{Bro-Spa20}, \cite{Bro-Spa21}, \cite{Bro22}). Thanks to \cite[Theorem E]{Ros-Generalized_HC_theory_for_Dade}, our approach also provides a way to tackle Dade's Conjecture and its inductive condition by utilizing the theory that has already been developed to verify the inductive conditions for the other counting conjectures.



If $(\L,\lambda)$ is an $e$-cuspidal pair of $\G$, then we say that $(\L,\lambda)$ is \emph{$e$-Brauer--Lusztig-cuspidal} in $\G$ if the associated $e$-Harish-Chandra series $\E(\G^F,(\L,\lambda))$ coincides with a Brauer--Lusztig block $\E(\G^F,B,[s])$ as defined in \cite[Definition 4.13]{Ros-Generalized_HC_theory_for_Dade}.

\begin{theoA}
\label{thm:Main iEBC follows from extensions}
Suppose that $\G$ is simple, simply connected not of type $\mathbf{E}_6$, $\mathbf{E}_7$ or $\mathbf{E}_8$ and consider $\ell\in\Gamma(\G,F)$ with $\ell \geq 5$. Let $\L$ be an $e$-split Levi subgroup of $\hspace{1pt}\G$ and suppose that the following conditions hold:
\begin{enumerate}
\item maximal extendibility holds with respect to $\n_\G(\L)^F\unlhd \n_{\wt{\G}}(\L)^F$ in the sense of Definition \ref{def:Maximal extendibility};
\item there exists an $(\aut_\mathbb{F}(\G^F)_\L\ltimes \irr(\wt{\G}^F/\G^F))$-equivariant extension map for $\cusp{\wt{\L}^F}$ with respect to $\wt{\L}^F\unlhd \n_{\wt{\G}}(\L)^F$ (see the discussion following Definition \ref{def:Maximal extendibility});
\item the requirements from Definition \ref{def:Star conditions} hold for $\L\leq \G$;
\end{enumerate}
Then Condition \ref{cond:Main iEBC} holds for every $e$-Brauer--Lusztig-cuspidal pair $(\L,\lambda)$ of $\hspace{1pt}\G$ such that $\out(\G^F)_\mathcal{B}$ is abelian and where $\mathcal{B}$ is the $\wt{\G}^F$-orbit of $\hspace{1pt}\bl(\lambda)^{\G^F}$. 
\end{theoA}

Assumptions (i), (ii) and (iii) of Theorem \ref{thm:Main iEBC follows from extensions} are part of an important ongoing project in representation theory of finite groups of Lie type and have been verified for groups of type $\bf{A}$ (under certain block theoretic restrictions) and $\bf{C}$ in \cite{Bro-Spa20} and \cite{Bro22} respectively (see Remark \ref{rmk:Global star} and Lemma \ref{rmk:Local star}). Assumption (i) holds in almost all cases since $\wt{\G}^F/\G^F$ is cyclic for groups not of type $\bf{D}$ (see \cite[Proposition 1.7.5]{Gec-Mal20}). Next, we observe that every $e$-cuspidal pair associated to an $\ell$-regular semisimple element is $e$-Brauer--Lusztig-cuspidal by \cite[Theorem A]{Ros-Generalized_HC_theory_for_Dade} and \cite[Theorem 4.1]{Cab-Eng99} while it is natural to expect that every $e$-cuspidal pair is $e$-Brauer--Lusztig-cuspidal under suitable assumptions on $\ell$ (see the discussion following \cite[Theorem 4.15]{Ros-Generalized_HC_theory_for_Dade}). Finally, the block theoretic restriction of Theorem \ref{thm:Main iEBC follows from extensions} holds, for instance, whenever $\G$ is not of type $\mathbf{A}$, $\mathbf{D}$ or $\mathbf{E}_6$. In Theorem \ref{thm:iEBC follows from extensions} we prove a slightly more general result by considering a larger class of blocks. We plan to circumvent the obstructions appearing in the remaining cases by applying techniques developed by Bonnafé--Dat--Rouquier \cite{Bon-Dat-Rou17} and Ruhstorfer \cite{Ruh22} on quasi-isolated blocks.

In Section \ref{sec:Stabilizers, extendibility and consequences}, as a corollary of Theorem \ref{thm:Main iEBC follows from extensions} and by using \cite[Theorem 4.1]{Cab-Spa17I} and \cite[Theorem 1.2 and Corollary 4.7]{Bro-Spa20}, we obtain Condition \ref{cond:Main iEBC} for some cases in type $\mathbf{A}$ (see Corollary \ref{cor:Type A}). Similarly, by using \cite[Theorem 3.1]{Cab-Spa17II} and \cite[Theorem 1.1 and Theorem 1.2]{Bro22}, we obtain Condition \ref{cond:Main iEBC} for all $e$-Brauer--Lusztig-cuspidal pairs of groups of type $\mathbf{C}$ whenever $\ell\geq 5$ (see Corollary \ref{cor:Type C}).




The paper is organised as follows. In Section \ref{sec:Preliminaries} we introduce our notation and recall some preliminary results. We also introduce the notion of $e$-Brauer--Lusztig-cuspidality in Definition \ref{def:e-BL-cuspidality} and state a weaker version of Condition \ref{cond:Main iEBC} by replacing $\G^F$-block isomorphisms of character triples with $\G^F$-central isomorphisms of character triples (see Condition \ref{cond:cEBC}). In Section \ref{sec:Constructing bijections} we prove Theorem \ref{thm:Main e-Harish-Chandra parametrization for connected center} which provides an extension of \cite[Theorem 3.2]{Bro-Mal-Mic93} to non-unipotent $e$-cuspidal pairs in groups with connected centre and type different from $\mathbf{E}_6$, $\mathbf{E}_7$ or $\mathbf{E}_8$. Then, assuming maximal extendibility, we develop a Clifford theory for $e$-Harish-Chandra series with respect to regular embeddings. This is done by applying the results obtained in \cite[Section 4]{Ros-Generalized_HC_theory_for_Dade}. As a consequence, we construct certain bijections needed in the criteria proved in the subsequent section (see Theorem \ref{thm:Bijection for connected center}). In Section \ref{sec:Criteria} we prove Theorem \ref{thm:Criterion} which provides a criterion for Condition \ref{cond:Main iEBC}. On the way to prove this result we also consider the weaker Condition \ref{cond:cEBC} and prove a criterion for this condition in Theorem \ref{thm:Criterion, central}. Finally, in Section \ref{sec:Stabilizers, extendibility and consequences} we combine the results obtained in Section \ref{sec:Constructing bijections} and Section \ref{sec:Criteria} in order to obtain Theorem \ref{cor:Main bijections for nonconnected centre} and Theorem \ref{thm:Main iEBC follows from extensions}. In Definition \ref{def:Star conditions} we give a definition of certain requirements on stabilizers and extendibility of characters that should be compared with \cite[Definition 2.2]{Cab-Spa19}. Then, applying the main results of \cite{Bro-Spa20} and \cite{Bro22}, we obtain consequences for groups of type $\mathbf{A}$ and $\mathbf{C}$ and prove Corollary \ref{cor:Bijections for nonconnected centre, type A and C}, Corollary \ref{cor:Type A} and Corollary \ref{cor:Type C}. We also prove similar results for Condition \ref{cond:cEBC} (see Corollary \ref{cor:cEBC for type A} and Corollary \ref{cor:cEBC for type C}).

\section{Preliminaries}
\label{sec:Preliminaries}

In this paper, $\G$ is a connected reductive group defined over an algebraic closure $\mathbb{F}$ of a finite field of characteristic $p$ and $F:\G\to \G$ is a Frobenius endomorphism endowing $\G$ with an $\mathbb{F}_q$-structure for a power $q$ of $p$. Let $(\G^*,F^*)$ be a group in duality with $(\G,F)$ with respect to a choice of an $F$-stable maximal torus $\T$ of $\G$ and an $F^*$-stable maximal torus $\T^*$ of $\G^*$. Then, there is a bijection $\L\mapsto \L^*$ between the set of Levi subgroups of $\G$ containing $\T$ and the set of Levi subgroups of $\G^*$ containing $\T^*$ (see \cite[p.123]{Cab-Eng04}). This bijection induces a correspondence between the set of $F$-stable Levi subgroups of $\G$ and the set of $F^*$-stable Levi subgroups of $\G^*$.

\subsection{Regular embeddings}
\label{sec:Regular embeddings}

Let $\G$, $\wt{\G}$ be connected reductive groups with Frobenius endomorphisms $F:\G\to \G$ and $\wt{F}:\wt{\G}\to \wt{\G}$. A morphism of algebraic groups $i:\G\to \wt{\G}$ is a regular embedding if $\wt{F}\circ i=i\circ F$ and $i$ induces an isomorphism of $\G$ with a closed subgroup $i(\G)$ of $\wt{\G}$, the centre $\z(\wt{\G})$ of $\wt{\G}$ is connected and $[i(\G),i(\G)]=[\wt{\G},\wt{\G}]$. In this case we can identify $\G$ with its image $i(\G)$ and $\wt{F}$ with an extension of $F$ to $\wt{\G}$ which, by abuse of notation, we denote again by $F$. 

Since $[\wt{\G},\wt{\G}]$ is contained in $\G$, we deduce that $\G$ is normal in $\wt{\G}$ and that $\wt{\G}/\G$ is abelian. Moreover, as $\wt{\G}$ is connected and reductive, we have $\wt{\G}=\z(\wt{\G})[\wt{\G},\wt{\G}]=\z(\wt{\G})\G$. In particular, it follows that $\z(\G)=\z(\wt{\G})\cap \G$. Similarly, $[\wt{\G}^F,\wt{\G}^F]\leq \G^F$ and hence $\G^F$ is a normal subgroup of $\wt{\G}^F$ with abelian quotient $\wt{\G}^F/\G^F$. Notice, however, that $\wt{\G}^F$ might be larger than $\z(\wt{\G}^F)\G^F$.

Let $\L$ be an $F$-stable Levi subgroup of $\G$. Then, the group $\wt{\L}:=\z(\wt{\G})\L$ is an $F$-stable Levi subgroup of $\wt{\G}$. In fact,  if $\L=\c_\G(\S)$ with $\S:=\z^\circ(\L)$, then $\wt{\L}:=\L\z(\wt{\G})=\c_\G(\S)\z(\wt{\G})\leq \c_{\wt{\G}}(\S)=\c_{\G\z(\wt{\G})}(\S)\leq \c_{\G}(\S)\z(\wt{\G})=\L\z(\wt{\G})=\wt{\L}$. Then, it is clear that $\L=\wt{\L}\cap \G$ and therefore $\n_\G(\L)=\n_\G(\S)$ and $\n_{\wt{\G}}(\wt{\L})=\n_{\wt{\G}}(\L)=\n_{\wt{\G}}(\S)$. In addition, as $\z(\wt{\G})$ is contained in $\wt{\L}$, observe that $\wt{\G}=\wt{\L}\G$ which implies $\wt{\G}/\G\simeq \wt{\L}/\L$. Similarly, we have $\wt{\G}^F=\wt{\L}^F\G^F$ and $\wt{\G}^F/\G^F\simeq \n_{\wt{\G}}(\L)^F/\n_\G(\L)^F\simeq\wt{\L}^F/\L^F$. Observe that, since $\wt{\L}$ has connected centre by \cite[Lemma 13.14]{Dig-Mic91} and $[\wt{\L},\wt{\L}]=[\L\z(\wt{\G}),\L\z(\wt{\G})]=[\L,\L]$, the map $i\mid_{\L}:\L\to\wt{\L}$ is a regular embedding.   

Next, consider pairs $(\G^*,F^*)$ and $(\wt{\G}^*,F^*)$ dual to $(\G,F)$ and $(\wt{\G},F)$ respectively. The map $i:\G\to \wt{\G}$ induces a surjective morphism $i^*:\wt{\G}^*\to \G^*$ such that $\ker(i^*)$ is a connected subgroup of $\z(\wt{\G}^*)$ (see \cite[Section 15.1]{Cab-Eng04}). When $\G$ is simply connected, we have $\ker(i^*)=\z(\wt{\G}^*)$: observe that $\z(\G^*)$ is trivial since $\G^*$ is adjoint and therefore, using the isomorphism $\wt{\G}^*/\ker(i^*)\simeq \G^*$, we deduce that $\z(\wt{\G}^*)\leq \ker(i^*)$. As shown in \cite[(15.2)]{Cab-Eng04}, there exists an isomorphism
\begin{align}
\label{central kernel and linear characters}
\ker(i^*)^F&\to \irr\left(\wt{\G}^F/\G^F\right)
\\
z&\mapsto \wh{z}_{\wt{\G}}\nonumber
\end{align}
If $\L$ is an $F$-stable Levi subgroup of $\G$, noticing that $\ker(i^*)\leq \z(\wt{\G}^*)\leq \wt{\L}^*$, it follows that $\ker(i^*)=\ker(i^*\mid_{\wt{\L}^*})$. As before we obtain a map $\ker(i^*\mid_{\wt{\L}^*})^F\to\irr(\wt{\L}^F/\L^F)$, $z\mapsto \wh{z}_{\wt{\L}}$ which coincides with the restriction of the map defined above, i.e. $\wh{z}_{\wt{\L}}=(\wh{z}_{\wt{\G}})_{\wt{\L}^F}$. If no confusion arises, we denote $\mathcal{K}:=\ker(i^*)^F=\ker(i^*\mid_{\wt{\L}^*})^F$ and obtain bijections
\begin{align*}
\mathcal{K}&\to \irr\left(\wt{\L}^F/\L^F\right)
\\
z&\mapsto \wh{z}_{\wt{\L}}
\end{align*}
for every $F$-stable Levi subgroup $\L\leq \G$.

To conclude this section, we define an action of the group $\mathcal{K}$ on the set of irreducible characters.

\begin{defin}
\label{def:Action by linear characters}
Let $\mathcal{K}$ be as defined above. For $z\in \mathcal{K}$ and $\chi\in\irr(\wt{\G}^F)$, let
\[\chi^z:=\chi\cdot\wh{z}_{\wt{\G}},\]
where $\wh{z}_{\wt{\G}}\in\irr(\wt{\G}^F/\G^F)$ corresponds to $z$ via the isomorphism \eqref{central kernel and linear characters}. Similarly, for an $F$-stable Levi subgroup $\L$ of $\G$, the group $\mathcal{K}$ acts on $\irr(\wt{\L}^F)$. Moreover, noticing that $\wt{\G}^F/\G^F\simeq \n_{\wt{\G}}(\L)^F/\n_\G(\L)^F$, we deduce that $z\in\mathcal{K}$ also acts on the characters $\psi\in\irr(\n_{\wt{\G}}(\L)^F)$ via
\[\psi^z:=\psi\cdot\wh{z}_{\n_{\wt{\G}}(\L)},\]
where $\wh{z}_{\n_{\wt{\G}}(\L)}$ denotes the restriction of $\wh{z}_{\wt{\G}}$ to $\n_{\wt{\G}}(\L)^F$. In the same way, we can define an action of $\mathcal{K}$ on $\irr(\wt{\K}^F)$ and on $\irr(\n_{\wt{\K}}(\L)^F)$ for every $F$-stable Levi subgroups $\L$ and $\K$ of $\G$ satisfying $\L\leq \K$.
\end{defin}

\subsection{Automorphisms}
\label{sec:Automorphisms}

For every bijective morphism of algebraic groups $\sigma:\G\to\G$ satisfying $\sigma\circ F=F\circ \sigma$, the restriction of $\sigma$ to $\G^F$, which by abuse of notation we denote again by $\sigma$, is an automorphism of the finite group $\G^F$. Let $\aut_\mathbb{F}(\G^F)$ be the set of automorphisms of $\G^F$ obtained in this way. As mentioned in \cite[Section 2.4]{Cab-Spa13}, a morphism $\sigma\in\aut_\mathbb{F}(\G^F)$ is determined by its restriction to $\G^F$ up to a power of $F$. In particular $\aut_\mathbb{F}(\G^F)$ acts on the set of $F$-stable closed connected subgroups $\H$ of $\G$ and the we can define the set $\aut_\mathbb{F}(\G^F)_\H$ whose elements are the restrictions to $\G^F$ of those bijective morphisms $\sigma$ considered above that stabilize $\H$. Observe that $\aut_\mathbb{F}(\G^F)=\aut(\G^F)$ whenever $\G$ is a simple algebraic group of simply connected type such that $\G^F/\z(\G^F)$ is a non-abelian simple group (see \cite[Section 1.15]{GLS} and the remarks in \cite[Section 2.4]{Cab-Spa13}).

Next, we consider the relation between the automorphisms of $\G^F$ and those of its dual $\G^{*F^*}$. Consider a pair $(\G^*,F^*)$ dual to $(\G,F)$. According to \cite[Section 2.4]{Cab-Spa13}, there exists an isomorphism
\[\aut_\mathbb{F}\left(\G^F\right)/{\rm Inn}\left(\G_{\rm ad}^F\right)\simeq \aut_\mathbb{F}\left(\G^{*F^*}\right)/{\rm Inn}\left(\G_{\rm ad}^{*F^*}\right).\]
If the coset of $\sigma$ corresponds to the coset of $\sigma^*$ via the above isomorphism, then we write $\sigma\sim \sigma^*$ (see \cite[Definition 2.1]{Cab-Spa13}).

\begin{lem}
\label{lem:Automorphisms, duality and Levi subgroups}
Let $\L\leq \K$ be $F$-stable Levi subgroups of $\G$ in duality with the Levi subgroups $\L^*\leq \K^*$ of $\G^*$. Then, for every $\sigma\in\aut_\mathbb{F}(\G^F)_{\L,\K}$ there exists $\sigma^*\in\aut_\mathbb{F}(\G^{*F^*})_{\L^*,\K^*}$ such that $\sigma\sim \sigma^*$.
\end{lem}

\begin{proof}
Notice that the groups $\aut_\mathbb{F}(\G^F)_{\L,\K}:=\aut_\mathbb{F}(\G^F)_\L\cap\aut_\mathbb{F}(\G^F)_\K$ and $\aut_\mathbb{F}(\G^{*F^*})_{\L^*,\K^*}:=\aut_\mathbb{F}(\G^{*F^*})_{\L^*}\cap\aut_\mathbb{F}(\G^{*F^*})_{\K^*}$ are well defined. If $\L=\K$ the result follows from \cite[Proposition 2.2]{Cab-Spa13} while a similar argument applies in the general case.
\end{proof}

Assume now that $\G$ is simple of simply connected type. Fix a maximally split torus $\T_0$ contained in an $F$-stable Borel subgroup $\B_0$ of $\G$. This choice corresponds to a set of simple roots $\Delta\subseteq \Phi:=\Phi(\G,\T_0)$. For every $\alpha\in\Phi$ consider a one-parameter subgroup $x_\alpha:\mathbb{G}_{\rm a}\to \G$. Then $\G$ is generated by the elements $x_\alpha(t)$, where $t\in\mathbb{G}_{\rm a}$ and $\alpha\in\pm\Delta$. Consider the \emph{field endomorphism} $F_0:\G\to \G$ given by $F_0(x_\alpha(t)):=x_{\alpha}(t^p)$ for every $t\in\mathbb{G}_{\rm a}$ and $\alpha\in\Phi$. Moreover, for every symmetry $\gamma$ of the Dynkin diagram of $\Delta$, we have a \emph{graph automorphism} $\gamma:\G\to \G$ given by $\gamma(x_\alpha(t)):=x_{\gamma(\alpha)}(t)$ for every $t\in\mathbb{G}_{\rm a}$ and $\alpha\in\pm\Delta$. Then, up to inner automorphisms of $\G$, any Frobenius endomorphism $F$ defining an $\mathbb{F}_q$-structure on $\G$ can be written as $F=F_0^m\gamma$, for some symmetry $\gamma$ and $m\in\mathbb{Z}$ with $q=p^m$ (see \cite[Theorem 22.5]{Mal-Tes}). One can construct a regular embedding $\G\leq\wt{\G}$ in such a way that the Frobenius endomorphism $F_0$ extends to an algebraic group endomorphism $F_0:\wt{\G}\to\wt{\G}$ defining an $\mathbb{F}_p$-structure on $\wt{\G}$. Moreover, every graph automorphism $\gamma$ can be extended to an algebraic group automorphism of $\wt{\G}$ commuting with $F_0$ (see \cite[Section 2B]{Mal-Spa16}). If we denote by $\mathcal{A}$ the group generated by $\gamma$ and $F_0$, then we can construct the semidirect product $\wt{\G}^F\rtimes \mathcal{A}$. Finally, we define the set of \emph{diagonal automorphisms} of $\G^F$ to be the set of those automorphisms induced by the action of $\wt{\G}^F$ on $\G^F$. If $\G^F/\z(\G^F)$ is a non-abelian simple group, then the group $\wt{\G}^F\rtimes \mathcal{A}$ acts on $\G^F$ and induces all the automorphisms of $\G^F$ (see, for instance, the proof of \cite[Proposition 3.4]{Spa12} and of \cite[Theorem 2.4]{Cab-Spa19}).

\subsection{Restrictions on primes}
\label{subsection:Good primes and Levi}

For the rest of this section we consider the following setting.

\begin{notation}
\label{notation}
Let $\G$ be a connected reductive linear algebraic group defined over an algebraic closure of a finite field of characteristic $p$ and $F:\G\to \G$ a Frobenius endomorphism defining an $\mathbb{F}_q$-structure on $\G$, for a power $q$ of $p$. Consider a prime $\ell$ different from $p$ and denote by $e$ the multiplicative order of $q$ modulo $\ell$ (modulo $4$ if $\ell=2$). All blocks are considered with respect to the prime $\ell$.
\end{notation}

Here we recall the definition of good primes and define the set $\Gamma(\G,F)$ (see also \cite[Notation 1.1]{Cab-Eng94}). First, recall that $\ell$ is a \emph{good prime} for $\G$ if it is good for each simple factor of $\G$, while the conditions for the simple factors are
\begin{align*}
{\rm \bf A}_n&: \text{every prime is good}
\\
{\rm \bf B}_n, {\rm \bf C}_n, {\rm \bf D}_n&: \ell\neq 2
\\
{\rm \bf G}_2, {\rm \bf F}_4, {\rm \bf E}_6, {\rm \bf E}_7&: \ell\neq 2,3
\\
{\rm \bf E}_8&: \ell\neq 2,3,5.
\end{align*}
We say that $\ell$ is a \emph{bad prime} for $\G$ if it is not a good prime. Then, we denote by $\gamma(\G,F)$ the set of primes $\ell$ such that: $\ell$ is odd, $\ell\neq p$, $\ell$ is good for $\G$ and $\ell$ doesn't divide $|\z(\G)^F:\z^\circ(\G)^F|$. Let $(\G^*,F^*)$ be in duality with $(\G,F)$ and set $\Gamma(\G,F):=(\gamma(\G,F)\cap \gamma(\G^*,F^*))\setminus\{3\}$ if $\G_{\rm ad}^F$ has a component of type $^3\mathbf{D}_4(q^m)$ and $\Gamma(\G,F):=\gamma(\G,F)\cap \gamma(\G^*,F^*)$ otherwise.

\subsection{$e$-Harish-Chandra theory}
\label{subsec:e-HC theory}

Let $\G$, $F$, $q$, $\ell$ and $e$ be as in Notation \ref{notation} and consider an $F$-stable Levi complement of a (not necessarily $F$-stable) parabolic subgroup $\P$ of $\G$. Deligne--Lusztig \cite{Del-Lus76} and Lusztig \cite{Lus76} defined two $\mathbb{Z}$-linear maps
\[\R_{\L\leq \P}^\G:\mathbb{Z}\irr\left(\L^F\right)\to\mathbb{Z}\irr\left(\G^F\right) \hspace{15pt}\text{and}\hspace{15pt} {^*\R}_{\L\leq \P}^\G:\mathbb{Z}\irr\left(\G^F\right)\to\mathbb{Z}\irr\left(\L^F\right)\]
called Deligne--Lusztig induction and restriction respectively. It is conjectured that $\R_{\L\leq \P}^\G$ and $^*\R_{\L\leq \P}^\G$ do not depend on the choice of $\P$. This would, for instance, follow by the Mackey formula which has been proved whenever $\G^F$ does not have components of type $^2\mathbf{E}_6(2)$, $\mathbf{E}_7(2)$ or $\mathbf{E}_8(2)$ (see \cite{Bon-Mic10}). In what follows, we just write $\R_\L^\G$ and $^*\R_\L^\G$ whenever the independence on the choice of $\P$ is known.

An $F$-stable torus $\T$ of $\G$ is called a $\Phi_e$-torus if its order polynomial is of the form $P_{(\T,F)}=\Phi_e^{n}$ for some non-negative integer $n$ and where $\Phi_e$ denotes the $e$-th cyclotomic polynomial (see \cite[Definition 13.3]{Cab-Eng04}). The centralisers of $\Phi_e$-tori are called $e$-split Levi subgroups. Then $(\L,\lambda)$ is an $e$-cuspidal pair of $(\G,F)$ (or simply of $\G$ when no confusion arises) if $\L$ is an $e$-split Levi subgroup of $\G$ and $\lambda\in\irr(\L^F)$ satisfies $^*\R_{\M\leq \Q}^\L(\lambda)=0$ for every $e$-split Levi subgroup $\M<\L$ and every parabolic subgroup $\Q$ of $\L$ containing $\M$ as Levi complement. An $e$-cuspidal pair $(\L,\lambda)$ is $(e,\ell')$-cuspidal if $\lambda$ lies in a Lusztig series associated with an $\ell$-regular semisimple element of the dual group $\L^{*F^*}$. To any $e$-cuspidal pair $(\L,\lambda)$ of $\G$ we associate the $e$-Harish-Chandra series $\E(\G^F,(\L,\lambda))$ consisting of the irreducible constituents of the virtual characters $\R_{\L\leq \P}^\G(\lambda)$ for every parabolic subgroup $\P$ of $\G$ containing $\L$ as a Levi subgroup.

Using Deligne--Lusztig induction, one can define a partial order relation on the set of $e$-pairs of $\G$. If $(\L,\lambda)$ and $(\M,\mu)$ are $e$-pairs of $\G$, then we write $(\L,\lambda)\leq_e(\M,\mu)$ if $\L\leq \M$ and $\mu$ is an irreducible constituent of $\R_{\L\leq \Q}^\M$ for some parabolic subgroup $\Q$ of $\M$ having $\L$ as a Levi complement. Then, $\ll_e$ denotes the transitive closure of $\leq_e$. It is conjectured that $\leq_e$ is transitive and hence coincides with $\ll_e$ (see \cite[Notation 1.11]{Cab-Eng99} and \cite[Proposition 3.6 and Proposition 4.10]{Ros-Generalized_HC_theory_for_Dade}).

Recall that a Brauer--Lusztig block is a non-empty set of characters of $\G^F$ of the form $\E(\G^F,B,[s]):=\E(\G^F,[s])\cap \irr(B)$, where $\E(\G^F,[s])$ denotes the rational Lustig series associated to the semisimple element $s\in\G^{*F^*}$ and $B$ is an $\ell$-block of $\G^F$. In \cite[Theorem A]{Ros-Generalized_HC_theory_for_Dade} the author gives a description of the Brauer--Lusztig blocks in terms of $e$-Harish-Chandra series under suitable assumptions. More precisely, we assume the following hypothesis.

\begin{hyp}
\label{hyp:Brauer--Lusztig blocks}
Let $\G$, $F:\G\to \G$, $q$, $\ell$ and $e$ be as in Notation \ref{notation}. Assume that:
\begin{enumerate}
\item $\ell\in\Gamma(\G,F)$ with $\ell\geq 5$ and the Mackey formula hold for $(\G,F)$;
\item either $\z((\G^*)_{\rm sc}^{F^*})_\ell=1$ or $\ell\in\Gamma((\G^*)_{\rm ad},F)$; and
\item for every $e$-split Levi subgroup $\K$ of $\G$, we have
\[\left\lbrace\kappa\in\irr\left(\K^F\right)\enspace\middle|\enspace(\L,\lambda)\ll_e(\K,\kappa)\right\rbrace=\irr\left(\R_\L^\K(\lambda)\right)\]
for every $(e,\ell')$-cuspidal pair $(\L,\lambda)$ of $\K$.
\end{enumerate}
\end{hyp}

Under Hypothesis \ref{hyp:Brauer--Lusztig blocks}, \cite[Theorem A]{Ros-Generalized_HC_theory_for_Dade} shows that for every $e$-cuspidal pair $(\L,\lambda)$ of $\G$ we have
\[\E\left(\G^F,(\L,\lambda)\right)\subseteq \E\left(\G^F,B,[s]\right)\]
where $s$ is a semisimple element of $\L^{*F^*}$ such that $\lambda\in\E(\L^F,[s])$, $B=\bl(\lambda)^{\G^F}$ and $\E(\G^F,B,[s])$ is the associated Brauer--Lusztig block (see \cite[Definition 4.13]{Ros-Generalized_HC_theory_for_Dade}). Inspired by this result, we introduce the following definition. 

\begin{defin}
\label{def:e-BL-cuspidality}
An \emph{$e$-Brauer--Lusztig-cuspidal pair} of $\G$ is an $e$-cuspidal pair $(\L,\lambda)$ of $\G$ such that
\[\E\left(\G^F,(\L,\lambda)\right)=\E\left(\G^F,B,[s]\right)\]
for some semisimple element $s$ of $\G^{*F^*}$ and some $\ell$-block $B$ of $\G^F$.
\end{defin}

By \cite[Theorem A]{Ros-Generalized_HC_theory_for_Dade} and \cite[Theorem 4.1]{Cab-Eng99} it follows that every $e$-cuspidal pair $(\L,\lambda)$ such that $\lambda$ lies in a Lusztig series associated to an $\ell$-regular semisimple element is $e$-Brauer--Lusztig-cuspidal. It is natural to expect that all $e$-cuspidal pairs are $e$-Brauer--Lusztig-cuspidal under Hypothesis \ref{hyp:Brauer--Lusztig blocks}. This would follow by certain block theoretic properties of the Jordan decomposition of characters. We refer the reader to the discussion following \cite[Theorem 4.15]{Ros-Generalized_HC_theory_for_Dade}.

We conclude this section with a remark on the validity of Hypothesis \ref{hyp:Brauer--Lusztig blocks}. Observe that the Mackey formula and Hypothesis \ref{hyp:Brauer--Lusztig blocks} (iii) are expected to hold for any connected reductive group.

\begin{rmk}
\label{rmk:Brauer-Lusztig for simple simply connected}
Let $\G$ be simple of simply connected type such that $\G^F\neq {^2 \mathbf{E}}_6(2)$, $\mathbf{E}_7(2)$, $\mathbf{E}_8(2)$ and consider $\ell\in\Gamma(\G,F)$ with $\ell\geq 5$. Under these restrictions, Hypothesis \ref{hyp:Brauer--Lusztig blocks} is satisfied. In fact, in this case, the Mackey formula holds by \cite{Bon-Mic10} while Hypothesis \ref{hyp:Brauer--Lusztig blocks} (iii) holds by \cite[Corollary 3.7]{Ros-Generalized_HC_theory_for_Dade}. Moreover, since $\G$ is simple and simply connected, our assumption on $\ell$ shows that $\ell\in\Gamma((\G^*)_{\rm ad},F)$ (see \cite[Table 13.11]{Cab-Eng04}).
\end{rmk}

\subsection{A non-blockwise version of Condition \ref{cond:Main iEBC}}

By replacing $\G^F$-block isomorphisms of character triples with $\G^F$-central isomorphisms of character triples (see \cite[Definition 3.3.4]{Ros-Thesis}) a non-blockwise version of Condition \ref{cond:Main iEBC} can be introduced (see \cite[Condition 7.13]{Ros-Generalized_HC_theory_for_Dade}).

\begin{cond}
\label{cond:cEBC}
Let $\G$, $F:\G\to \G$, $q$, $\ell$ and $e$ be as in Notation \ref{notation} and consider an $e$-cuspidal pair $(\L,\lambda)$ of $\G$. Then there exists a defect preserving $\aut_\mathbb{F}(\G^F)_{(\L,\lambda)}$-equivariant bijection
\[\Omega^\G_{(\L,\lambda)}:\E\left(\G^F,(\L,\lambda)\right)\to\irr\left(\n_\G(\L)^F\enspace\middle|\enspace \lambda\right)\]
such that
\[\left(X_\vartheta,\G^F,\vartheta\right)\isoc{\G^F}\left(\n_{X_\vartheta}(\L),\n_{\G^F}(\L),\Omega^\G_{(\L,\lambda)}(\vartheta)\right)\]
for every $\vartheta\in\E\left(\G^F,(\L,\lambda)\right)$ and where $X:=\G^F\rtimes \aut_\mathbb{F}(\G^F)$.
\end{cond}

We point out that it is much easier to verify Condition \ref{cond:cEBC} than it is to verify Condition \ref{cond:Main iEBC}. As a hint to this fact, the reader should compare Theorem \ref{thm:cEBC follows from extensions} and Theorem \ref{thm:iEBC follows from extensions}.

Moreover, \cite[Theorem 7.14]{Ros-Generalized_HC_theory_for_Dade} shows that Condition \ref{cond:cEBC} implies a non-blockwise version of the Character Triple Conjecture (see \cite[Definition 3.5.5]{Ros-Thesis}). In analogy with \cite[Theorem 1.3]{Spa17}, we expect that the non-blockwise version of the Character Triple Conjecture could be used as an inductive condition to obtain a reduction theorem for the non-blockwise version of Dade's Conjecture. We refer the reader to \cite[Section 7.1]{Ros-Generalized_HC_theory_for_Dade} for a more detailed discussion on this topic. On the way to prove our results for Condition \ref{cond:Main iEBC}, we also obtain similar statements for the simpler Condition \ref{cond:cEBC}.

\section{Parametrization of non-unipotent $e$-Harish-Chandra series}
\label{sec:Constructing bijections}

In this section we start by proving Theorem \ref{thm:Main e-Harish-Chandra parametrization for connected center} and hence extend \cite[Theorem 3.2]{Bro-Mal-Mic93} to non-unipotent $e$-cuspidal pairs of reductive groups with connected centre and type different from $\mathbf{E}_6$, $\mathbf{E}_7$ or $\mathbf{E}_8$. Then, assuming maximal extendibility for $e$-cuspidal characters of $e$-split Levi subgroups, we prove Theorem \ref{thm:Bijection for connected center} and obtain certain bijections that are part of the requirements of the criteria we prove in Section \ref{sec:Criteria} (see Assumption \ref{ass:Assumption for the criterion, central} (ii) and Assumption \ref{ass:Assumption for the criterion} (ii)).

\subsection{$e$-Harish-Chandra theory for groups with connected centre}
\label{sec:e-Harisch_chandra for connected center}

In what follows we make use of the fact that, under suitable assumptions, there exists a Jordan decomposition that commutes with Deligne--Lusztig induction (see \cite[Theorem 4.7.2 and Theorem 4.7.5]{Gec-Mal20}). More precisely, we consider the following hypothesis.

\begin{hyp}
\label{hyp:e-Harish-Chandra theory, no regular embedding}
Let $\G$, $F:\G\to \G$, $q$, $\ell$ and $e$ be as in Notation \ref{notation} and suppose that $\G$ has connected centre and that $[\G,\G]$ is simple not of type $\mathbf{E}_6$, $\mathbf{E}_7$ or $\mathbf{E}_8$.
\end{hyp}

\begin{theo}
\label{thm:Jordan decomposition for connected center case}
Assume Hypothesis \ref{hyp:e-Harish-Chandra theory, no regular embedding}. Then there exists a collection of bijections
\[J_{\L,s}:\E\left(\L^F,[s]\right)\to \E\left(\c_{\L^*}(s)^{F^*},[1]\right)\]
for every $F$-stable Levi subgroup $\L$ of $\G$ and every semisimple element $s\in \L^{*F^*}$, such that the following properties are satisfied:
\begin{enumerate}
\item $J_{\L,s}\left(\lambda\right)^{\sigma^*}=J_{\L,\sigma^*(s)}\left(\lambda^\sigma\right)$ for every $\lambda\in\E(\L^F,[s])$, every $\sigma\in\aut_\mathbb{F}(\G^F)_{\L}$ and $\sigma^*\in\aut_\mathbb{F}(\G^{*F^*})_{\L^*}$ with $\sigma\sim\sigma^*$ (see \cite[Proposition 2.2]{Cab-Spa13});
\item $J_{\K,s}\circ \R_{\L}^{\K}=\R^{\c_{\K^*}(s)}_{\c_{\L^*}(s)}\circ J_{\L,s}$ for every $F$-stable Levi subgroup $\K$ of $\G$ containing $\L$;
\item $\lambda(1)=\left|\L^{*F^*}:\c_{\L^*}(s)^{F^*}\right|_{p'}\cdot J_{\L,s}(\lambda)(1)$ for every $\lambda\in\E(\L^F,[s])$; and
\item if $z\in\z(\L^{*F^*})$ corresponds to the character $\wh{z}_{\L}\in\irr(\L^F)$ via \cite[(8.19)]{Cab-Eng04}, then
\[J_{\L,s}\left(\lambda\right)=J_{\L,sz}\left(\lambda\cdot\wh{z}_{\L}\right)\]
for every $\lambda\in\E(\L^F,[s])$, or equivalently
\[J^{-1}_{\L,s}\left(\nu\right)\cdot\wh{z}_{\L}=J^{-1}_{\L,sz}\left(\nu\right)\]
for every $\nu\in\E(\c_{\L^*}(s)^{F^*},[1])=\E(\c_{\L^*}(sz)^{F^*},[1])$
\end{enumerate}
\end{theo}

\begin{proof}
The required bijections are constructed in \cite[Theorem 7.1]{Dig-Mic90} and satisfy (iv) by \cite[Theorem 7.1 (iii)]{Dig-Mic90}. The properties (i) and (ii) follows from \cite[Theorem 3.1]{Cab-Spa13} and \cite[Theorem 4.7.2 and Theorem 4.7.5]{Gec-Mal20} respectively. For (iii) see, for instance, the description given in \cite[(2.1)]{Mal07}.
\end{proof}

As a consequence of the equivariance of the above Jordan decomposition, we obtain an isomorphism of relative Weyl groups. This result should be compared with \cite[Corollary 3.3]{Cab-Spa13}

\begin{cor}
\label{cor:Frattini isomorphism for connected center}
Assume Hypothesis \ref{hyp:e-Harish-Chandra theory, no regular embedding}, let $\L\leq\K$ be $F$-stable Levi subgroups of $\G$. Then, there exists a collection of isomorphisms
\[i^{\K}_{\L,\lambda}:W_{\K}\left(\L,\lambda\right)^F\to W_{\c_{\K^*}(s)}\left(\c_{\L^*}(s),J_{\L,s}\left(\lambda\right)\right)^{F^*}\]
for every $\lambda\in\E(\L^F,[s])$, such that
\[\sigma^*\circ i^{\K}_{\L,\lambda}=i^{\K}_{\L,\lambda^\sigma}\circ \sigma\]
for every $\sigma\in\aut_\mathbb{F}(\G^F)_{\K,\L}$ and $\sigma^*\in\aut_\mathbb{F}(\G^{*F^*})_{\K^*,\L^*}$ with $\sigma\sim\sigma^*$ (see Lemma \ref{lem:Automorphisms, duality and Levi subgroups}). Moreover, if $z\in \z(\K^{*F^*})$ corresponds to the character $\wh{z}_{\L}\in\irr(\L^F)$ via \cite[(8.19)]{Cab-Eng04}, then
\[W_{\K}\left(\L,\lambda\right)^F=W_{\K}\left(\L,\lambda\cdot\wh{z}_{\L}\right)^F,\]
\[W_{\c_{\K^*}(s)}\left(\c_{\L^*}(s),J_{\L,s}\left(\lambda\right)\right)^{F^*}=W_{\c_{\K^*}(sz)}\left(\c_{\L^*}(sz),J_{\L,sz}\left(\lambda\cdot\wh{z}_{\L}\right)\right)^{F^*}\]
and
\[i^{\K}_{\L,\lambda}=i^{\K}_{\L,\lambda\cdot\wh{z}_{\L}}\]
\end{cor}

\begin{proof}
The first statement follows from the proof of \cite[Corollary 3.3]{Cab-Spa13}. The second statement follows from Theorem \ref{thm:Jordan decomposition for connected center case} (iv).
\end{proof}

Before proving Theorem \ref{thm:Main e-Harish-Chandra parametrization for connected center}, we state an equivariant version of \cite[Theorem 3.2]{Bro-Mal-Mic93}. The following statement is a slight improvement of \cite[Theorem 3.4]{Cab-Spa13}.

\begin{theo}
\label{thm:Broue-Malle-Michel for connected center}
Let $\H$ be a connected reductive group with a Frobenius endomorphism $F:\H\to\H$ defining an $\mathbb{F}_q$-structure on $\H$, $\ell$ a prime not dividing $q$ and $e$ the order of $q$ modulo $\ell$ (modulo $4$ if $\ell=2$). For any $e$-split Levi subgroup $\M$ of $\hspace{2pt}\H$ and $\mu\in\E(\M^F,[1])$ with $(\M,\mu)$ a unipotent $e$-cuspidal pair, there exists an $\aut_\mathbb{F}(\H^F)_{(\M,\mu)}$-equivariant bijection
\[I^{\H}_{(\M,\mu)}:\irr\left(W_\H(\M,\mu)^F\right)\to\E\left(\H^F,(\M,\mu)\right)\]
such that
\[I^{\H}_{(\M,\mu)}(\eta)(1)_\ell=\left|\H^F:\n_\H(\M,\mu)^F\right|_\ell\cdot\mu(1)_\ell\cdot \eta(1)_\ell\]
for every $\eta\in\irr(W_\H(\M,\mu)^F)$.
\end{theo}

\begin{proof}
This follows from the proof of \cite[Theorem 3.4]{Cab-Spa13} applied to arbitrary $e$-split Levi subgroups (see the comment in the proof of \cite[Proposition 5.5]{Bro-Spa20}). Regarding the statement on character degrees, see \cite[Theorem 4.2]{Mal07} and the argument used to prove \cite[Lemma 5.3]{Bro-Spa20}.
\end{proof}

Let $\K$ be an $F$-stable Levi subgroup of $\G$ and consider an $e$-cuspidal pair $(\L,\lambda)$ of $\K$. Let $s$ be a semisimple element of $\L^{*F^*}$ such that $\lambda\in\E(\L^F,[s])$. By \cite[Proposition 1.10]{Cab-Eng99}, the unipotent character $J_{\L,s}(\lambda)$ is $e$-cuspidal. Moreover, using the fact that $\L$ is an $e$-split Levi subgroup of $\K$, we conclude that $\c_{\L^*}(s)$ is an $e$-split Levi subgroup of $\c_{\K^*}(s)$. This shows that $(\c_{\L^*}(s),J_{\L,s}(\lambda))$ is a unipotent $e$-cuspidal pair of $\c_{\K^*}(s)$. Now, we can define the map  
\begin{equation}
\label{eq:e-Harisch-Chandra parametrization for connected center}
I^{\K}_{(\L,\lambda)}:\irr\left(W_{\K}\left(\L,\lambda\right)^F\right)\to\E\left(\K^F,\left(\L,\lambda\right)\right)
\end{equation}
given by
\[I_{(\L,\lambda)}^{\K}\left(\eta\right):=J_{\K,s}^{-1}\left(I_{(\c_{\L^*}(s),J_{\L,s}\left(\lambda\right))}^{\c_{\K^*}(s)}\left(\left(\eta\right)^{i^{\K}_{\L,\lambda}}\right)\right)\]
for every $\eta\in\irr(W_{\K}(\L,\lambda)^F)$ and where $\eta^{i^{\K}_{\L,\lambda}}\in\irr(W_{\c_{\K^*}(s)}(\c_{\L^*}(s),J_{\L,s}(\lambda))^{F^*})$ corresponds to $\eta$ via the isomorphism $i^{\K}_{\L,\lambda}$ of Corollary \ref{cor:Frattini isomorphism for connected center} and $I_{(\c_{\L^*}(s),J_{\L,s}\left(\lambda\right))}^{\c_{\K^*}(s)}$ is the map constructed in Theorem \ref{thm:Broue-Malle-Michel for connected center}.

\begin{lem}
\label{lem:e-Harisch-Chandra parametrization for connected center, equivariant}
Assume Hypothesis \ref{hyp:e-Harish-Chandra theory, no regular embedding}. Then the map $I^{\K}_{(\L,\lambda)}$ is an $\aut_\mathbb{F}(\G^F)_{\K,(\L,\lambda)}$-equivariant bijection.
\end{lem}

\begin{proof}
First, we observe that the map $I^{\K}_{(\L,\lambda)}$ is a bijection because of Theorem \ref{thm:Jordan decomposition for connected center case} (ii), in fact
\begin{align*}
I^{\K}_{(\L,\lambda)}\left(\irr\left(W_{\K}\left(\L,\lambda\right)^F\right)\right)&=J_{\K,s}^{-1}\left(\irr\left(\R^{\c_{\K^*}(s)}_{\c_{\L^*}(s)}\left(J_{\L,s}\left(\lambda\right)\right)\right)\right)
\\
&=\irr\left(J_{\K,s}^{-1}\circ \R^{\c_{\K^*}(s)}_{\c_{\L^*}(s)}\circ J_{\L,s}\left(\lambda\right)\right)
\\
&=\irr\left(\R^{\K}_{\L}\left(\lambda\right)\right).
\end{align*}
Next, to show that the bijection is equivariant, let $\sigma\in\aut_\mathbb{F}(\G^F)_{\K,\L}$ and consider $\sigma^*\in\aut_\mathbb{F}(\G^{*F^*})_{\K^*,\L^*}$ with $\sigma\sim\sigma^*$ (see Lemma \ref{lem:Automorphisms, duality and Levi subgroups}). If $\sigma\in\aut_\mathbb{F}(\G^F)_{\K,(\L,\lambda)}$, then $\sigma^*$ stabilizes the $\L^{*F^*}$-orbit of $s$. Without loss of generality, we may assume that $\sigma^*(s)=s$. Then Theorem \ref{thm:Jordan decomposition for connected center case} (i) implies that $\sigma^*$ stabilizes $J_{\L,s}(\lambda)$. Applying Theorem \ref{thm:Jordan decomposition for connected center case} (i) and the equivariance properties of Corollary \ref{cor:Frattini isomorphism for connected center} and Theorem \ref{thm:Broue-Malle-Michel for connected center}, we conclude that
\begin{align*}
I^{\K}_{(\L,\lambda)}\left(\eta\right)^\sigma &=J^{-1}_{\K,s}\left(I^{\c_{\K^*}(s)}_{(\c_{\L^*}(s),J_{\L,s}\left(\lambda\right))}\left(\eta\hspace{2pt}^{i^{\K}_{\L,\lambda}}\right)\right)^\sigma
\\
&=J^{-1}_{\K,s}\left(I^{\c_{\K^*}(s)}_{(\c_{\L^*}(s),J_{\L,s}\left(\lambda\right))}\left(\left(\eta\hspace{2pt}^{i^{\K}_{\L,\lambda}}\right)^{\sigma^*}\right)\right)
\\
&=J^{-1}_{\K,s}\left(I^{\c_{\K^*}(s)}_{(\c_{\L^*}(s),J_{\L,s}\left(\lambda\right))}\left(\left(\eta\hspace{2pt}^\sigma\right)^{i^{\K}_{\L,\lambda}}\right)\right)
\\
&=I^{\K}_{(\L,\lambda)}\left(\eta\hspace{2pt}^\sigma\right)
\end{align*}
for every $\eta\in\irr(W_{\K}(\L,\lambda)^F)$.
\end{proof}

\begin{lem}
\label{lem:e-Harisch-Chandra parametrization for connected center, degree}
Assume Hypothesis \ref{hyp:e-Harish-Chandra theory, no regular embedding}. Then $I^{\K}_{(\L,\lambda)}(\eta)(1)_\ell=\left|\K^F:\n_{\K}(\L,\lambda)^F\right|_\ell\cdot\lambda(1)_\ell\cdot \eta(1)_\ell$ for every $\eta\in\irr(W_{\K}(\L,\lambda)^F)$.
\end{lem}

\begin{proof}
By the condition on character degrees given in Theorem \ref{thm:Broue-Malle-Michel for connected center} together with Theorem \ref{thm:Jordan decomposition for connected center case} (iii), we deduce that
\begin{align*}
\eta(1)_\ell &=\left(\eta\right)^{i^{\K}_{\L,\lambda}}(1)_\ell=\dfrac{I^{\c_{\K^*}(s)}_{(\c_{\L^*}(s),J_{\L,s}\left(\lambda\right))}\left(\left(\eta\right)^{i^{\K}_{\L,\lambda}}\right)(1)_\ell}{J_{\L,s}\left(\lambda\right)(1)_\ell\cdot \left|\c_{\K^*}(s)^{F^*}:\n_{\c_{\K^*}(s)}(\c_{\L^*}(s),J_{\L,s}\left(\lambda\right))^{F^*}\right|_\ell}
\\
&=\dfrac{I^{\K}_{(\L,\lambda)}\left(\eta\right)(1)_\ell\cdot \left|\c_{\K^*}(s)^{F^*}\right|_\ell\cdot \left|\L^F\right|_\ell}{\lambda(1)_\ell\cdot |\c_{\L^*}(s)^{F^*}|_\ell\cdot |\K^F|_\ell\cdot \left|\c_{\K^*}(s)^{F^*}:\n_{\c_{\K^*}(s)}(\c_{\L^*}(s),J_{\L,s}\left(\lambda\right))^{F^*}\right|_\ell}
\\
&=\dfrac{I^{\K}_{(\L,\lambda)}\left(\eta\right)(1)_\ell}{\lambda(1)_\ell\cdot \left|\K^F:\n_{\K}(\L,\lambda)^F\right|_\ell}
\end{align*}
for every $\eta\in\irr(W_{\K}(\L,\lambda)^F)$. The results follows immediately from the above equality.
\end{proof}

\begin{lem}
\label{lem:e-Harisch-Chandra parametrization for connected center, central}
Assume Hypothesis \ref{hyp:e-Harish-Chandra theory, no regular embedding}. If $z\in\z(\K^{*F^*})$ corresponds to the characters $\wh{z}_{\L}\in\irr(\L^F)$ and $\wh{z}_{\K}\in\irr(\K^F)$ via \cite[(8.19)]{Cab-Eng04}, then $\lambda\cdot\wh{z}_{\L}$ is $e$-cuspidal, $W_{\K}(\L,\lambda)^F=W_{\K}(\L,\lambda\cdot\wh{z}_{\L})^F$ and
\[I^{\K}_{(\L,\lambda)}\left(\eta\right)\cdot \wh{z}_{\K}=I^{\K}_{(\L,\lambda\cdot\wh{z}_{\L})}\left(\eta\right)\]
for every $\eta\in\irr(W_{\K}(\L,\lambda)^F)$.
\end{lem}

\begin{proof}
According to \cite[Proposition 12.1]{Bon06} the character $\lambda\cdot\wh{z}_{\L}$ is $e$-cuspidal, while Corollary \ref{cor:Frattini isomorphism for connected center} shows that $W_{\K}(\L,\lambda)^F=W_{\K}(\L,\lambda\cdot\wh{z}_{\L})^F$ and that $i^{\K}_{\L,\lambda}=i^{\K}_{\L,\lambda\cdot \wh{z}_{\L}}$. Using Theorem \ref{thm:Jordan decomposition for connected center case} (iv) we obtain
\[I_{(\c_{\L^*}(s),J_{\L,s}\left(\lambda\right))}^{\c_{\K^*}(s)}=I_{(\c_{\L^*}(sz),J_{\L,sz}\left(\lambda\cdot \wh{z}_{\L}\right))}^{\c_{\K^*}(sz)}\]
and
\begin{align*}
I^{\K}_{(\L,\lambda)}\left(\eta\right)\cdot \wh{z}_{\K}&=J_{\K,s}^{-1}\left(I_{(\c_{\L^*}(s),J_{\L,s}\left(\lambda\right))}^{\c_{\K^*}(s)}\left(\left(\eta\right)^{i^{\K}_{\L,\lambda}}\right)\right)\cdot \wh{z}_{\K}
\\
&=J_{\K,s}^{-1}\left(I_{(\c_{\L^*}(sz),J_{\L,sz}\left(\lambda\cdot \wh{z}_{\L}\right))}^{\c_{\K^*}(sz)}\left(\left(\eta\right)^{i^{\K}_{\L,\lambda\cdot \wh{z}_{\L}}}\right)\right)\cdot \wh{z}_{\K}
\\
&=J_{\K,sz}^{-1}\left(I_{(\c_{\L^*}(sz),J_{\L,sz}\left(\lambda\cdot \wh{z}_{\L}\right))}^{\c_{\K^*}(sz)}\left(\left(\eta\right)^{i^{\K}_{\L,\lambda\cdot \wh{z}_{\L}}}\right)\right)
\\
&=I^{\K}_{(\L,\lambda\cdot\wh{z}_{\L})}\left(\eta\right)
\end{align*}
for every $\eta\in\irr(W_{\K}(\L,\lambda)^F)$.
\end{proof}

Now, combining Lemma \ref{lem:e-Harisch-Chandra parametrization for connected center, equivariant}, Lemma \ref{lem:e-Harisch-Chandra parametrization for connected center, degree} and Lemma \ref{lem:e-Harisch-Chandra parametrization for connected center, central}, we obtain Theorem \ref{thm:Main e-Harish-Chandra parametrization for connected center}.

To conclude, we prove one final result which, although not used directly in the subsequent sections, might be of independent interest. Under the Hyposthesis \ref{hyp:e-Harish-Chandra theory, no regular embedding}, the bijections $I_{(\L,\lambda)}^\K$ from \eqref{eq:e-Harisch-Chandra parametrization for connected center} extend by linearity to $\mathbb{Z}$-linear bijections
\begin{equation}
\label{eq:Extending linearly}
I^\K_{(\L,\lambda)}:\mathbb{Z}\irr\left(W_\K(\L,\lambda)^F\right)\to\mathbb{Z}\E\left(\K^F,(\L,\lambda)\right).
\end{equation}
If we consider the definition given in \eqref{eq:e-Harisch-Chandra parametrization for connected center} and replace the maps $I_{(\c_{\L^*}(s),J_{\L,s}\left(\lambda\right))}^{\c_{\K^*}(s)}$ given by Theorem \ref{thm:Broue-Malle-Michel for connected center} with those given by \cite[Theorem 3.2 (2)]{Bro-Mal-Mic93}, then we obtain a collection of isometries
\[\mathcal{I}^\K_{(\L,\lambda)}:\mathbb{Z}\irr\left(W_\K(\L,\lambda)^F\right)\to\mathbb{Z}\E\left(\K^F,(\L,\lambda)\right)\]
that satisfy certain important properties. However, notice that the maps given in \eqref{eq:Extending linearly} agree with the new maps $\mathcal{I}^\K_{(\L,\lambda)}$ only up to a choice of signs.

The next result should be compared to \cite[Theorem 3.2]{Bro-Mal-Mic93} and \cite[Theorem 1.4 (b)]{Kes-Mal13} (see also \cite[Definition 2.9]{Kes-Mal13}).

\begin{theo}
\label{thm:Compatibility with DL-induction}
Assume Hypothesis \ref{hyp:e-Harish-Chandra theory, no regular embedding}. Then there exist a collection of isometries
\[\mathcal{I}^\K_{(\L,\lambda)}:\mathbb{Z}\irr\left(W_\K(\L,\lambda)^F\right)\to\mathbb{Z}\E\left(\K^F,(\L,\lambda)\right)\]
where $\K$ runs over the set of $e$-split Levi subgroups of $\hspace{1pt}\G$ and $(\L,\lambda)$ over the set of $e$-cuspidal pairs of $\hspace{1pt}\K$ such that:
\begin{enumerate}
\item for all $\K$ and all $(\L,\lambda)$ we have
\[\R_\K^\G\circ\mathcal{I}^\K_{(\L,\lambda)}=\mathcal{I}^\G_{(\L,\lambda)}\circ {\rm Ind}^{W_\G(\L,\lambda)^F}_{W_\K(\L,\lambda)^F};\]
\item the collection $(\mathcal{I}^\K_{(\L,\lambda)})_{\K,(\L,\lambda)}$ is stable under the action of the Weyl group $W_{\G^F}$;
\item $\mathcal{I}^\K_{(\L,\lambda)}$ maps the trivial character of the trivial group $W_\L(\L,\lambda)^F$ to $\lambda$.
\end{enumerate}
\end{theo}

\begin{proof}
As explained previously, the maps are constructed as in \eqref{eq:e-Harisch-Chandra parametrization for connected center} by replacing the maps $I_{(\c_{\L^*}(s),J_{\L,s}\left(\lambda\right))}^{\c_{\K^*}(s)}$ given by Theorem \ref{thm:Broue-Malle-Michel for connected center} with those given by \cite[Theorem 3.2 (2)]{Bro-Mal-Mic93}. Consider $\K$ and $(\L,\lambda)$ as above and fix $\eta\in\irr(W_\K(\L,\lambda)^F)$. Since $W_\K(\L,\lambda)^F\leq W_\G(\L,\lambda)^F$, the construction given in the proof of \cite[Corollary 3.3]{Cab-Spa13} shows that the map $i^\K_{\L,\lambda}$ given by Corollary \ref{cor:Frattini isomorphism for connected center} coincides with the restriction of $i^\G_{\L,\lambda}$ to $W_\K(\L,\lambda)^F$. In particular, if we write $\rho^{i^\K_{\L,\lambda}}$ to denote the element of $\mathbb{Z}\irr(W_{\c_{\K^*}(s)}(\c_{\L^*}(s),J_{\L,s}(\lambda))^{F^*})$ corresponding to $\rho\in\mathbb{Z}\irr(W_\K(\L,\lambda)^F)$ via the isomorphism $i^{\K}_{\L,\lambda}$, it follows by elementary character theory that
\begin{equation}
\label{eq:Relative Weyl group and induction}
\left({\rm Ind}^{W_\G(\L,\lambda)^F}_{W_\K(\L,\lambda)^F}(\eta)\right)^{i^\G_{\L,\lambda}}={\rm Ind}^{W_{\c_{\G^*}(s)}(\c_{\L^*}(s),J_{\L,s}(\lambda))^{F^*}}_{W_{\c_{\K^*}(s)}(\c_{\L^*}(s),J_{\L,s}(\lambda))^{F^*}}\left(\eta^{i^\K_{\L,\lambda}}\right).
\end{equation}
By the definition given in \eqref{eq:e-Harisch-Chandra parametrization for connected center} and applying Theorem \ref{thm:Jordan decomposition for connected center case} (ii) and \cite[Theorem 3.2 (2.a)]{Bro-Mal-Mic93} we conclude that
\begin{align*}
\R^\G_\K\circ\mathcal{I}^\K_{(\L,\lambda)}(\eta)&=\R_\K^\G\left(J^{-1}_{\K,s}\left(\mathcal{I}_{(\c_{\L^*}(s),J_{\L,s}(\lambda))}^{\c_{\K^*}(s)}\left(\eta^{i^\K_{\L,\lambda}}\right)\right)\right)
\\
&=J^{-1}_{\G,s}\circ\R_{\c_{\K^*}(s)}^{\c_{\G^*}(s)}\left(\mathcal{I}_{(\c_{\L^*}(s),J_{\L,s}(\lambda))}^{\c_{\K^*}(s)}\left(\eta^{i^\K_{\L,\lambda}}\right)\right)
\\
&=J^{-1}_{\G,s}\circ\mathcal{I}_{(\c_{\L^*}(s),J_{\L,s}(\lambda))}^{\c_{\G^*}(s)}\left({\rm Ind}^{W_{\c_{\G^*}}(\c_{\L^*}(s),J_{\L,s}(\lambda))^{F^*}}_{W_{\c_{\K^*}}(\c_{\L^*}(s),J_{\L,s}(\lambda))^{F^*}}\left(\eta^{i^\K_{\L,\lambda}}\right)\right)
\\
&=J^{-1}_{\G,s}\circ\mathcal{I}_{(\c_{\L^*}(s),J_{\L,s}(\lambda))}^{\c_{\G^*}(s)}\left(\left({\rm Ind}^{W_\G(\L,\lambda)^F}_{W_\K(\L,\lambda)^F}(\eta)\right)^{i^\G_{\L,\lambda}}\right)
\\
&=\mathcal{I}^\G_{(\L,\lambda)}\circ{\rm Ind}^{W_\G(\L,\lambda)^F}_{W_\K(\L,\lambda)^F}(\eta)
\end{align*}
where the penultimate equality holds because of \eqref{eq:Relative Weyl group and induction}. This proves (i).

The other properties follow by a similar argument. First, (ii) follows from \cite[Theorem 3.2 (2.b)]{Bro-Mal-Mic93} together with Theorem \ref{thm:Jordan decomposition for connected center case} (i) and recalling the compatibility with automorphisms obtained in Corollary \ref{cor:Frattini isomorphism for connected center}. Secondly, to prove (iii) we observe that the trivial character of $W_\L(\L,\lambda)^F$ maps to the trivial character of $W_{\c_{\L^*}(s)}(\c_{\L^*}(s),J_{\L,s}(\lambda))^F$ via the isomorphism $i^\L_{\L,\lambda}$ while the character $J_{\L,s}(\lambda)$ is mapped to $\lambda$ via $J_{\L,s}^{-1}$. Then (iii) follows from \cite[Theorem 3.2 (2.c)]{Bro-Mal-Mic93}.
\end{proof}

\subsection{Consequences of equivariant maximal extendibility}
\label{sec:Equivariant maximal extendility}

We start by recalling the definition of maximal extendibility (see \cite[Definition 3.5]{Mal-Spa16}).

\begin{defin}
\label{def:Maximal extendibility}
Let $Y\unlhd X$ be finite groups and consider $\mathcal{Y}\subseteq \irr(Y)$. Then, we say that \emph{maximal extendibility} holds for $\mathcal{Y}$ with respect to $Y\unlhd X$ if every $\vartheta\in\mathcal{Y}$ extends to $X_\vartheta$. In this case, an \emph{extension map} is any map
\[\Lambda:\mathcal{Y}\to\coprod\limits_{Y\leq X'\leq X}\irr(X')\]
such that for every $\vartheta\in\mathcal{Y}$, the character $\Lambda(\vartheta)\in\irr(X_\vartheta)$ is an extension of $\vartheta$. If $\mathcal{Y}=\irr(Y)$, then we just say that maximal extendibility holds with respect to $Y\unlhd X$.
\end{defin}

In this section we consider the following hypothesis.

\begin{hyp}
\label{hyp:e-Harish-Chandra theory}
Let $\G$, $F:\G\to \G$, $q$, $\ell$ and $e$ be as in Notation \ref{notation} and suppose that $\G$ is a simple algebraic group not of type $\mathbf{E}_6$, $\mathbf{E}_7$ or $\mathbf{E}_8$.
\end{hyp}

As in Section \ref{sec:Automorphisms}, let $i:\G\to\wt{\G}$ be a regular embedding compatible with $F$ and consider the group $\mathcal{A}$ generated by field and graph automorphisms of $\G$ in such a way that $\mathcal{A}$ acts on $\wt{\G}^F$. Then we can define the semidirect product $\wt{\G}^F\rtimes \mathcal{A}$. For every $F$-stable closed connected subgroup $\H$ of $\G$ we denote by $(\wt{\G}^F\mathcal{A})_\H$ the stabilizer of $\H$ under the action of $\wt{\G}^F\mathcal{A}$.

Consider $\mathcal{K}$ as in Section \ref{sec:Regular embeddings}. We form the external semidirect product $(\wt{\G}^F\mathcal{A})\ltimes \mathcal{K}$ where, for $x\in \wt{\G}^F\mathcal{A}$ and $z\in \mathcal{K}$, the element $z^x$ is defined as the unique element of $\mathcal{K}$ corresponding to $(\wh{z}_{\wt{\G}})^x\in\irr(\wt{\G}^F/\G^F)$ via \eqref{central kernel and linear characters}. For every $F$-stable Levi subgroup $\L$ of $\G$, notice that $(\wt{\G}^F\mathcal{A})_\L\ltimes \mathcal{K}$ acts on $\irr(\wt{\L}^F)$ via
\[\wt{\lambda}^{xz}:=\wt{\lambda}^x\cdot \wh{z}_{\wt{\L}}\]
for every $\wt{\lambda}\in\irr(\wt{\L}^F)$, $x\in (\wt{\G}^F\mathcal{A})_\L$ and $z\in \mathcal{K}$. We denote by $((\wt{\G}^F\mathcal{A})_\L\ltimes \mathcal{K})_{\wt{\lambda}}$ the stabilizer of $\wt{\lambda}$ under this action.

Let $\wt{\L}$ and $\wt{\K}$ be $F$-stable Levi subgroups of $\wt{\G}$ with $\wt{\L}\leq \wt{\K}$ and consider an extension map $\wt{\Lambda}$ with respect to $\wt{\L}\unlhd \n_{\wt{\K}}(\L)^F$. In this case notice that
\[\wt{\Lambda}(\wt{\lambda})^{xz}:=\wt{\Lambda}(\wt{\lambda})^x\cdot\wh{z}_{\n_{\wt{\K}}(\wt{\L},\wt{\lambda}^x)^F}\]
is an extension of $\wt{\lambda}^{xz}$ to $\n_{\wt{\K}}(\wt{\L},\wt{\lambda}^{xz})^F=\n_{\wt{\K}}(\wt{\L},\wt{\lambda}^x)^F$, where $\wh{z}_{\n_{\wt{\K}}(\wt{\L},\wt{\lambda}^x)^F}$ denotes the restriction of $\wh{z}_{\wt{\K}}$ to $\n_{\wt{\K}}(\wt{\L},\wt{\lambda}^x)^F$.

The next definition should be compared with condition ${\rm{B}}(d)$ of \cite[Definition 2.2]{Cab-Spa19} with $d=e$.

\begin{defin}
We say that an extension map $\wt{\Lambda}$ with respect to $\wt{\L}^F\unlhd \n_{\wt{\K}}(\L)^F$ is $((\wt{\G}^F\mathcal{A})_{\K,\L}\ltimes \mathcal{K})$-equivariant if $\wt{\Lambda}(\wt{\lambda}^{xz})=\wt{\Lambda}(\wt{\lambda})^{xz}$ for every $\wt{\lambda}\in\irr(\wt{\L}^F)$, $x\in(\wt{\G}^F\mathcal{A})_{\K,\L}$ and $z\in \mathcal{K}$. Moreover, if $\cusp{\wt{\L}^F}$ denotes the set of (irreducible) $e$-cuspidal characters of $\L^F$, then $(\wt{\G}^F\mathcal{A})_{\K,\L}\ltimes \mathcal{K}$ acts on $\cusp{\wt{\L}^F}$ (see \cite[Proposition 12.1]{Bon06}) and therefore we can also consider a $((\wt{\G}^F\mathcal{A})_{\K,\L}\ltimes \mathcal{K})$-equivariant extension map $\wt{\Lambda}$ for $\cusp{\wt{\L}^F}$ with respect to $\wt{\L}^F\unlhd \n_{\wt{\K}}(\L)^F$.
\end{defin}

Let $\wt{\K}$ be an $F$-stable Levi subgroup of $\wt{\G}$ and consider an $e$-cuspidal pair $(\wt{\L},\wt{\lambda})$ of $\wt{\K}$. Using the bijection $I^{\wt{\K}}_{(\wt{\L},\wt{\lambda})}$ from \eqref{eq:e-Harisch-Chandra parametrization for connected center} and assuming the existence of an extension map $\wt{\Lambda}$ for $\cusp{\wt{\L}^F}$ with respect to $\wt{\L}^F\unlhd \n_{\wt{\K}}(\L)^F$, we can define the map
\begin{align}
\label{eq:Bijection for connected center}
\Upsilon_{(\wt{\L},\wt{\lambda})}^{\wt{\K}}:\E\left(\wt{\K}^F,\left(\wt{\L},\wt{\lambda}\right)\right)&\to\irr\left(\n_{\wt{\K}}(\L)^F\enspace\middle|\enspace \wt{\lambda}\right)
\\
I^{\wt{\K}}_{(\wt{\L},\wt{\lambda})}\left(\wt{\eta}\right)&\mapsto\left(\wt{\Lambda}\left(\wt{\lambda}\right)\cdot \wt{\eta}\right)^{\n_{\wt{\K}}(\L)^F}\nonumber
\end{align}
for every $\wt{\eta}\in\irr\left(W_{\wt{\K}}\left(\wt{\L},\wt{\lambda}\right)^F\right)$. Notice that $\Upsilon_{(\wt{\L},\wt{\lambda})}^{\wt{\K}}$ is a bijection by the Clifford correspondence and Gallagher's theorem (see \cite[Theorem 6.11 and Corollary 6.17]{Isa76}).

First we show that the bijection $\Upsilon_{(\wt{\L},\wt{\lambda})}^{\wt{\K}}$ from \eqref{eq:Bijection for connected center} preserves the $\ell$-defect of characters. Recall that for any finite group $X$ and any $\chi\in\irr(X)$, the \emph{$\ell$-defect} of $\chi$ is the non-negative integer $d(\chi)$ such that $\ell^{d(\chi)}\chi(1)_\ell=|X|_\ell$.

\begin{lem}
\label{lem:Bijection for connected center, defect}
Assume Hypothesis \ref{hyp:e-Harish-Chandra theory} and suppose there exists an extension map $\wt{\Lambda}$ for $\cusp{\wt{\L}^F}$ with respect to $\wt{\L}^F\unlhd \n_{\wt{\K}}(\L)^F$. For every $\wt{\eta}\in\irr\left(W_{\wt{\K}}\left(\wt{\L},\wt{\lambda}\right)^F\right)$ we have
\[d\left(I^{\wt{\K}}_{(\wt{\L},\wt{\lambda})}\left(\wt{\eta}\right)\right)=d\left(\left(\wt{\Lambda}\left(\wt{\lambda}\right)\cdot \wt{\eta}\right)^{\n_{\wt{\K}}(\L)^F}\right).\]
\end{lem}

\begin{proof}
This follows immediately from Lemma \ref{lem:e-Harisch-Chandra parametrization for connected center, degree} applied to $\wt{\G}$ after noticing that induction of characters preserves the defect (this follows from the degree formula for induced characters).
\end{proof}

The bijection $\Upsilon_{(\wt{\L},\wt{\lambda})}^{\wt{\K}}$ from \eqref{eq:Bijection for connected center} also preserves central characters.

\begin{lem}
\label{lem:Bijection for connected center, restriction to center}
Assume Hypothesis \ref{hyp:e-Harish-Chandra theory} and suppose there exists an extension map $\wt{\Lambda}$ for $\cusp{\wt{\L}^F}$ with respect to $\wt{\L}^F\unlhd \n_{\wt{\K}}(\L)^F$. For every $\wt{\eta}\in\irr\left(W_{\wt{\K}}\left(\wt{\L},\wt{\lambda}\right)^F\right)$ we have
\[\irr\left(I^{\wt{\K}}_{(\wt{\L},\wt{\lambda})}\left(\wt{\eta}\right)_{\z(\wt{\K}^F)}\right)=\irr\left(\left(\left(\wt{\Lambda}\left(\wt{\lambda}\right)\cdot \wt{\eta}\right)^{\n_{\wt{\K}}(\L)^F}\right)_{\z\left(\wt{\K}^F\right)}\right).\]
\end{lem}

\begin{proof}
First, by Clifford theory we deduce that
\begin{equation}
\label{eq:Bijection for connected center, restriction to center 1}
\irr\left(\left(\left(\wt{\Lambda}\left(\wt{\lambda}\right)\cdot \wt{\eta}\right)^{\n_{\wt{\K}}(\L)^F}\right)_{\z\left(\wt{\G}^F\right)}\right)=\irr\left(\wt{\lambda}_{\z(\wt{\G}^F)}\right).
\end{equation}
On the other hand, by using the character formula \cite[Proposition 12.2 (i)]{Dig-Mic91}, we obtain
\[\R^{\wt{\K}}_{\wt{\L}}(\wt{\lambda})_{\z\left(\wt{\K}^F\right)}=\R^{\wt{\K}}_{\wt{\L}}(\wt{\lambda})(1)\cdot \wt{\lambda}_{\z\left(\wt{\K}^F\right)}\]
and hence
\begin{equation}
\label{eq:Bijection for connected center, restriction to center 2}
\irr\left(I^{\wt{\K}}_{(\wt{\L},\wt{\lambda})}\left(\wt{\eta}\right)_{\z(\wt{\K}^F)}\right)=\irr\left(\wt{\lambda}_{\z(\wt{\K}^F)}\right).
\end{equation}
Now the result follows by combining \eqref{eq:Bijection for connected center, restriction to center 1} with \eqref{eq:Bijection for connected center, restriction to center 2}.
\end{proof}

Next, we show that the bijection $\Upsilon_{(\wt{\L},\wt{\lambda})}^{\wt{\K}}$ from \eqref{eq:Bijection for connected center} is compatible with block induction.

\begin{lem}
\label{lem:Bijection for connected center, block induction}
Assume Hypothesis \ref{hyp:Brauer--Lusztig blocks}, Hypothesis \ref{hyp:e-Harish-Chandra theory} and suppose there exists an extension map $\wt{\Lambda}$ for $\cusp{\wt{\L}^F}$ with respect to $\wt{\L}^F\unlhd \n_{\wt{\K}}(\L)^F$. Then
\[\bl\left(\wt{\chi}\right)=\bl\left(\Upsilon_{(\wt{\L},\wt{\lambda})}^{\wt{\K}}\left(\wt{\chi}\right)\right)^{\wt{\K}^F}\]
for every $\wt{\chi}\in\E(\wt{\K}^F,(\wt{\L},\wt{\lambda}))$.
\end{lem}

\begin{proof}
Since $\bl(\wt{\lambda})^{\wt{\K}^F}=\bl(\wt{\chi})$ by \cite[Proposition 4.9]{Ros-Generalized_HC_theory_for_Dade} and $\bl(\Upsilon_{(\wt{\L},\wt{\lambda})}^{\wt{\K}}(\wt{\chi}))=\bl(\wt{\lambda})^{\n_{\wt{\K}}(\L)^F}$ by \cite[Lemma 6.5]{Ros-Generalized_HC_theory_for_Dade}, the result follows from the transitivity of block induction.
\end{proof}

Finally, we show that the bijection $\Upsilon_{(\wt{\L},\wt{\lambda})}^{\wt{\K}}$ from \eqref{eq:Bijection for connected center} is equivariant.

\begin{lem}
\label{lem:Bijection for connected center, equivariance}
Assume Hypothesis \ref{hyp:e-Harish-Chandra theory} and suppose there exists a $(\wt{\G}^F\mathcal{A})_{\K,\L}\ltimes \mathcal{K}$-equivariant extension map $\wt{\Lambda}$ for $\cusp{\wt{\L}^F}$ with respect to $\wt{\L}^F\unlhd \n_{\wt{\K}}(\L)^F$. Then $\Upsilon_{(\wt{\L},\wt{\lambda})}^{\wt{\K}}$ is $\left(\left(\wt{\G}^F\mathcal{A}\right)_{\K,\L}\ltimes \mathcal{K}\right)_{\wt{\lambda}}$-equivariant.
\end{lem}

\begin{proof}
Let $(x,z)\in\left(\left(\wt{\G}^F\mathcal{A}\right)_{\K,\L}\ltimes \mathcal{K}\right)_{\wt{\lambda}}$. Since $\wt{\lambda}=\wt{\lambda}^x\cdot \wh{z}_{\wt{\L}}$, we have
\[\n_{\wt{\K}}(\L,\wt{\lambda})^F=\n_{\wt{\K}}(\L,\wt{\lambda}^x\cdot \wh{z}_{\wt{\L}})^F=\n_{\wt{\K}}(\L,\wt{\lambda}^x)^F.\]
By using the equivariance properties of $\wt{\Lambda}$, we obtain
\begin{align}
\left(\left(\wt{\Lambda}\left(\wt{\lambda}\right)\cdot \wt{\eta}\right)^{\n_{\wt{\K}}(\L)^F}\right)^{(x,z)}&=\left(\left(\wt{\Lambda}\left(\wt{\lambda}\right)\cdot \wt{\eta}\right)^x\right)^{\n_{\wt{\K}}(\L)^F}\cdot \wh{z}_{\n_{\wt{\K}}(\L)} \nonumber
\\
&=\left(\left(\wt{\Lambda}\left(\wt{\lambda}\right)\cdot \wt{\eta}\right)^x\cdot \wh{z}_{\n_{\wt{\K}}(\L,\wt{\lambda})^F}\right)^{\n_{\wt{\K}}(\L)^F} \label{eq:Bijection for connected center, equivariance 1}
\\
&=\left(\wt{\Lambda}\left(\wt{\lambda}^x\cdot \wh{z}_{\wt{\L}} \right)\cdot \wt{\eta}^x\right)^{\n_{\wt{\K}}(\L)^F} \nonumber
\\
&=\left(\wt{\Lambda}\left(\wt{\lambda}\right)\cdot \wt{\eta}^x\right)^{\n_{\wt{\K}}(\L)^F}. \nonumber
\end{align}
On the other hand, considering Lemma \ref{lem:e-Harisch-Chandra parametrization for connected center, equivariant} and Lemma \ref{lem:e-Harisch-Chandra parametrization for connected center, central} with respect to $\wt{\G}$ it follows that
\begin{align}
I^{\wt{\K}}_{(\wt{\L},\wt{\lambda})}\left(\wt{\eta}\right)^{(x,z)}&=I^{\wt{\K}}_{(\wt{\L},\wt{\lambda})}\left(\wt{\eta}\right)^x\cdot \wh{z}_{\wt{\K}}\nonumber
\\
&=I^{\wt{\K}}_{(\wt{\L},\wt{\lambda}^x\cdot \wh{z}_{\wt{\L}})}\left(\wt{\eta}^x\right) \label{eq:Bijection for connected center, equivariance 2}
\\
&=I^{\wt{\K}}_{(\wt{\L},\wt{\lambda})}\left(\wt{\eta}^x\right). \nonumber
\end{align}
Now, the result follows immediately from \eqref{eq:Bijection for connected center, equivariance 1} and \eqref{eq:Bijection for connected center, equivariance 2}.
\end{proof}

\subsection{$e$-Harish-Chandra series and regular embeddings}
\label{subsec:e-HC series and regular embeddings}

We use the results obtained in the previous two subsections in order to obtain the bijections needed in the criteria we prove in Section \ref{sec:Criteria} (see Theorem \ref{thm:Criterion, central} and Theorem \ref{thm:Criterion}).

To start, we study the behaviour of $e$-Harish-Chandra series with respect to the regular embedding $i:\G\to \wt{\G}$. Fix an $F$-stable Levi subgroup $\K$ of $\G$ and an $e$-cuspidal pair $(\L,\lambda)$ of $\K$. Observe that $\wt{\K}:=\K\z(\wt{\G})$ is an $F$-stable Levi subgroup of $\wt{\G}$ and that $(\wt{\L},\wt{\lambda})$ is an $e$-cuspidal pair of $\wt{\K}$ for every $\wt{\lambda}\in\irr(\wt{\L}^F\mid \lambda)$ and where $\wt{\L}:=\L\z(\wt{\G})$ (see \cite[Proposition 10.10]{Bon06}).

\begin{defin}
\label{def:e-HC and regular embedding}
Let $\HC(\wt{\K}^F,(\L,\lambda))$ be the set of $e$-Harish-Chandra series $\E(\wt{\K}^F,(\wt{\L},\wt{\lambda}))$ with $\wt{\lambda}\in\irr(\wt{\L}^F\mid \lambda)$. The group $\mathcal{K}$ from Section \ref{sec:Regular embeddings} acts on the set $\HC(\wt{\K}^F,(\L,\lambda))$ via
\[\E\left(\wt{\K}^F,\left(\wt{\L},\wt{\lambda}\right)\right)^z:=\E\left(\wt{\K}^F,\left(\wt{\L},\wt{\lambda}\cdot \wh{z}_{\wt{\L}}\right)\right)\]
for every $\E(\wt{\K}^F,(\wt{\L},\wt{\lambda}))\in\HC(\wt{\K}^F,(\L,\lambda))$, $z\in \mathcal{K}$ and where $\wh{z}_{\wt{\L}}$ corresponds to $z$ via \eqref{central kernel and linear characters}. Here notice that, as $\lambda$ is $e$-cuspidal, then so are $\wt{\lambda}$ and $\wt{\lambda}\cdot \wh{z}_{\wt{\L}}$ (see \cite[Proposition 10.10 and Proposition 10.11]{Bon06}). Moreover, if we define $\E(\wt{\K}^F,(\wt{\L},\wt{\lambda}))\cdot \wh{z}_{\wt{\K}}$ to be the set of characters $\wt{\chi}\cdot \wh{z}_{\wt{\K}}$ for $\wt{\chi}\in\E(\wt{\K}^F,(\wt{\L},\wt{\lambda}))$, then 
\[\E(\wt{\K}^F,(\wt{\L},\wt{\lambda}))^z=\E(\wt{\K}^F,(\wt{\L},\wt{\lambda}))\cdot \wh{z}_{\wt{\K}}\]
by \cite[Proposition 10.11]{Bon06}.
\end{defin}

We want to compare the action of $\mathcal{K}$ on $\HC(\wt{\K}^F,(\L,\lambda))$ with the action of $\mathcal{K}$ on the set of characters $\irr(\wt{\L}^F\mid \lambda)$. First, observe that \cite[Problem 6.2]{Isa76} implies that both actions are transitive.

\begin{lem}
\label{lem:Covering Brauer--Lusztig blocks and cuspidality with actions}
Assume Hypothesis \ref{hyp:Brauer--Lusztig blocks} and let $\wt{\lambda}_i\in\irr(\wt{\L}^F\mid\lambda)$ for $i=1,2$. Let $z\in \mathcal{K}$, then
\[\E\left(\wt{\K}^F,\left(\wt{\L},\wt{\lambda}_1\right)\right)=\E\left(\wt{\K}^F,\left(\wt{\L},\wt{\lambda}_2\right)\right)^z\]
if and only if
\[\wt{\lambda}_1=\wt{\lambda}_2^x\cdot \wh{z}_{\wt{\L}}\]
for some $x\in\n_{\wt{\K}}(\L,\lambda)^F$. 
\end{lem}

\begin{proof}
First, assume $\E(\wt{\K}^F,(\wt{\L},\wt{\lambda}_1))=\E(\wt{\K}^F,(\wt{\L},\wt{\lambda}_2))^z$. By \cite[Proposition 4.10]{Ros-Generalized_HC_theory_for_Dade}, there exists $u\in \wt{\K}^F$ such that $(\wt{\L},\wt{\lambda}_1)=(\wt{\L},\wt{\lambda}_2\cdot \wh{z}_{\wt{\L}})^u$. This implies that $u\in \n_{\wt{\K}}(\L)^F$ and that $\wt{\lambda}_1=\wt{\lambda}_2^u\cdot \wh{z}_{\wt{\L}}$. Moreover, since $\wt{\lambda}_1$ lies over both $\lambda$ and $\lambda^u$, it follows from Clifford's theorem that $\lambda=\lambda^{uv}$, for some $v\in \wt{\L}^F$. Then $x:=uv\in\n_{\wt{\K}}(\L,\lambda)^F$ and $\wt{\lambda}_1=\wt{\lambda}_2^x\cdot \wh{z}_{\wt{\L}}$. Conversely, if $\wt{\lambda}_1=\wt{\lambda}_2^x\cdot \wh{z}_{\wt{\L}}$ for some $x\in\n_{\wt{\K}}(\L,\lambda)^F$, then \cite[Proposition 10.11]{Bon06} yields the desired equality.
\end{proof}

\begin{cor}
\label{cor:Covering Brauer--Lusztig blocks and cuspidality with actions}
Assume Hypothesis \ref{hyp:Brauer--Lusztig blocks} and consider $\wt{\lambda}\in\irr(\wt{\L}^F\mid \lambda)$. Then
\[\mathcal{K}_{\E(\wt{\K}^F,(\wt{\L},\wt{\lambda}))}\leq \n_{\wt{\K}}(\L,\lambda)^F(\n_{\wt{\K}}(\L,\lambda)^F\ltimes\mathcal{K})_{\wt{\lambda}}\]
where $\mathcal{K}_{\E(\wt{\K}^F,(\wt{\L},\wt{\lambda}))}$ denotes the stabilizer of $\E(\wt{\K}^F,(\wt{\L},\wt{\lambda}))$ under the action of $\mathcal{K}$ on $\HC(\wt{\K}^F,(\L,\lambda))$.
\end{cor}

\begin{proof}
Let $z\in\mathcal{K}$ stabilize $\E(\wt{\K}^F,(\wt{\L},\wt{\lambda}))$. By Lemma \ref{lem:Covering Brauer--Lusztig blocks and cuspidality with actions} there exists $x\in\n_{\wt{\K}}(\L,\lambda)^F$ such that $\wt{\lambda}=\wt{\lambda}^x\cdot \wh{z}_{\wt{\L}}$ and hence $z=x^{-1}xz\in\n_{\wt{\K}}(\L,\lambda)^F(\n_{\wt{\K}}(\L,\lambda)^F\ltimes\mathcal{K})_{\wt{\lambda}}$.
\end{proof}

For every finite group $X$ with subgroup $Y\leq X$ and every subset of characters $\mathcal{Y}\subseteq \irr(Y)$, we denote by $\irr(X\mid \mathcal{Y})$ the set of characters $\chi\in\irr(X)$ lying over some character $\psi\in\mathcal{Y}$.

\begin{prop}
\label{prop:Covering Brauer--Lusztig blocks and cuspidality with actions, partitions}
Assume Hypothesis \ref{hyp:Brauer--Lusztig blocks} and let $\wt{\lambda}\in\irr(\wt{\L}^F\mid \lambda)$. If $\mathcal{T}$ is a transversal for the cosets of $\mathcal{K}_{\E(\wt{\K}^F,(\wt{\L},\wt{\lambda}))}$ in $\mathcal{K}$, then
\begin{equation}
\label{eq:Global partition}
\irr\left(\wt{\K}^F\enspace\middle|\enspace \E\left(\K^F,(\L,\lambda)\right)\right)=\coprod\limits_{z\in\mathcal{T}}\E\left(\wt{\K}^F,\left(\wt{\L},\wt{\lambda}\right)\right)\cdot \wh{z}_{\wt{\K}}
\end{equation}
and
\begin{equation}
\label{eq:Local partition}
\irr\left(\n_{\wt{\K}}(\L)^F\enspace\middle|\enspace\lambda\right)=\coprod\limits_{z\in\mathcal{T}}\irr\left(\n_{\wt{\K}}(\L)^F\enspace\middle|\enspace \wt{\lambda}\right)\cdot \wh{z}_{\n_{\wt{\K}}(\L)},
\end{equation}
where $\irr(\n_{\wt{\K}}(\L)^F\mid \wt{\lambda})\cdot \wh{z}_{\n_{\wt{K}}(\L)}$ is the set of characters $\wt{\psi}\cdot\wh{z}_{\n_{\wt{\K}}(\L)}$ for $\wt{\psi}\in \irr(\n_{\wt{\K}}(\L)^F\mid \wt{\lambda})$.
\end{prop}

\begin{proof}
Set $\wt{\mathcal{G}}:=\irr(\wt{\K}^F\mid \E(\K^F,(\L,\lambda)))$ and $\wt{\mathcal{N}}:=\irr(\n_{\wt{\K}}(\L)^F\mid\lambda)$. First, we claim that $\wt{\mathcal{G}}$ is the union of the $e$-Harish-Chandra series in the set $\HC(\wt{\K}^F,(\L,\lambda))$. In fact, if $\wt{\chi}\in\wt{\mathcal{G}}$, then there exists $\chi\in\E(\K^F,(\L,\lambda))$ lying below $\wt{\chi}$. By \cite[Corollary 3.3.25]{Gec-Mal20}, it follows that $\wt{\chi}$ is an irreducible constituent of $\R_{\wt{\L}}^{\wt{\K}}(\lambda^{\wt{\L}^F})$ and therefore there exists $\wt{\nu}\in\irr(\wt{\L}^F\mid \lambda)$ such that $\wt{\chi}\in\E(\wt{\K},(\wt{\L},\wt{\nu}))$. On the other hand, if $\wt{\nu}\in\irr(\wt{\L}^F\mid \lambda)$ and $\wt{\chi}\in\E(\wt{\K}^F,(\wt{\L},\wt{\nu}))$, then \cite[Corollary 3.3.25]{Gec-Mal20} implies that $\wt{\chi}$ lies over some character $\chi\in\E(\K^F,(\L,\lambda))$. Since the action of $\mathcal{K}$ on $\HC(\wt{\K}^F,(\L,\lambda))$ is transitive and recalling the definition of $\mathcal{T}$, we hence obtain \eqref{eq:Global partition}.

Now we prove \eqref{eq:Local partition}. By Clifford theory, we know that every element of $\mathcal{G}$ lies above some character $\wt{\nu}\in\irr(\wt{\L}\mid \lambda)$. Since $\mathcal{K}$ is transitive on $\irr(\wt{\L}^F\mid \lambda)$, we deduce that $\wt{\mathcal{N}}$ is contained in the union
\[\bigcup\limits_{z\in\mathcal{K}}\irr\left(\n_{\wt{\K}}(\L)^F\enspace\middle|\enspace \wt{\lambda}\right)\cdot \wh{z}_{\n_{\wt{\K}}(\L)}.\]
Moreover, we claim that the above union coincides with
\begin{equation}
\label{eq:Covering Brauer--Lusztig blocks and cuspidality with actions, partitions}
\bigcup\limits_{z\in\mathcal{T}}\irr\left(\n_{\wt{\K}}(\L)^F\enspace\middle|\enspace \wt{\lambda}\right)\cdot \wh{z}_{\n_{\wt{\K}}(\L)}.
\end{equation}
To see this, let $z\in\mathcal{K}$ and write $z=z_0t$, for some $z_0\in\mathcal{K}_{\E(\wt{\K}^F,(\wt{\L},\wt{\lambda}))}$ and $t\in \mathcal{T}$. By Corollary \ref{cor:Covering Brauer--Lusztig blocks and cuspidality with actions} we obtain $z_0\in\n_{\wt{\K}}(\L,\lambda)^F(\n_{\wt{\K}}(\L,\lambda)^F\ltimes\mathcal{K})_{\wt{\lambda}}$ and therefore
\[\irr\left(\n_{\wt{\K}}(\L)^F\enspace\middle|\enspace \wt{\lambda}\right)\cdot \wh{z}_{\n_{\wt{\K}}(\L)}=\irr\left(\n_{\wt{\K}}(\L)^F\enspace\middle|\enspace \wt{\lambda}\right)\cdot \wh{t}_{\n_{\wt{\K}}(\L)}.\]
This proves our claim and it remains to show that the union in \eqref{eq:Covering Brauer--Lusztig blocks and cuspidality with actions, partitions} is disjoint. Assume that, for some $z\in\mathcal{T}$, there exists a character $\wt{\psi}$ inside both $\irr(\n_{\wt{\K}}(\L)^F\mid \wt{\lambda})$ and $\irr(\n_{\wt{\K}}(\L)^F\mid \wt{\lambda})\cdot \wh{z}_{\n_{\wt{\K}}(\L)}$. By \cite[Problem 5.3]{Isa76} we deduce that $\irr(\n_{\wt{\K}}(\L)^F\mid \wt{\lambda})\cdot \wh{z}_{\n_{\wt{\K}}(\L)}=\irr(\n_{\wt{\K}}(\L)^F\mid \wt{\lambda}\cdot \wh{z}_{\wt{\L}})$ and hence $\wt{\psi}$ lies above $\wt{\lambda}$ and $\wt{\lambda}\cdot \wh{z}_{\wt{\L}}$. By Clifford's theorem $\wt{\lambda}=(\wt{\lambda}\cdot \wh{z}_{\wt{\L}})^u=\wt{\lambda}^u\cdot \wh{z}_{\wt{\L}}$, for some $u\in\n_{\wt{\K}}(\L)^F$ and now Lemma \ref{lem:Covering Brauer--Lusztig blocks and cuspidality with actions} implies $\E(\wt{\K}^F,(\wt{\L},\wt{\lambda}))=\E(\wt{\K}^F,(\wt{\L},\wt{\lambda}))^z$. By the definition of $\mathcal{T}$ it follows that the union in \eqref{eq:Covering Brauer--Lusztig blocks and cuspidality with actions, partitions} is disjoint.
\end{proof}

As a corollary of Proposition \ref{prop:Covering Brauer--Lusztig blocks and cuspidality with actions, partitions} and using the bijection $\Upsilon_{(\wt{\L},\wt{\lambda})}^{\wt{\K}}$ from \eqref{eq:Bijection for connected center}, we are finally able to prove the main result of this subsection. The bijection described in the following theorem is part of the requirements of the criteria we prove in Section \ref{sec:Criteria} (see Assumption \ref{ass:Assumption for the criterion, central} (ii) and Assumption \ref{ass:Assumption for the criterion} (ii)).

\begin{theo}
\label{thm:Bijection for connected center}
Assume Hypothesis \ref{hyp:Brauer--Lusztig blocks} and Hypothesis \ref{hyp:e-Harish-Chandra theory}. Suppose there exists a $((\wt{\G}^F\mathcal{A})_{\K,\L}\ltimes \mathcal{K})$-equivariant extension map for $\cusp{\wt{\L}^F}$ with respect to $\wt{\L}^F\unlhd \n_{\wt{\K}}(\L)^F$. Then, there exists a defect preserving $\left((\wt{\G}^F\mathcal{A})_{\K,(\L,\lambda)}\ltimes \mathcal{K}\right)$-equivariant bijection
\[\wt{\Omega}_{(\L,\lambda)}^\K:\irr\left(\wt{\K}^F\enspace\middle|\enspace \E\left(\K^F,(\L,\lambda)\right)\right)\to\irr\left(\n_{\wt{\K}}(\L)^F\enspace\middle|\enspace\lambda\right)\]
such that, for every $\wt{\chi}\in \irr(\wt{\K}^F\mid \E(\K^F,(\L,\lambda)))$, the following conditions hold:
\begin{enumerate}
\item $\irr\left(\wt{\chi}_{\z(\wt{\K}^F)}\right)=\irr\left(\wt{\Psi}(\wt{\chi})_{\z(\wt{\K}^F)}\right)$;
\item $\bl\left(\wt{\chi}\right)=\bl\left(\wt{\Psi}\left(\wt{\chi}\right)\right)^{\wt{\K}^F}$.
\end{enumerate}
\end{theo}

\begin{proof}
Set $\wt{\mathcal{G}}:=\irr(\wt{\K}^F\mid \E(\K^F,(\L,\lambda)))$ and $\wt{\mathcal{N}}:=\irr(\n_{\wt{\K}}(\L)^F\mid \lambda)$ and fix $\wt{\lambda}\in\irr(\wt{\L}^F\mid\lambda)$. Let
\[\Upsilon_{(\wt{\L},\wt{\lambda})}^{\wt{\K}}:\E\left(\wt{\K}^F,\left(\wt{\L},\wt{\lambda}\right)\right)\to\irr\left(\n_{\wt{\K}}(\L)^F\enspace\middle|\enspace\wt{\lambda}\right)\]
be the bijection constructed in \eqref{eq:Bijection for connected center}. Let $\wt{\mathbb{T}}_{\rm glo}$ be a $((\wt{\G}^F\mathcal{A})_{\K,\L}\ltimes \mathcal{K})_{\wt{\lambda}}$-transversal in $\E(\wt{\K}^F,(\wt{\L},\wt{\lambda}))$ and observe that, by Lemma \ref{lem:Bijection for connected center, equivariance}, the set $\wt{\mathbb{T}}_{\rm loc}:=\{\Upsilon_{(\wt{\L},\wt{\lambda})}^{\wt{\K}}(\wt{\chi})\mid \wt{\chi}\in\wt{\mathbb{T}}_{\rm glo}\}$ is a $((\wt{\G}^F\mathcal{A})_{\K,\L}\ltimes \mathcal{K})_{\wt{\lambda}}$-transversal in $\irr(\n_{\wt{\K}}(\L)^F\mid \lambda)$.

Next, we fix a transversal $\mathcal{T}$ for $\mathcal{K}_{\E(\wt{\K}^F,(\wt{\L},\wt{\lambda}))}$ in $\mathcal{K}$ and
we claim that
\begin{equation}
\label{eq:Bijection for connected center, 1}
\wt{\K}^F_\L\left(\left(\wt{\G}^F\mathcal{A}\right)_{\K,\L}\ltimes \mathcal{K}\right)_{\wt{\lambda}}\cdot\mathcal{T}=\wt{\K}^F_\L\left(\wt{\G}^F\mathcal{A}\right)_{\K,(\L,\lambda)}\ltimes \mathcal{K}.
\end{equation}
To prove this equality, consider $xz\in(\wt{\G}^F\mathcal{A})_{\K,(\L,\lambda)}\ltimes \mathcal{K}$. Then both $\wt{\lambda}$ and $\wt{\lambda}^x\cdot \wh{z}_{\wt{\L}}$ lie over $\lambda$ and by \cite[Problem 6.2]{Isa76} there exists $u\in\mathcal{K}$ such that $\wt{\lambda}=\wt{\lambda}^x\cdot \wh{z}_{\wt{\L}}\cdot \wh{u}_{\wt{\L}}$. Therefore $xz\in ((\wt{\G}^F\mathcal{A})_{\K,\L}\ltimes \mathcal{K})_{\wt{\lambda}}\cdot\mathcal{K}$. On the other hand, applying Corollary \ref{cor:Covering Brauer--Lusztig blocks and cuspidality with actions}, we obtain $\mathcal{K}_{\E(\wt{\K}^F,(\wt{\L},\wt{\lambda}))}\leq \wt{\K}^F_\L(\wt{\K}^F_\L\ltimes \mathcal{K})_{\wt{\lambda}}$ and by the definition of $\mathcal{T}$, we conclude that
\[\wt{\K}^F_\L\left(\left(\wt{\G}^F\mathcal{A}\right)_{\K,\L}\ltimes \mathcal{K}\right)_{\wt{\lambda}}\cdot \mathcal{T}\geq \wt{\K}^F_\L\left(\left(\wt{\G}^F\mathcal{A}\right)_{\K,(\L,\lambda)}\ltimes \mathcal{K}\right).\]
To prove the remaining inclusion it's enough to show that
\[\wt{\L}^F\left(\left(\wt{\G}^F\mathcal{A}\right)_{\K,(\L,\lambda)}\ltimes \mathcal{K}\right)_{\wt{\lambda}}=\left(\left(\wt{\G}^F\mathcal{A}\right)_{\K,\L}\ltimes \mathcal{K}\right)_{\wt{\lambda}}.\]
Since $\wt{\lambda}$ is $\wt{\L}^F$-invariant, one inclusion is trivial. So let $xz\in((\wt{\G}^F\mathcal{A})_{\K,\L}\ltimes \mathcal{K})_{\wt{\lambda}}$ and observe that $\wt{\lambda}=\wt{\lambda}^x\cdot \wh{z}_{\wt{\L}}$ lies both over $\lambda$ and over $\lambda^x$. By Clifford's theorem there exists $y\in\wt{\L}^F$ such that $\lambda=\lambda^{xy}$ and hence $xz\in\wt{\L}^F((\wt{\G}^F\mathcal{A})_{\K,(\L,\lambda)}\ltimes \mathcal{K})_{\wt{\lambda}}$. This proves the claim.

Now, using \eqref{eq:Bijection for connected center, 1}, we show that $\wt{\mathbb{T}}_{\rm glo}$ is a $((\wt{\G}^F\mathcal{A})_{\K,(\L,\lambda)}\ltimes \mathcal{K})$-transversal in $\wt{\mathcal{G}}$. Consider $\wt{\chi}\in\wt{\mathcal{G}}$. By Proposition \ref{prop:Covering Brauer--Lusztig blocks and cuspidality with actions, partitions} there exist unique $z\in\mathcal{T}$ and $\wt{\chi}_0'\in\E(\wt{\K}^F,(\wt{\L},\wt{\lambda}))$ such that $\wt{\chi}=\wt{\chi}_0'\cdot \wh{z}_{\wt{\K}}$. Let $\wt{\chi}_0$ be the unique element in $\wt{\mathbb{T}}_{\rm glo}$ such that $\wt{\chi}_0'=\wt{\chi}_0^x\cdot \wh{u}_{\wt{\K}}$, for some $xu\in((\wt{\G}^F\mathcal{A})_{\K,\L}\ltimes \mathcal{K})_{\wt{\lambda}}$. Then $\wt{\chi}=\wt{\chi}_0^x\cdot\wh{u}_{\wh{K}}\cdot \wh{z}_{\wt{K}}$, for $xuz\in((\wt{\G}^F\mathcal{A})_{\K,\L}\ltimes \mathcal{K})_{\wt{\lambda}}\cdot \mathcal{T}$. But using \eqref{eq:Bijection for connected center, 1} and since $\wt{\chi}$ and $\wt{\chi}_0$ are $\n_{\wt{\K}}(\L)^F$-invariant, we conclude that $\wt{\chi}=\wt{\chi}_0^y\cdot\wh{v}_{\wt{\K}}$, for some $y\in(\wt{\G}^F\mathcal{A})_{\K,(\L,\lambda)}$ and $v\in \mathcal{K}$. This argument also shows that $\wt{\chi}_0$ is the unique element of $\wt{\mathbb{T}}_{\rm glo}$ with this property.

Similarly, using \eqref{eq:Bijection for connected center, 1}, we deduce that the set $\wt{\mathbb{T}}_{\rm loc}$ is a $((\wt{\G}^F\mathcal{A})_{\K,(\L,\lambda)}\ltimes \mathcal{K})$-transversal in $\wt{\mathcal{N}}$. Now, the map
\[\wt{\Omega}_{(\L,\lambda)}^\K:\wt{\mathcal{G}}\to\wt{\mathcal{N}}\]
defined by
\[\wt{\Omega}_{(\L,\lambda)}^\K\left(\wt{\chi}^x\cdot \wh{z}_{\wt{\K}}\right):=\Upsilon_{(\wt{\L},\wt{\lambda})}^{\wt{\K}}(\wt{\chi})^x\cdot \wh{z}_{\n_{\wt{\K}}(\L)},\]
for every $\wt{\chi}\in\wt{\mathbb{T}}_{\rm glo}$, $x\in(\wt{\G}^F\mathcal{A})_{\K,(\L,\lambda)}$ and $z\in\mathcal{K}$, is a $((\wt{\G}^F\mathcal{A})_{\K,(\L,\lambda)}\ltimes \mathcal{K})$-equivariant bijection. The remaining properties follow from Lemma \ref{lem:Bijection for connected center, defect}, Lemma \ref{lem:Bijection for connected center, restriction to center} and Lemma \ref{lem:Bijection for connected center, block induction} after noticing that $\wh{z}_{\wt{\K}}$ and $\wh{z}_{\n_{\wt{\K}}(\L)}$ are linear characters and that
\[\bl\left(\wt{\psi}\cdot\wh{z}_{\n_{\wt{\K}}(\L)}\right)^{\wt{\K}^F}=\bl\left(\wt{\psi}\right)^{\wt{\K}^F}\cdot \wh{z}_{\wt{\K}}\]
for every $\wt{\psi}\in\irr(\n_{\wt{\K}}(\L)^F)$ and $z\in\mathcal{K}$.
\end{proof}

\section{The criteria}
\label{sec:Criteria}

In this section we prove Theorem \ref{thm:Criterion, central} and Theorem \ref{thm:Criterion} which serve as criteria for Condition \ref{cond:cEBC} and Condition \ref{cond:Main iEBC} respectively. The assumptions of these criteria consist of two main parts: first we assume the existence of certain bijections (see Assumption \ref{ass:Assumption for the criterion, central} (ii) and Assumption \ref{ass:Assumption for the criterion} (ii)). These bijections have been constructed in Theorem \ref{thm:Bijection for connected center} under suitable restrictions. Secondly, we need to control the action of automorphisms on irreducible characters (see Assumption \ref{ass:Assumption for the criterion, central} (iii)-(iv) and Assumption \ref{ass:Assumption for the criterion} (iii)-(iv)) in order to construct projective representations via an application of \cite[Lemma 2.11]{Spa12}. This second problem is part of an important ongoing project in representation theory of finite groups of Lie type. Moreover, in order to obtain $\G^F$-block isomorphisms of character triples, in the assumption of the criterion for Condition \ref{cond:Main iEBC} we need to include certain block theoretic requirements (see Assumption \ref{ass:Assumption for the criterion} (v)-(vi)). These restrictions are analogous to those introduced in \cite[Theorem 4.1]{Cab-Spa15}, \cite[Theorem 2.4]{Bro-Spa20} and \cite[Theorem 4.5]{Bro-Spa21}.

Throughout this section we consider the following hypothesis. 

\begin{hyp}
\label{hyp:Reduction}
Assume Hypothesis \ref{hyp:Brauer--Lusztig blocks} and suppose that $\G$ is simple of simply connected type such that $\O_\ell(\G^F)\leq \z(\G^F)$.
\end{hyp}

\begin{rmk}
\label{rmk:Reduction}
Observe that under Hypothesis \ref{hyp:Reduction} the prime $\ell$ does not divide $|\z(\G)^F|$ and therefore $\O_\ell(\G^F)=1$. Moreover, by using Remark \ref{rmk:Brauer-Lusztig for simple simply connected} we deduce that Hypothesis \ref{hyp:Reduction} holds whenever $\G$ is simple of simply connected type with $\G^F\neq {^2\mathbf{E}}_6(2)$, $\mathbf{E}_7(2)$, $\mathbf{E}_8(2)$ and such that $\G^F/\z(\G^F)$ is a non-abelian simple group and $\ell\in\Gamma(\G,F)$ with $\ell\geq 5$.
\end{rmk}

The results presented in this section should be compared to \cite[Theorem 2.12]{Spa12}, \cite[Theorem 4.1]{Cab-Spa15}, \cite[Theorem 2.4]{Bro-Spa21}, \cite[Theorem 2.1]{Ruh22} and \cite[Theorem 9.2]{Ruh21}.

\subsection{The criterion for Condition \ref{cond:cEBC}}
\label{sec:The criterion for cEBC}

We start by dealing with Condition \ref{cond:cEBC}. The results obtained in this subsection are then used in the next one to prove the criterion for Condition \ref{cond:Main iEBC} under additional restrictions.

Consider $\G$, $F:\G\to\G$, $q$, $\ell$ and $e$ as in Notation \ref{notation} and let $i:\G\to\wt{\G}$ be a regular embedding. We recall once more that the group $\mathcal{K}$ introduced in Section \ref{sec:Regular embeddings} acts on the sets of irreducible characters of $\wt{\G}^F$ and $\n_{\wt{\G}}(\L)^F$ (see Definition \ref{def:Action by linear characters}).

\begin{ass}
\label{ass:Assumption for the criterion, central}
Let $(\L,\lambda)$ be an $e$-cuspidal pair of $\G$, set
\[\mathcal{G}:=\E\left(\G^F,(\L,\lambda)\right) \hspace{10pt}\text{and}\hspace{10pt}\mathcal{N}:=\irr\left(\n_\G(\L)^F\enspace\middle|\enspace \lambda\right)\]
and consider
\[\wt{\mathcal{G}}:=\irr\left(\wt{\G}^F\enspace\middle|\enspace \mathcal{G}\right)\hspace{10pt}\text{and}\hspace{10pt}\wt{\mathcal{N}}:=\irr\left(\n_{\wt{\G}}(\L)^F\enspace\middle|\enspace \mathcal{N}\right).\]
Assume that:
\begin{enumerate}
\item 
\begin{enumerate}[label=(\alph*)]
\item There is a semidirect decomposition $\wt{\G}^F\rtimes \mathcal{A}$, with $\mathcal{A}$ a finite abelian group, such that
\[\c_{\wt{\G}^F\mathcal{A}}(\G^F)=\z(\wt{\G}^F)\hspace{15pt}\text{and}\hspace{15pt}\wt{\G}^F\mathcal{A}/\z(\wt{\G}^F)\simeq \aut(\G^F);\]
\item Maximal extendibility holds with respect to $\G^F\unlhd \wt{\G}^F$;
\item Maximal extendibility holds with respect to $\n_\G(\L)^F\unlhd \n_{\wt{\G}}(\L)^F$.
\end{enumerate}
\item For $A:=(\wt{\G}^F\mathcal{A})_{(\L,\lambda)}$ there exists a defect preserving $(A\ltimes \mathcal{K})$-equivariant bijection
\[\wt{\Omega}_{(\L,\lambda)}^\G:\wt{\mathcal{G}}\to \wt{\mathcal{N}}\]
such that $\irr\left(\wt{\chi}_{\z(\wt{\G}^F)}\right)=\irr\left(\wt{\Omega}_{(\L,\lambda)}^\G(\wt{\chi})_{\z(\wt{\G}^F)}\right)$ for every $\wt{\chi}\in\wt{\mathcal{G}}$.
\item For every $\wt{\chi}\in\wt{\mathcal{G}}$ there exists $\chi\in\mathcal{G}\cap \irr\left(\wt{\chi}_{\G^F}\right)$ such that:
\begin{enumerate}[label=(\alph*)]
\item $\left(\wt{\G}^F\mathcal{A}\right)_{\chi}=\wt{\G}^F_\chi\mathcal{A}_{\chi}$;
\item $\chi$ extends to $\chi'\in\irr\left(\G^F\mathcal{A}_{\chi}\right)$.
\end{enumerate}
\item For every $\wt{\psi}\in\wt{\mathcal{N}}$ there exists $\psi\in\mathcal{N}\cap \irr\left(\wt{\psi}_{\n_\G(\L)^F}\right)$ such that:
\begin{enumerate}[label=(\alph*)]
\item $\left(\wt{\G}^F \mathcal{A}\right)_{\L,\psi}=\n_{\wt{\G}}(\L)^F_\psi\left(\G^F\mathcal{A}\right)_{\L,\psi}$; 
\item $\psi$ extends to $\psi'\in\irr\left(\left(\G^F \mathcal{A}\right)_{\L,\psi}\right)$.
\end{enumerate}
\end{enumerate}
\end{ass}

Our aim is to show that Assumption \ref{ass:Assumption for the criterion, central} implies Condition \ref{cond:cEBC}. Before giving a proof of this result, we show that Assumption \ref{ass:Assumption for the criterion, central} (iii.a) and Assumption \ref{ass:Assumption for the criterion, central} (iv.a) are equivalent in the presence of an equivariant bijection $\Omega^\G_{(\L,\lambda)}:\mathcal{G}\to\mathcal{N}$.

\begin{lem}
\label{lem:Local star condition follows from Global star condition, central}
Assume Hypothesis \ref{hyp:Reduction}. Let $(\L,\lambda)$ be an $e$-cuspidal pair of $\G$ and suppose that there exists a $(\wt{\G}^F\mathcal{A})_{(\L,\lambda)}$-equivariant bijection
\[\Omega_{(\L,\lambda)}^\G:\E(\G^F,(\L,\lambda))\to\irr\left(\n_\G(\L)^F\enspace\middle|\enspace \lambda\right).\]
If $\chi\in\E(\G^F,(\L,\lambda))$ and $\psi:=\Omega^\G_{(\L,\lambda)}(\chi)$, then
\begin{equation}
\label{eq:Global star condition, central}
\left(\wt{\G}^F\mathcal{A}\right)_{\chi}=\wt{\G}^F_\chi\mathcal{A}_{\chi}
\end{equation}
if and only if
\begin{equation}
\label{eq:Local star condition, central}
\left(\wt{\G}^F \mathcal{A}\right)_{\L,\psi}=\n_{\wt{\G}}(\L)^F_\psi\left(\G^F\mathcal{A}\right)_{\L,\psi}.
\end{equation}
\end{lem}

\begin{proof}
As the two implications can be shown by similar arguments, we only show that \eqref{eq:Global star condition, central} implies \eqref{eq:Local star condition, central}. To start, consider the subgroups
\[T:=\n_\G(\L)^F\left(\wt{\G}^F_{(\L,\lambda),\chi}\cdot (\G^F\mathcal{A})_{(\L,\lambda),\chi}\right)=\n_\G(\L)^F\left(\wt{\G}^F_{(\L,\lambda),\psi}\cdot (\G^F\mathcal{A})_{(\L,\lambda),\psi}\right)\]
and
\[V:=\n_\G(\L)^F(\wt{\G}^F \mathcal{A})_{(\L,\lambda),\chi}=\n_\G(\L)^F(\wt{\G}^F \mathcal{A})_{(\L,\lambda),\psi},\]
where the equalities follow since $\Omega_{(\L,\lambda)}^\G$ is equivariant by assumption.

Define $U(\chi):=(\wt{\G}^F\mathcal{A})_{\L,\chi}$ and $U(\psi):=(\wt{\G}^F\mathcal{A})_{\L,\psi}$. We claim that $U(\chi)=U(\psi)$. To prove this fact, notice that it is enough to show that $U(\chi)$ and $U(\psi)$ are contained in $V$, in fact this would imply $U(\chi)=U(\chi)\cap V=(\wt{\G}^F\mathcal{A})_{\L}\cap V=U(\psi)\cap V=U(\psi)$. If $x\in U(\chi)$, then $\chi\in\E(\G^F,(\L,\lambda))\cap \E(\G^F,(\L,\lambda)^x)$ and, by \cite[Proposition 4.10]{Ros-Generalized_HC_theory_for_Dade}, there exists $y\in\G^F$ such that $(\L,\lambda)=(\L,\lambda)^{xy}$. Notice that $y\in\n_\G(\L)^F$ and hence $x\in V$. On the other hand, if $x\in U(\psi)$, then $\psi$ lies over $\lambda^x$ and by Clifford's theorem $\lambda^{xy}=\lambda$, for some $y\in \n_\G(\L)^F$. Also in this case $x\in V$. Now $U(\chi)=U(\psi)$ and we denote this group by $U$.

Next, we claim that $T=U$. If this is true, then we deduce that $T\leq \n_{\wt{\G}}(\L)^F_\psi(\G^F\mathcal{A})_{\L,\psi}\leq U=T$ and therefore \eqref{eq:Local star condition, central} holds. First, observe that $T\leq U$. As $T\cap \G^F=\n_\G(\L)^F=U\cap \G^F$ and $T\leq U\leq (\wt{\G}^F\mathcal{A})_{\chi}$, it is enough to show that $T\G^F=(\wt{\G}^F\mathcal{A})_{\chi}$. First, repeating the same argument as before, a Frattini argument shows that
\begin{equation}
\label{eq:Local star condition follows from Global star condition 1, central}
\left(\wt{\G}^F\mathcal{A}\right)_{\chi}=\G^F\left(\wt{\G}^F\mathcal{A}\right)_{(\L,\lambda),\chi}
\end{equation}
and
\begin{equation}
\label{eq:Local star condition follows from Global star condition 2, central}
\wt{\G}^F_\chi=\G^F\wt{\G}^F_{(\L,\lambda),\chi}.
\end{equation}
Then using the hypothesis we finally obtain
\begin{align*}
\left(\wt{\G}^F\mathcal{A}\right)_{\chi}&\stackrel{\eqref{eq:Local star condition follows from Global star condition 1, central}}{=}\G^F\left(\wt{\G}^F\mathcal{A}\right)_{(\L,\lambda),\chi}
\\
&\stackrel{\eqref{eq:Global star condition, central}}{=}\G^F\left(\wt{\G}_\chi^F\left(\G^F\mathcal{A}\right)_{\chi}\right)_{(\L,\lambda)}
\\
&\stackrel{\eqref{eq:Local star condition follows from Global star condition 2, central}}{=}\G^F\left(\G^F\wt{\G}^F_{(\L,\lambda),\chi}\left(\G^F\mathcal{A}\right)_{\chi}\right)_{(\L,\lambda)}
\\
&=\G^F\wt{\G}^F_{(\L,\lambda),\chi}\left(\G^F\mathcal{A}\right)_{(\L,\lambda),\chi}
\\
&=\G^FT.
\end{align*}
This concludes the proof.
\end{proof}

We are now ready to prove the criterion for Condition \ref{cond:cEBC}. It should be clear from the proof of this result that, by using Lemma \ref{lem:Local star condition follows from Global star condition, central}, only one amongst Assumption \ref{ass:Assumption for the criterion, central} (iii.a) and Assumption \ref{ass:Assumption for the criterion, central} (iv.a) is actually necessary. In fact, the equivariant map required in Lemma \ref{lem:Local star condition follows from Global star condition, central} is constructed in the following proof independently form the choices of characters satisfying Assumption \ref{ass:Assumption for the criterion, central} (iii.a) and Assumption \ref{ass:Assumption for the criterion, central} (iv.a).

\begin{theo}
\label{thm:Criterion, central}
Assume Hypothesis \ref{hyp:Reduction} and Assumption \ref{ass:Assumption for the criterion, central} with respect to an $e$-cuspidal pair $(\L,\lambda)$ of $\G$. Then Condition \ref{cond:cEBC} holds for $(\L,\lambda)$ and $\G$.
\end{theo}

\begin{proof}
We start by fixing an $\left(A\ltimes \mathcal{K}\right)$-transversal $\wt{\mathbb{T}}_{\rm glo}$ in $\wt{\mathcal{G}}$. As $\wt{\Omega}_{(\L,\lambda)}^\G$ is $(A\ltimes\mathcal{K})$-equivariant, we deduce that the set $\wt{\mathbb{T}}_{\rm loc}:=\{\wt{\Omega}_{(\L,\lambda)}^\G(\wt{\chi})\mid \wt{\chi}\in\wt{\mathbb{T}}_{\rm glo}\}$ is an $(A\ltimes \mathcal{K})$-transversal in $\wt{\mathcal{N}}$. For every $\wt{\chi}\in\wt{\mathbb{T}}_{\rm glo}$, we chose a character $\chi\in\mathcal{G}\cap\irr(\wt{\chi}_{\G^F})$ satisfying Assumption \ref{ass:Assumption for the criterion, central} (iii). Denote by $\mathbb{T}_{\rm glo}$ the set of such characters $\chi$, where $\wt{\chi}$ runs over $\wt{\mathbb{T}}_{\rm glo}$. Similarly, for every $\wt{\psi}\in \wt{\mathbb{T}}_{\rm loc}$, fix a character $\psi\in\mathcal{N}\cap \irr(\wt{\psi}_{\n_\G(\L)^F})$ satisfying Assumption \ref{ass:Assumption for the criterion, central} (iv) and denote by $\mathbb{T}_{\rm loc}$ the set of such characters $\psi$. Observe that $\mathbb{T}_{\rm glo}$ (resp. $\mathbb{T}_{\rm loc}$) is an $A$-transversal in $\mathcal{G}$ (resp. in $\mathcal{N}$). In fact, if $\chi\in\mathcal{G}$, then let $\wt{\chi}\in\wt{\mathcal{G}}$ lying over $\chi$. Then there exists $\wt{\chi}_0\in\wt{\mathbb{T}}_{\rm glo}$ such that $\wt{\chi}_0=\wt{\chi}^{xz}$, for some $x\in A$ and $z\in\mathcal{K}$. Let $\chi_0\in\mathbb{T}_{\rm glo}$ correspond to $\wt{\chi}_0$ and observe that $\chi^x$ and $\chi_0$ lie under $\wt{\chi}_0$. By Clifford's theorem there exists $y\in\wt{\G}^F$ such that $\chi^{xy}=\chi_0$. Now $\chi\in\E(\G^F,(\L,\lambda))\cap \E(\G^F,(\L,\lambda)^{xy})$ and, by \cite[Proposition 4.10]{Ros-Generalized_HC_theory_for_Dade}, there exists $u\in\G^F$ such that $(\L,\lambda)=(\L,\lambda)^{xyu}$. Set $v:=xyu$ and notice that $\chi_0=\chi^v$ and that $v\in A$. Moreover, if $\chi_0=\chi^x$ for some $\chi,\chi_0\in\mathbb{T}_{\rm glo}$ and $x\in A$, then $\wt{\chi}_0$ and $\wt{\chi}^x$ lie over $\chi$, where $\wt{\chi}_0$ (resp. $\wt{\chi}$) is the element of $\wt{\mathbb{T}}_{\rm glo}$ corresponding to $\chi_0$ (resp. $\chi$). Therefore $\wt{\chi}_0=\wt{\chi}^{xz}$, for some $z\in\mathcal{K}$, which implies $\wt{\chi}_0=\wt{\chi}$ and so $\chi_0=\chi$. This shows that $\mathbb{T}_{\rm glo}$ is an $A$-transversal in $\mathcal{G}$. This argument also shows that, for every $\wt{\chi}\in\wt{\mathbb{T}}_{\rm glo}$, there exists a unique character $\chi\in\mathbb{T}_{\rm glo}\cap\irr(\wt{\chi}_{\G^F})$ and that, for every $\chi\in\mathbb{T}_{\rm glo}$, there exists a unique $\wt{\chi}\in\wt{\mathbb{T}}_{\rm glo}\cap \irr(\chi^{\wt{\G}^F})$. A similar argument shows that $\mathbb{T}_{\rm loc}$ is a transversal in $\mathcal{N}$ and that the correspondence between $\psi$ and $\wt{\psi}$ defines a bijection between $\mathbb{T}_{\rm loc}$ and $\wt{\mathbb{T}}_{\rm loc}$.

Now, setting
\[\Omega_{(\L,\lambda)}^\G\left(\chi^x\right):=\psi^x\]
for every $x\in A$ and $\chi\in\mathbb{T}_{\rm glo}$, where $\psi$ is the unique character in $\mathbb{T}_{\rm loc}$ lying below $\wt{\psi}:=\wt{\Omega}_{(\L,\lambda)}^\G(\wt{\chi})$ and $\wt{\chi}$ is the unique character in $\wt{\mathbb{T}}_{\rm glo}$ lying over $\chi$, defines an $A$-equivariant bijection between $\mathcal{G}$ and $\mathcal{N}$. By Assumption \ref{ass:Assumption for the criterion, central} (i.a) this means that $\Omega_{(\L,\lambda)}^\G$ is $\aut_\mathbb{F}(\G^F)_{(\L,\lambda)}$-equivariant.

To show that $\Omega_{(\L,\lambda)}^\G$ preserves the defect, we use Assumption \ref{ass:Assumption for the criterion, central} (i.b) and (i.c). Clearly it's enough to show that $d(\chi)=d(\psi)$, for $\chi\in\mathbb{T}_{\rm glo}$ and $\psi:=\Omega_{(\L,\lambda)}^\G(\chi)\in\mathbb{T}_{\rm loc}$. Let $\wt{\chi}$ (resp. $\wt{\psi}$) be the unique element of $\wt{\mathbb{T}}_{\rm glo}$ (resp. $\wt{\mathbb{T}}_{\rm loc}$) lying over $\chi$ (resp. $\psi$). Then $\wt{\Omega}_{(\L,\lambda)}^\G(\wt{\chi})=\wt{\psi}$ and $d(\wt{\chi})=d(\wt{\psi})$ by Assumption \ref{ass:Assumption for the criterion, central} (ii). Moreover, since $\wt{\G}^F/\G^F\simeq \n_{\wt{\G}}(\L)^F/\n_\G(\L)^F$ is abelian and using Assumption \ref{ass:Assumption for the criterion, central} (i.b) and (i.c), we deduce that the Clifford correspondent $\wh{\chi}\in\irr(\wt{\G}^F_\chi)$ of $\wt{\chi}$ over $\chi$ is an extension of $\chi$ and, similarly, that the Clifford correspondent $\wh{\psi}\in\irr(\n_{\wt{\G}}(\L)^F_\psi)$ of $\wt{\psi}$ over $\psi$ is an extension of $\psi$. As a consequence
\[\ell^{d(\chi)}=\ell^{d(\wh{\chi})}\cdot \left|\wt{\G}^F_\chi:\G^F\right|_\ell\]
and
\[\ell^{d(\psi)}=\ell^{d(\wh{\psi})}\cdot \left|\n_{\wt{\G}}(\L)^F_\psi:\n_\G(\L)^F\right|_\ell.\]
Therefore, as the defect is preserved by induction of characters, we obtain $d(\wh{\chi})=d(\wt{\chi})=d(\wt{\psi})=d(\wh{\psi})$ and it remains to show that $|\wt{\G}^F_\chi:\G^F|_\ell=|\n_{\wt{\G}}(\L)^F_\psi:\n_\G(\L)^F|_\ell$. This follows from the proof of Lemma \ref{lem:Local star condition follows from Global star condition, central}: in fact there it is shown that $\n_{\wt{\G}}(\L)^F_\psi=\n_{\wt{\G}}(\L)^F_\chi$ and therefore $\wt{\G}^F_\chi/\G^F\simeq \n_{\wt{\G}}(\L)^F_\chi/\n_\G(\L)^F=\n_{\wt{\G}}(\L)^F_\psi/\n_\G(\L)^F$.

Next, we prove the condition on character triples. Applying a simplified version of \cite[Theorem 5.3]{Spa17} adapted to $\G^F$-central isomorphic character triples (this immediately follows by the argument used in the proof of \cite[Theorem 5.3]{Spa17}), it is enough to show that
\begin{equation}
\label{eq:Criterion 1, central}
\left((\wt{\G}^F\mathcal{A})_{\chi},\G^F,\chi\right)\isoc{\G^F}\left((\wt{\G}^F\mathcal{A})_{\L,\psi},\n_\G(\L)^F,\Omega_{(\L,\lambda)}^\G(\chi)\right).
\end{equation}
Moreover, as the equivalence relation $\isoc{\G^F}$ is compatible with conjugation, it's enough to prove this condition for a fixed $\chi\in\mathbb{T}_{\rm glo}$ and $\psi:=\Omega_{(\L,\lambda)}^\G(\chi)\in\mathbb{T}_{\rm loc}$.

First of all, notice that the required group theoretical properties are satisfied by the proof of Lemma \ref{lem:Local star condition follows from Global star condition, central}. In fact, there we have shown that $(\wt{\G}^F\mathcal{A})_{\L,\chi}=(\wt{\G}^F\mathcal{A})_{\L,\psi}$ and that $(\wt{\G}^F\mathcal{A})_{\chi}=\G^F(\wt{\G}^F\mathcal{A})_{\L,\chi}$, while
\begin{align*}
\c_{\left(\wt{\G}^F\mathcal{A}\right)_{\chi}}\left(\G^F\right)\leq\c_{\left(\wt{\G}^F\mathcal{A}\right)_{\chi}}\left(\L^F\right)\leq \left(\wt{\G}^F\mathcal{A}\right)_{\L,\chi}= \left(\wt{\G}^F\mathcal{A}\right)_{\L,\psi}.
\end{align*}

To construct the relevant projective representations, we make use of \cite[Lemma 2.11]{Spa12}. As before, consider the corresponding $\wt{\chi}\in\wt{\mathbb{T}}_{\rm glo}$ and $\wt{\psi}\in\wt{\mathbb{T}}_{\rm loc}$ with $\wt{\Omega}_{(\L,\lambda)}^\G(\wt{\chi})=\wt{\psi}$, $\wt{\chi}$ lying over $\chi$ and $\wt{\psi}$ lying over $\psi$. Furthermore, consider the Clifford correspondent $\wh{\chi}\in\irr(\wt{\G}^F_\chi\mid \chi)$ of $\wt{\chi}$ and the Clifford correspondent $\wh{\psi}\in\irr(\n_{\wt{\G}}(\L)^F_\psi\mid\psi)$ of $\wt{\psi}$. Let $\wh{\mathcal{D}}_{\rm glo}$ be a representation affording $\wh{\chi}$ and notice that, by the choice of $\chi$ and using Assumption \ref{ass:Assumption for the criterion, central} (iii.b), there exists a representation $\mathcal{D}'_{\rm glo}$ affording an extension $\chi'\in\irr(\G^F\mathcal{A}_{\chi})$ of $\chi$. Similarly, let $\wh{\mathcal{D}}_{\rm loc}$ be a representation affording $\wh{\psi}$ and observe that, by the choice of $\psi$, there is a representation $\mathcal{D}'_{\rm loc}$ affording an extension $\psi'\in\irr((\G^F\mathcal{A})_{\L,\psi})$ of $\psi$. Applying \cite[Lemma 2.11]{Spa12} with $L:=\G^F$, $\wt{L}:=\wt{\G}^F_\chi$, $C:=\G^F\mathcal{A}_{\chi}$, $X:=(\wt{\G}^F\mathcal{A})_{\chi}$ and recalling that $X=\wt{L}C$ because Assumption \ref{ass:Assumption for the criterion, central} (iii.a) holds for $\chi$, we deduce that the map
\[\mathcal{P}_{\rm glo}:\left(\wt{\G}^F\mathcal{A}\right)_{\chi}\to\GL_{\chi(1)}(\mathbb{C})\]
given by $\mathcal{P}_{\rm glo}(x_1x_2):=\wh{\mathcal{D}}_{\rm glo}(x_1)\mathcal{D}'_{\rm glo}(x_2)$, for every $x_1\in \wt{\G}^F_\chi$ and $x_2\in\G^F\mathcal{A}_{\chi}$, is a projective representation associated with $\chi$ whose factor set $\alpha_{\rm glo}$ satisfies
\begin{equation}
\label{eq:Criterion 2, central}
\alpha_{\rm glo}(x_1x_2,y_1y_2)=\mu_{x_2}^{\rm glo}(y_1)
\end{equation}
for every $x_1,y_1\in\wt{\G}^F_\chi$ and $x_2,y_2\in\G^F\mathcal{A}_{\chi}$, where $\mu_{x_2}^{\rm glo}\in\irr(\wt{\G}^F_\chi/\G^F)$ is determined by the equality $\wh{\chi}=\mu_{x_2}^{\rm glo}\wh{\chi}^{x_2}$ via Gallagher's theorem. In a similar way, considering $L:=\n_\G(\L)^F$, $\wt{L}:=\n_{\wt{\G}}(\L)^F_\chi$, $C:=(\G^F\mathcal{A})_{\L,\chi}$, $X:=(\wt{\G}^F\mathcal{A})_{\L,\chi}$ and noticing that $X=\wt{L}C$ because Assumption \ref{ass:Assumption for the criterion, central} (iv.a) holds for $\psi$, we deduce that the map
\[\mathcal{P}_{\rm loc}:\left(\wt{\G}^F\mathcal{A}\right)_{\L,\chi}\to\GL_{\psi(1)}(\mathbb{C})\]
given by $\mathcal{P}_{\rm loc}(x_1x_2):=\wh{\mathcal{D}}_{\rm loc}(x_1)\mathcal{D}'_{\rm loc}(x_2)$, for every $x_1\in \n_{\wt{\G}}(\L)^F_\chi$ and $x_2\in(\G^F\mathcal{A})_{\L,\chi}$, is a projective representation associated with $\psi$ whose factor set $\alpha_{\rm loc}$ satisfies
\begin{equation}
\label{eq:Criterion 3, central}
\alpha_{\rm loc}(x_1x_2,Zy_1y_2)=\mu_{x_2}^{\rm loc}(y_1)
\end{equation}
for every $x_1,y_1\in\n_{\wt{\G}}(\L)^F_\chi$ and $x_2,y_2\in(\G^F\mathcal{A})_{\L,\chi}$, where $\mu_{x_2}^{\rm loc}\in\irr(\n_{\wt{\G}}(\L)^F_\chi/\n_\G(\L)^F)$ is determined by $\wh{\psi}=\mu_{x_2}^{\rm loc}\wh{\psi}^{x_2}$. In order to obtain the condition on factor sets required to prove \eqref{eq:Criterion 1, central} we have to show that the restriction of $\alpha_{\rm glo}$ to $(\G^F\mathcal{A})_{\L,\chi}\times (\G^F\mathcal{A})_{\L,\chi}$ coincides with $\alpha_{\rm loc}$. Using \eqref{eq:Criterion 2, central} and \eqref{eq:Criterion 3, central}, it is enough to show that
\[\left(\mu_x^{\rm glo}\right)_{\n_{\wt{\G}}(\L)^F_\chi}=\mu_x^{\rm loc}\]
for every $x\in(\G^F\mathcal{A})_{\L,\chi}$ and where $\wh{\chi}=\mu_x^{\rm glo}\wh{\chi}^x$ and $\wh{\psi}=\mu_x^{\rm loc}\wh{\psi}^x$. To prove this equality, since $(\G^F\mathcal{A})_{\L,\chi}=\n_\G(\L)^F A_\chi$ (see the proof of Lemma \ref{lem:Local star condition follows from Global star condition, central}), we may assume $x\in A_\chi$. Then, we conclude since $\wt{\Omega}_{(\L,\lambda)}^\G$ is $(A\ltimes \mathcal{K})$-equivariant.

To conclude we need to check one of the equivalent conditions of \cite[Lemma 3.4]{Spa17}. Recalling that $\c_{(\wt{\G}^F\mathcal{A})_{\chi}}(\G^F)=\z(\wt{\G}^F)$ by Assumption \ref{ass:Assumption for the criterion, central} (i.a), if $\zeta_{\rm glo}$ and $\zeta_{\rm loc}$ are the scalar functions of $\mathcal{P}_{\rm glo}$ and $\mathcal{P}_{\rm loc}$ respectively, we have to show that $\zeta_{\rm glo}$ and $\zeta_{\rm loc}$ coincide as characters of $\z(\wt{\G}^F)$. By the definition of $\mathcal{P}_{\rm glo}$, it follows that $\zeta_{\rm glo}$ coincides with the unique irreducible constituent $\nu$ of $\wh{\chi}_{\z(\wt{\G}^F)}$. Moreover, by Clifford theory we know that $\nu$ is also the unique irreducible constituent of $\wt{\chi}_{\z(\wt{\G}^F)}$. Therefore, we conclude that $\{\zeta_{\rm glo}\}=\irr(\wt{\chi}_{\z(\wt{\G}^F)})$ and a similar argument shows that $\{\zeta_{\rm loc}\}=\irr(\wt{\psi}_{\z(\wt{\G}^F)})$. Then, Assumption \ref{ass:Assumption for the criterion, central} (ii) implies that $\zeta_{\rm glo}=\zeta_{\rm loc}$. This completes the proof.
\end{proof}

\subsection{The criterion for Condition \ref{cond:Main iEBC}}
\label{sec:The criterion for iEBC}

We now prove a criterion for Condition \ref{cond:Main iEBC}. To do so, we sharpen the argument used in the proof of Theorem \ref{thm:Criterion, central}. As mentioned at the beginning of this section, some additional restrictions are required in order to deal with the necessary block theoretic demands. Recall that whenever $\ell\in\Gamma(\G,F)$, $\L$ is an $e$-split Levi subgroup of $\G$ and $\lambda\in\irr(\L^F)$, the induced block $\bl(\lambda)^{\G^F}$ is defined (see the comment preceding \cite[Proposition 4.9]{Ros-Generalized_HC_theory_for_Dade}).

\begin{ass}
\label{ass:Assumption for the criterion}
Let $(\L,\lambda)$ be an $e$-cuspidal pair of $\G$ and suppose that $B:=\bl(\lambda)^{\G^F}$ is defined. Set
\[\mathcal{G}:=\E\left(\G^F,(\L,\lambda)\right) \hspace{10pt}\text{and}\hspace{10pt}\mathcal{N}:=\irr\left(\n_\G(\L)^F\enspace\middle|\enspace \lambda\right)\]
and consider
\[\wt{\mathcal{G}}:=\irr\left(\wt{\G}^F\enspace\middle|\enspace \mathcal{G}\right)\hspace{10pt}\text{and}\hspace{10pt}\wt{\mathcal{N}}:=\irr\left(\n_{\wt{\G}}(\L)^F\enspace\middle|\enspace \mathcal{N}\right).\]
Assume that:
\begin{enumerate}
\item 
\begin{enumerate}[label=(\alph*)]
\item There is a semidirect decomposition $\wt{\G}^F\rtimes \mathcal{A}$, with $\mathcal{A}$ a finite abelian group, such that
\[\c_{(\wt{\G}^F\mathcal{A})_Z/Z}(\G^F/Z)=\z(\wt{\G}^F)/Z\hspace{15pt}\text{and} \hspace{15pt}(\wt{\G}^F\mathcal{A})_Z/\z(\wt{\G}^F)\simeq \aut(\G^F/Z)\]
for every $Z\leq \z(\G^F)$ (see \cite[Lemma 2.1]{Ros-Generalized_HC_theory_for_Dade});
\item Maximal extendibility holds with respect to $\G^F\unlhd \wt{\G}^F$;
\item Maximal extendibility holds with respect to $\n_\G(\L)^F\unlhd \n_{\wt{\G}}(\L)^F$.
\end{enumerate}
\item For $A:=(\wt{\G}^F\mathcal{A})_{(\L,\lambda)}$ there exists a defect preserving $(A\ltimes \mathcal{K})$-equivariant bijection
\[\wt{\Omega}_{(\L,\lambda)}^\G:\wt{\mathcal{G}}\to \wt{\mathcal{N}}\]
such that, for every $\wt{\chi}\in\wt{\mathcal{G}}$, the following conditions hold:
\begin{enumerate}[label=(\alph*)]
\item $\irr\left(\wt{\chi}_{\z(\wt{\G}^F)}\right)=\irr\left(\wt{\Omega}_{(\L,\lambda)}^\G(\wt{\chi})_{\z(\wt{\G}^F)}\right)$;
\item $\bl\left(\wt{\chi}\right)=\bl\left(\wt{\Omega}_{(\L,\lambda)}^\G\left(\wt{\chi}\right)\right)^{\wt{\G}^F}$.
\end{enumerate}
\item For every $\wt{\chi}\in\wt{\mathcal{G}}$ there exists $\chi\in\irr\left(\wt{\chi}_{\G^F}\right)$ such that:
\begin{enumerate}[label=(\alph*)]
\item $\left(\wt{\G}^F\mathcal{A}\right)_\chi=\wt{\G}^F_\chi \mathcal{A}_\chi$;
\item $\chi$ extends to $\chi'\in\irr\left(\G^F\mathcal{A}_\chi\right)$.
\end{enumerate}
\item For every $\wt{\psi}\in\wt{\mathcal{N}}$ there exists $\psi\in\mathcal{N}\cap \irr\left(\wt{\psi}_{\n_\G(\L)^F}\right)$ such that:
\begin{enumerate}[label=(\alph*)]
\item $\left(\wt{\G}^F \mathcal{A}\right)_{\L,\psi}=\n_{\wt{\G}}(\L)^F_\psi\left(\G^F\mathcal{A}\right)_{\L,\psi}$; 
\item $\psi$ extends to $\psi'\in\irr\left(\left(\G^F \mathcal{A}\right)_{\L,\psi}\right)$.
\end{enumerate}
\item Assume one of the following conditions:
\begin{enumerate}[label=(\alph*)]
\item ${\rm Out}(\G^F)_\mathcal{B}$ is abelian, where $\mathcal{B}$ is the $\wt{\G}^F$-orbit of $B$. In particular (iii) holds for every $\wt{\G}^F$-conjugate of $\chi$ (see the proof of \cite[Lemma 4.7]{Bro-Spa21}); or
\item for every subgroup $\G^F\leq J\leq\wt{\G}^F$ we have that every block $C\in\Bl(J\mid B)$ is $\wt{\G}^F$-invariant.
\end{enumerate}
\item The pair $(\L,\lambda)$ is $e$-Brauer--Lusztig-cuspidal in the sense of Definition \ref{def:e-BL-cuspidality}.
\end{enumerate}
\end{ass}

\begin{rmk}
\label{rmk:Condition on block stability}
Here we comment on Assumption \ref{ass:Assumption for the criterion}. First, observe that (v.a) holds for every block of $\G^F$ whenever $\G$ is a simple algebraic group not of type $\mathbf{A}$, $\mathbf{D}$ or $\mathbf{E}_6$. Next, notice that condition (v.b) holds for blocks of maximal defect (see \cite[Proposition 5.4]{Cab-Spa15} and observe that the proof of this result holds in general in our situation by \cite[Proposition 13.19]{Cab-Eng04}) and for unipotent blocks: if $B$ is a unipotent block of $\G^F$, then there exists a unipotent character $\chi\in\irr(B)$. By \cite[Proposition 13.20]{Dig-Mic91} we deduce that $\chi$ extends to a character $\wt{\chi}\in\irr(\wt{\G}^F)$. If $\G^F\leq J\leq \wt{\G}^F$ and $C$ is a block of $J$ that covers $B$, then we can find a character $\psi\in\irr(C)$ that lies above $\chi$. Since $\wt{\chi}_J$ is an irreducible character of $J$ lying above $\chi$, we deduce that $\psi=\wt{\chi}_J\wh{z}_J$ for some $z\in\mathcal{K}$ corresponding to $\wh{z}_{\wt{\G}}\in\irr(\wt{\G}^F/\G^F)$ and where $\wh{z}_J$ is the restriction of $\wh{z}_{\wt{\G}}$ to $J$. Then $\psi$ is $\wt{\G}^F$-invariant and therefore $C$ is $\wt{\G}^F$-invariant. This proves that (v.b) holds for unipotent blocks.

Next, we point out that the character $\chi$ from Assumption \ref{ass:Assumption for the criterion} (iii) is not required to lie in $\mathcal{G}$. In fact, if such a character $\chi$ exists, then a character with the same properties and lying in $\mathcal{G}$ can always be found under Assumption \ref{ass:Assumption for the criterion} (v)-(vi). To see this, fix $\wt{\chi}\in\wt{\mathcal{G}}$ and $\chi\in\irr(\wt{\chi}_{\G^F})$ satisfying Assumption \ref{ass:Assumption for the criterion} (iii). By the definition of $\wt{\mathcal{G}}$ there exists $\chi_0\in\irr(\wt{\chi}_{\G^F})\cap \mathcal{G}$. In particular $\chi$ and $\chi_0$ are $\wt{\G}^F$-conjugate. Now, if (v.a) holds, then all $\wt{\G}^F$-conjugates of $\chi$ satisfy Assumption \ref{ass:Assumption for the criterion} (iii.a) and (iii.b) according to the proof of \cite[Lemma 4.7]{Bro-Spa21}. Then $\chi_0$ is the character we were looking for. If (v.b) holds, then $B$ is $\wt{\G}^F$-invariant and, since $\bl(\chi_0)=B$, we deduce that $\bl(\chi)=B$. On the other hand, if $s$ is a semisimple element of $\L^{*F^*}$ such that $\lambda\in\E(\L^F,[s])$, then $\chi_0\in\mathcal{G}\subseteq \E(\G^F,[s])$ by \cite[Proposition 15.7]{Cab-Eng04}. Thus $\chi\in\E(\G^F,[s])$ by \cite[Proposition 15.6]{Cab-Eng04} and we conclude that $\chi\in\irr(B)\cap \E(\G^F,[s])=\mathcal{G}$ by applying Assumption \ref{ass:Assumption for the criterion} (vi).

Finally, as discussed in Section \ref{subsec:e-HC theory}, it is expected that Assumption \ref{ass:Assumption for the criterion} (vi) holds for every $e$-cuspidal pair under suitable assumptions on the prime $\ell$.
\end{rmk}

We now prove the criterion for Condition \ref{cond:Main iEBC}. Our argument makes use of the notion of Dade's \emph{ramification group}. For every block $b$ of a normal subgroup $N$ of $G$, Dade introduced a normal subgroup $G[b]$ of the subgroup $G_b$ such that $G[b]\leq G_\chi$ for every $\chi\in\irr(b)$. Here we use the following equivalent definition given by Murai in \cite{Mur13} (see also \cite[Definition 3.1]{Cab-Spa15}).

\begin{defin}
For every $N\unlhd G$ and $b\in\Bl(G)$ define
\[G[b]:=\left\lbrace g\in G_b\enspace\middle|\enspace\lambda_{b^{(g)}}\left(\cl_{\langle N,g\rangle}(h)^+\right)\neq 0,\text{ for some }h\in Ng\right\rbrace\]
where $b^{(g)}$ is any block of $\langle N,g\rangle$ covering $b$, $\lambda_{b^{(g)}}$ is the central character associated to $b^{(g)}$ and $\cl_{\langle N,g\rangle}(h)^+$ is the conjugacy class sum of $h$ in $\langle N,g\rangle$. It can be shown that this definition does not depend on the choices of the blocks $b^{(g)}$ covering $b$. 
\end{defin}

See \cite{Dad73}, \cite{Mur13} and \cite{Kos-Spa15} for further details on ramification groups.

Before proving the criterion for Condition \ref{cond:Main iEBC}, we need the following result in which we show how to choose transversals with good properties.

\begin{prop}
\label{prop:Criterion}
Assume Hypothesis \ref{hyp:Reduction} and Assumption \ref{ass:Assumption for the criterion}. Let $\wt{\mathbb{T}}_{\rm glo}$ be any $(A\ltimes \mathcal{K})$-transversal in $\wt{\mathcal{G}}$ and consider the $(A\ltimes\mathcal{K})$-transversal $\wt{\mathbb{T}}_{\rm loc}:=\{\wt{\Omega}_{(\L,\lambda)}^\G(\wt{\chi})\mid \wt{\chi}\in\wt{\mathbb{T}}_{\rm glo}\}$ in $\wt{\mathcal{N}}$. Then there exist $A$-transversals $\mathbb{T}_{\rm glo}$ in $\mathcal{G}$ and $\mathbb{T}_{\rm loc}$ in $\mathcal{N}$ with the following properties:
\begin{enumerate}
\item Every $\chi\in\mathbb{T}_{\rm glo}$ satisfies Assumption \ref{ass:Assumption for the criterion} (iii.a) and (iii.b);
\item Every $\psi\in\mathbb{T}_{\rm loc}$ satisfies Assumption \ref{ass:Assumption for the criterion} (iv.a) and (iv.b);
\item For every $\chi\in\mathbb{T}_{\rm glo}$ there exists a unique $\wt{\chi}\in\wt{\mathbb{T}}_{\rm glo}$ lying over $\chi$. Conversely $\chi$ is the only character of $\mathbb{T}_{\rm glo}$ lying under $\wt{\chi}$;
\item For every $\psi\in\mathbb{T}_{\rm loc}$ there exists a unique $\wt{\psi}\in\wt{\mathbb{T}}_{\rm loc}$ lying over $\psi$. Conversely $\psi$ is the only character of $\mathbb{T}_{\rm loc}$ lying under $\wt{\psi}$;
\item Let $\chi\in\mathbb{T}_{\rm glo}$ and $\psi\in\mathbb{T}_{\rm loc}$ such that $\wt{\Omega}_{(\L,\lambda)}^\G(\wt{\chi})=\wt{\psi}$, where $\wt{\chi}$ is the unique character of $\wt{\mathbb{T}}_{\rm glo}$ lying above $\chi$ and $\wt{\psi}$ is the unique character ot $\wt{\mathbb{T}}_{\rm loc}$ lying above $\psi$. Then
\[\bl\left(\wh{\chi}_J\right)=\bl\left(\wh{\psi}_{\n_{\wt{\G}}(\L)^F_\chi\cap J}\right)^J\]
for every $\G^F\leq J\leq \wt{\G}^F$, where $\wh{\chi}\in\irr(\wt{\G}^F_\chi)$ is the Clifford correspondent of $\wt{\chi}$ over $\chi$ and $\wh{\psi}\in\irr(\n_{\wt{\G}}(\L)^F_\psi)$ is the CLifford correspondent of $\wt{\psi}$ over $\psi$.
\end{enumerate}
\end{prop}

\begin{proof}
For every $\wt{\psi}\in \wt{\mathbb{T}}_{\rm loc}$ fix a character $\psi\in\mathcal{N}\cap \irr(\wt{\psi}_{\n_\G(\L)^F})$ satisfying Assumption \ref{ass:Assumption for the criterion} (iv) and denote by $\mathbb{T}_{\rm loc}$ the set of such characters $\psi$, while $\wt{\psi}$ runs over $\wt{\mathbb{T}}_{\rm loc}$. As proved in Theorem \ref{thm:Criterion, central}, the set $\mathbb{T}_{\rm loc}$ is an $A$-transversal in $\mathcal{N}$ satisfying (iv) above. Next, for every $\wt{\chi}\in\wt{\mathbb{T}}_{\rm glo}$, we are going to find a character $\chi\in\mathcal{G}\cap \irr(\wt{\chi}_{\G^F})$ satisfying Assumption \ref{ass:Assumption for the criterion} (iii.a) and (iii.b) and such that
\begin{equation}
\label{eq:Criterion 0}
\bl\left(\wh{\chi}_J\right)=\bl\left(\wh{\psi}_{\n_{\wt{\G}}(\L)^F_\chi\cap J}\right)^J
\end{equation}
for every $\G^F\leq J\leq \wt{\G}^F_\chi$ and where $\wh{\chi}\in\irr(\wt{\G}^F_\chi\mid \chi)$ is the Clifford correspondent of $\wt{\chi}$ over $\chi$ and $\wh{\psi}\in\irr(\n_{\wt{\G}}(\L)^F_\psi\mid \psi)$ is the Clifford correspondent of $\wt{\psi}$ over $\psi$ with $\wt{\psi}:=\wt{\Omega}_{(\L,\lambda)}^\G(\wt{\chi})$ and $\psi\in\mathbb{T}_{\rm loc}$ corresponding to $\wt{\psi}$. Then, as shown in the proof of Theorem \ref{thm:Criterion, central}, the set $\mathbb{T}_{\rm glo}$ of such characters $\chi$ while $\wt{\chi}$ runs over $\wt{\mathbb{T}}_{\rm glo}$ will be an $A$-transversal in $\mathcal{G}$ satisfying (iii) above. Moreover (v) will be satisfied by our choice. 

We first prove the claim assuming Assumption \ref{ass:Assumption for the criterion} (v.a). We start by showing that, for every $\wt{\chi}\in\wt{\mathbb{T}}_{\rm glo}$, there exists a character $\chi\in\mathcal{G}\cap \irr(\wt{\chi}_{\G^F})$ such that
\begin{equation}
\label{eq:Criterion 1}
\bl\left(\wh{\chi}_{\wt{\G}^F[B]}\right)=\bl\left(\wh{\psi}_{\n_{\wt{\G}}(\L)^F[C]}\right)^{\wt{\G}^F[B]},
\end{equation}
where $\wh{\chi}\in\irr(\wt{\G}^F_\chi\mid \chi)$ is the Clifford correspondent of $\wt{\chi}$ over $\chi$ and $\wh{\psi}\in\irr(\n_{\wt{\G}}(\L)^F_\psi\mid \psi)$ is the Clifford correspondent of $\wt{\psi}$ over $\psi$ with $\wt{\psi}:=\wt{\Omega}_{(\L,\lambda)}^\G(\wt{\chi})$ and $\psi\in\mathbb{T}_{\rm loc}$ corresponding to $\wt{\psi}$ and $C:=\bl(\psi)$. Notice that, as pointed out in Remark \ref{rmk:Condition on block stability}, under Assumption \ref{ass:Assumption for the criterion} (v.a) such a character $\chi$ automatically satisfies Assumption \ref{ass:Assumption for the criterion} (iii.a) and (iii.b).

Set $b:=\bl(\lambda)$ and recall that, as every block of $\n_\G(\L)^F$ is $\L^F$-regular (see \cite[Lemma 6.5]{Ros-Generalized_HC_theory_for_Dade}), $C$ must coincide with $b^{\n_\G(\L)^F}$ and therefore $C^{\G^F}=b^{\G^F}=B$. Moreover, for $E:=\z^\circ(\L)_\ell^F$, we have $\n_X(\L)=\n_X(E)$ for every $\G^F\leq X\leq \wt{\G}^F$ (see \cite[Proposition 2.6]{Ros-Generalized_HC_theory_for_Dade}). Then, for every $\G^F\leq Y\leq X\leq \wt{\G}^F$, every $C_0\in\Bl(\n_Y(\L))$ and $C_1\in\Bl(\n_X(\L)\mid C_0)$, the induced block $B_1:=C_1^{X}$ is well defined and covers $C_0^Y$ (see \cite[Theorem B]{Kos-Spa15}): in fact for a defect group $D\in\delta(C_0)$ we have $E\leq \O_\ell(\n_Y(\L)^F)\leq D$ and hence $\c_X(D)\leq \n_X(E)=\n_X(\L)$.

Consider $\wt{C}:=\bl(\wt{\psi})$, $\wt{B}:=\bl(\wt{\chi})$ and recall that $\wt{B}=(\wt{C})^{\wt{\G}^F}$ by Assumption \ref{ass:Assumption for the criterion} (ii.b). Notice that $\wt{\G}^F[B]=\n_{\wt{\G}}(\L)^F[C]\cdot \G^F$ (see \cite[Lemma 3.2 (c) and Lemma 3.6]{Kos-Spa15}) and set $C_1:=\bl(\wh{\psi}_{\n_{\wt{\G}}(\L)^F[C]})$ and $B_1:=C_1^{\wt{\G}^F[B]}$. By the previous paragraph (applied with $Y=\G^F$ and $X=\wt{\G}^F[B]$) the block $B_1$ covers $B$ and the exact same argument (applied with $Y=\wt{\G}^F[B]$ and $X=\wt{\G}^F$) can be used to show that $\wt{B}$ covers $B_1$. In particular there exists $\chi_1\in\irr(B_1)$ lying under $\wt{\chi}$. We claim that $\chi_{1,\G^F}$ is irreducible and lies in $\mathcal{G}$. If $\chi$ is an irreducible constituent of $\chi_{1,\G^F}$, then $B_1$ covers $\bl(\chi)$. As $B$ is $\wt{\G}^F[B]$-invariant, we conclude that $\bl(\chi)=B$. Then $\wt{\G}^F[B]\leq \wt{\G}^F_\chi$ and Assumption \ref{ass:Assumption for the criterion} (i.b) implies that $\chi_{1,\G^F}=\chi$. Furthermore, since for every $\G^F\leq J\leq \wt{\G}^F_\chi$ there exists a unique irreducible character of $J$ lying over $\chi$ and under $\wt{\chi}$, we conclude that $\chi_1=\wh{\chi}_{\wt{\G}^F[B]}$, where $\wh{\chi}\in\irr(\wt{\G}^F_\chi)$ is the Clifford correspondent of $\wt{\chi}$ over $\chi$. To conclude, since $\wt{\chi}\in\wt{\mathcal{G}}$ covers $\chi_1$ and hence $\chi$, \cite[Proposition 15.6]{Cab-Eng04} implies that $\chi\in\E(\G^F,[s])$ where $s\in\L^{*F^*}_{\rm ss}$ such that $\lambda\in\E(\L^F,[s])$. By Assumption \ref{ass:Assumption for the criterion} (vi) we conclude that $\chi\in\mathcal{G}\cap \irr(\wt{\chi}_{\G^F})$ and satisfies \eqref{eq:Criterion 1}.

Next, we deduce \eqref{eq:Criterion 0} from \eqref{eq:Criterion 1}. First, since $\bl(\wh{\psi}_{\n_{\wt{\G}}(\L)^F[C]})$ is covered by $\bl(\wh{\psi})$, by the same argument used before (applied with $Y=\wt{\G}^F[B]$ and $X=\wt{\G}^F_\chi$) we deduce that $\bl(\wh{\psi}_{\n_{\wt{\G}}(\L)^F[C]})^{\wt{\G}^F[B]}=\bl(\wh{\chi}_{\wt{\G}^F[B]})$ is covered by $\bl(\wh{\psi})^{\wt{\G}^F_\chi}$. Since $\wt{\G}^F_\chi$ has a unique block that covers $\bl(\wh{\chi}_{\wt{\G}^F[B]})$ (see \cite[Theorem 3.5]{Mur13}), we conclude that $\bl(\wh{\psi})^{\wt{\G}^F_\chi}=\bl(\wh{\chi})$. Finally, for $\G^F\leq J\leq \wt{\G}^F_\chi$, observe that $\bl(\wh{\chi}_J)$ is $\wt{\G}^F_\chi$-stable and therefore it is the unique block of $J$ covered by $\bl(\wh{\chi})$. Since, again by using the previous argument (applied with $Y=J$ and $X=\wt{\G}^F_\chi$), $\bl(\wh{\psi}_{\n_{\wt{\G}}(\L)^F_\chi\cap J})^J$ is covered by $\bl(\wh{\psi})^{\wt{\G}^F_\chi}=\bl(\wh{\chi})$ we conclude that $\chi$ is a character of $\irr(\wt{\chi}_{\G^F})\cap \mathcal{G}$ satisfying Assumption \ref{ass:Assumption for the criterion} (iii.a) and (iii.b) and such that \eqref{eq:Criterion 0} holds. This proves the claim under Assumption \ref{ass:Assumption for the criterion} (v.a).

We now prove the claim under Assumption \ref{ass:Assumption for the criterion} (v.b). Consider $\chi\in\irr(\wt{\chi}_{\G^F})$ satisfying Assumption \ref{ass:Assumption for the criterion} (iii) and notice that, as shown in Remark \ref{rmk:Condition on block stability}, under Assumption \ref{ass:Assumption for the criterion} (v.b) we automatically have $\chi\in\mathcal{G}$. As shown in the previous part, the block $\wh{B}:=\bl(\wh{\psi})^{\wt{\G}^F_\chi}$ is covered by $\wt{B}:=\bl(\wt{\chi})$ and covers $B$. Since $\wt{B}$ covers $\wh{B}$, we deduce that $\wh{B}$ and $\bl(\wh{\chi})$ are $\wt{\G}^F$-conjugate. On the other hand our assumption implies that $\wh{B}$ is $\wt{\G}^F$-stable and therefore coincides with $\bl(\wh{\chi})$. This shows that $\bl(\wh{\chi})=\bl(\wh{\psi})^{\wt{\G}^F_\chi}$ and, arguing as in the final part of the previous paragraph, we conclude that \eqref{eq:Criterion 0} holds. This completes the proof.
\end{proof}

We can finally prove the criterion for Condition \ref{cond:Main iEBC}.

\begin{theo}
\label{thm:Criterion}
Assume Hypothesis \ref{hyp:Reduction} and Assumption \ref{ass:Assumption for the criterion} with respect to the $e$-cuspidal pair $(\L,\lambda)$. Then Condition \ref{cond:Main iEBC} holds for $(\L,\lambda)$ and $\G$.
\end{theo}

\begin{proof}
Choose transversals $\wt{\mathbb{T}}_{\rm glo}$, $\wt{\mathbb{T}}_{\rm loc}$, $\mathbb{T}_{\rm glo}$ and $\mathbb{T}_{\rm loc}$ as in Proposition \ref{prop:Criterion}. As in the proof of Theorem \ref{thm:Criterion, central}, setting
\[\Omega_{(\L,\lambda)}^\G\left(\chi^x\right):=\psi^x\]
for every $x\in A$ and $\chi\in\mathbb{T}_{\rm glo}$, where $\psi$ is the unique character in $\mathbb{T}_{\rm loc}$ lying below $\wt{\psi}:=\wt{\Omega}_{(\L,\lambda)}^\G(\wt{\chi})$ and $\wt{\chi}$ is the unique character in $\wt{\mathbb{T}}_{\rm glo}$ lying over $\chi$, defines an $A$-equivariant bijection between $\mathcal{G}$ and $\mathcal{N}$. By Assumption \ref{ass:Assumption for the criterion} (i.a) this means that $\Omega_{(\L,\lambda)}^\G$ is $\aut(\G^F)_{(\L,\lambda)}$-equivariant.

The argument used in the proof of Theorem \ref{thm:Criterion, central} shows that $\Omega_{(\L,\lambda)}^\G$ is defect preserving and that
\[\ker\left(\chi_{\z\left(\G^F\right)}\right)=\ker\left(\Omega_{(\L,\lambda)}^\G(\chi)_{\z\left(\G^F\right)}\right)\]
for every $\chi\in\mathcal{G}$. By \cite[Theorem 5.3]{Spa17}, we deduce that to conclude the proof it's enough to show that
\begin{equation}
\label{eq:Criterion 2}
\left((\wt{\G}^F\mathcal{A})_\chi/Z,\G^F/Z,\overline{\chi}\right)\iso{\G^F/Z}\left((\wt{\G}^F\mathcal{A})_{\L,\chi}/Z,\n_\G(\L)^F/Z,\overline{\psi}\right),
\end{equation}
for every $\chi\in\mathcal{G}$, $\psi:=\Omega_{(\L,\lambda)}^\G(\chi)$ and where $Z:=\ker(\chi_{\z(\G^F)})$ while $\overline{\chi}$ and $\overline{\psi}$ correspond to $\chi$ and $\psi$ respectively via inflation of characters. Moreover, as the equivalence relation $\iso{\G^F/Z}$ is compatible with conjugation, it is enough to prove \eqref{eq:Criterion 2} for a fixed $\chi\in\mathbb{T}_{\rm glo}$ and $\psi:=\Omega_{(\L,\lambda)}^\G(\chi)\in\mathbb{T}_{\rm loc}$. As before, consider the corresponding $\wt{\chi}\in\wt{\mathbb{T}}_{\rm glo}$ and $\wt{\psi}\in\wt{\mathbb{T}}_{\rm loc}$ with $\wt{\Omega}_{(\L,\lambda)}^\G(\wt{\chi})=\wt{\psi}$, $\wt{\chi}$ lying over $\chi$ and $\wt{\psi}$ lying over $\psi$. Furthermore, consider the Clifford correspondent $\wh{\chi}\in\irr(\wt{\G}^F_\chi\mid \chi)$ of $\wt{\chi}$ and the Clifford correspondent $\wh{\psi}\in\irr(\n_{\wt{\G}}(\L)^F_\psi\mid\psi)$ of $\wt{\psi}$. 

Proceeding as in the proof of Theorem \ref{thm:Criterion, central}, we can construct a projective representation associated with $\overline{\chi}$
\[\overline{\mathcal{P}}_{\rm glo}:\left(\wt{\G}^F\mathcal{A}\right)_\chi/Z\to\GL_{\chi(1)}(\mathbb{C})\]
given by $\overline{\mathcal{P}}_{\rm glo}(Zx_1x_2):=\wh{\mathcal{D}}_{\rm glo}(x_1)\mathcal{D}'_{\rm glo}(x_2)$ for every $x_1\in \wt{\G}^F_\chi$ and $x_2\in(\G^F\mathcal{A})_\chi$. Similarly, we obtain a projective representation associated with $\overline{\psi}$
\[\overline{\mathcal{P}}_{\rm loc}:\left(\wt{\G}^F\mathcal{A}\right)_{\L,\chi}/Z\to\GL_{\psi(1)}(\mathbb{C})\]
given by $\overline{\mathcal{P}}_{\rm loc}(Zx_1x_2):=\wh{\mathcal{D}}_{\rm loc}(x_1)\mathcal{D}'_{\rm loc}(x_2)$ for every $x_1\in \n_{\wt{\G}}(\L)^F_\psi$ and $x_2\in(\G^F\mathcal{A})_{\L,\psi}$. Moreover, by the proof of Theorem \ref{thm:Criterion, central}, we know that
\[\left((\wt{\G}^F\mathcal{A})_\chi/Z,\G^F/Z,\overline{\chi}\right)\isoc{\G^F/Z}\left((\wt{\G}^F\mathcal{A})_{\L,\psi}/Z,\n_\G(\L)^F/Z,\overline{\psi}\right)\]
via the projective representations $(\overline{\mathcal{P}}_{\rm glo},\overline{\mathcal{P}}_{\rm loc})$. Consider the factor sets $\overline{\alpha}_{\rm glo}$ of $\overline{\mathcal{P}}_{\rm glo}$ and $\overline{\alpha}_{\rm loc}$ of $\overline{\mathcal{P}}_{\rm loc}$. Let $S$ be the group generated by the values of $\overline{\alpha}_{\rm glo}$ and denote by $A_{\rm glo}$ the central extension of $(\wt{\G}^F\mathcal{A})_\chi/Z$ by $S$ induced by $\overline{\alpha}_{\rm glo}$. Let $\epsilon:A_{\rm glo}\to(\wt{\G}^F\mathcal{A})_\chi/Z$ be the canonical morphism with kernel $S$. As $\overline{\alpha}_{\rm glo}$ is trivial on $(\wt{\G}^F_\chi/Z)\times (\wt{\G}^F_\chi/Z)$, every subgroup $X\leq \wt{\G}^F_\chi/Z$ is isomorphic to the subgroup $X_0:=\{(x,1)\mid x\in X\}$ of $A_{\rm glo}$ and $\epsilon^{-1}(X)=X_0\times S$. In particular, we have $H_{\rm glo}:=\epsilon^{-1}\left(\wt{\G}^F_\chi/Z\right)=(\wt{\G}^F_\chi/Z)_0\times S$. The map given by
\[\mathcal{Q}_{\rm glo}(x,s):=s\overline{\mathcal{P}}_{\rm glo}(x),\]
for every $s\in S$ and $x\in (\wt{\G}^F\mathcal{A})_\chi/Z$, is an irreducible representation of $A_{\rm glo}$ affording an extension $\chi_1$ of the character $\overline{\chi}_0$ of $(\G^F/Z)_0$ corresponding to $\overline{\chi}$. Notice that
\begin{equation}
\label{eq:Criterion 3}
\chi_{1,H_{\rm glo}}=\left(\overline{\wh{\chi}}\right)_0\times \iota,
\end{equation}
where $\iota(s):=s$ and $(\overline{\wh{\chi}})_0$ is the character of $(\wt{\G}^F_\chi/Z)_0$ corresponding to $\overline{\wh{\chi}}\in\irr(\wt{\G}^F_\chi/Z)$. Next, set $A_{\rm loc}:=\epsilon^{-1}((\wt{\G}^F\mathcal{A})^F_{\L,\chi}/Z)$ and notice that, because the factor set $\overline{\alpha}_{\rm loc}$ of $\overline{\mathcal{P}}_{\rm loc}$ is the restriction of the factor set $\overline{\alpha}_{\rm glo}$ of $\overline{\mathcal{P}}_{\rm glo}$, the map given by
\[\mathcal{Q}_{\rm loc}(x,s):=s\overline{\mathcal{P}}_{\rm loc}(x),\]
for every $s\in S$ and $x\in (\wt{\G}^F\mathcal{A})^F_{\L,\chi}/Z$, is an irreducible representation of $A_{\rm loc}$ affording an extension $\psi_1$ of the character $\overline{\psi}_0$ of $(\n_\G(\L)^F/Z)_0$ corresponding to $\overline{\psi}$. As before, we have
\begin{equation}
\label{eq:Criterion 4}
\psi_{1,H_{\rm loc}}=\left(\overline{\wh{\psi}}\right)_0\times \iota,
\end{equation}
where $H_{\rm loc}:=\epsilon^{-1}(\n_{\wt{\G}}(\L)^F_\chi/Z)=(\n_{\wt{\G}}(\L)^F_\chi/Z)_0\times S$ and $(\overline{\wh{\psi}})_0$ is the character of $(\n_{\wt{\G}}(\L)^F_\chi/Z)_0$ corresponding to $\overline{\wh{\psi}}\in\irr(\n_{\wt{\G}}(\L)^F_\chi/Z)$. Now, \eqref{eq:Criterion 3}, \eqref{eq:Criterion 4} and \eqref{eq:Criterion 0} imply that
\begin{equation}
\label{eq:Criterion 5}
\bl\left(\chi_{1,J}\right)=\bl\left(\psi_{1,J\cap H_{\rm glo}}\right)^J
\end{equation}
for every $(\G^F/Z)_0\leq J\leq H_{\rm glo}$ (see the argument at the end of the proof of \cite[Proposition 4.2]{Cab-Spa15}). By \cite[Theorem C]{Kos-Spa15} there exists $\varphi_1\in\irr(A_{\rm glo}[B_0])$ such that $\varphi_{1,(\G^F/Z)_0}$ is irreducible and lies in the block $B_0$ and
\begin{equation}
\label{eq:Criterion 6}
\bl\left(\varphi_{1,J}\right)=\bl\left(\psi_{1,J\cap A_{\rm loc}}\right)^J
\end{equation}
for every $(\G^F/Z)_0\leq J\leq A_{\rm glo}[B_0]$. It follows from \eqref{eq:Criterion 5} and \eqref{eq:Criterion 6} that
\[\bl\left(\varphi_{1,J}\right)=\bl\left(\psi_{1,J\cap H_{\rm loc}}\right)^J=\bl\left(\chi_{1,J}\right)\]
for every $(\G^F/Z)_0\leq J\leq H_{\rm glo}[B_0]=H_{\rm glo}\cap A_{\rm glo}[B_0]$. In particular $B_0=\bl(\chi_{1,(\G^F/Z)_0})=\bl(\overline{\chi}_0)$. Therefore the conditions of \cite[Lemma 3.2]{Cab-Spa15} are satisfied and we obtain an extension $\chi_2\in\irr(A_{\rm glo})$ of $\chi_{1,H_{\rm glo}}$ satisfying
\begin{equation}
\label{eq:Criterion 7}
\bl\left(\varphi_{1,J}\right)=\bl\left(\chi_{2,J}\right)
\end{equation}
for every $(\G^F/Z)_0\leq J\leq A_{\rm glo}[B_0]$. From \eqref{eq:Criterion 6} and \eqref{eq:Criterion 7} we obtain
\[\bl\left(\psi_{1,J\cap A_{\rm loc}}\right)^J=\bl\left(\chi_{2,J}\right)\]
for every $(\G^F/Z)_0\leq J\leq A_{\rm glo}[B_0]$. The latter equation, together with \cite[Theorem 3.5]{Mur13}, yields
\begin{align}
\bl\left(\psi_{1,J\cap A_{\rm loc}}\right)^J&=\left(\bl\left(\psi_{1,J\cap A_{\rm loc}\cap A_{\rm glo}[B_0]}\right)^{J\cap A_{\rm loc}}\right)^J\nonumber
\\
&=\left(\bl\left(\chi_{2,J\cap A_{\rm glo}[B_0]}\right)\right)^J\label{eq:Criterion 8}
\\
&=\bl\left(\chi_{2,J}\right)\nonumber
\end{align}
for every $(\G^F/Z)_0\leq J\leq A_{\rm glo}$. Finally, observe that using Assumption \ref{ass:Assumption for the criterion} (i.a) and \cite[Theorem 4.1 (d)]{Spa17} we obtain
\begin{align*}
\c_{A_{\rm glo}}((\G^F/Z)_0)&=\c_{A_{\rm glo}}((\G^F/Z)_0\times S)
\\
&\leq \epsilon^{-1}\left(\c_{(\wt{\G}^F\mathcal{A})_\chi/Z}\left(\G^F/Z\right)\right)
\\
&=\epsilon^{-1}\left(\z(\wt{\G}^F)/Z\right)
\\
&=\left(\z\left(\wt{\G}^F\right)/Z\right)_0\times S.
\end{align*}
Recalling that $\irr(\chi_{\z(\G^F)})=\irr(\psi_{\z(\G^F)})$, we obtain $\irr(\wh{\chi}_{\z(\wt{\G}^F)})=\irr(\wh{\psi}_{\z(\wt{\G}^F)})$ and hence
\begin{align}
\irr\left(\chi_{2,(\z(\wt{\G}^F)/Z)_0\times S}\right)&=\irr\left(\chi_{1,(\z(\wt{\G}^F)/Z)_0\times S}\right)\nonumber
\\
&=\irr\left(\left(\overline{\wh{\chi}}\right)_{0,(\z(\wt{\G}^F)/Z)_0}\times \iota\right)\nonumber
\\
&=\irr\left(\left(\overline{\wh{\psi}}\right)_{0,(\z(\wt{\G}^F)/Z)_0}\times \iota\right)\label{eq:Criterion 9}
\\
&=\irr\left(\psi_{1,(\z(\wt{\G}^F)/Z)_0\times S}\right).\nonumber
\end{align}
Thanks to \eqref{eq:Criterion 8} and \eqref{eq:Criterion 9}, we can apply \cite[Lemma 3.10]{Spa17} which implies
\[\left(A_{\rm glo},\left(\G^F/Z\right)_0,\overline{\chi}_0\right)\iso{(\G^F/Z)_0}\left(A_{\rm loc},\left(\n_{\G}(\L)^F/Z\right)_0,\overline{\psi}_0\right).\]
Then \eqref{eq:Criterion 2} follows by using \cite[Theorem 4.1 (i)]{Spa17}. This completes the proof.
\end{proof}

\section{Stabilizers, extendibility and consequences}
\label{sec:Stabilizers, extendibility and consequences}

In this section, we combine the results obtained in Section \ref{sec:Constructing bijections} and Section \ref{sec:Criteria} and show how to obtain the parametrization of $e$-Harish-Chandra series introduced in \eqref{eq:e-HC series, parametrization with normalizer} and its compatibility with automorphisms and Clifford theory, considered in Condition \ref{cond:Main iEBC} and Condition \ref{cond:cEBC}, by assuming certain requirements on stabilizers and extendibility of characters. This proves Theorem \ref{cor:Main bijections for nonconnected centre} and Theorem \ref{thm:Main iEBC follows from extensions} and, in particular provides a strategy to extend the parametrization given in \cite[Theorem 3.2]{Bro-Mal-Mic93} to non-unipotent $e$-cuspidal pairs in groups with disconnected centre. At the end of this section, we apply the main results of \cite{Bro-Spa20} and \cite{Bro22} and obtain some consequences for groups of type $\mathbf{A}$ and $\mathbf{C}$. 

\subsection{Proof of Theorem \ref{cor:Main bijections for nonconnected centre} and Theorem \ref{thm:Main iEBC follows from extensions}}

Let $\G$, $F:\G\to\G$, $q$, $\ell$ and $e$ as in Notation \ref{notation} with $\G$ simple of simply connected type. Let $i:\G\to\wt{\G}$ be a regular embedding compatible with the action of $F$ and consider the group $\mathcal{A}$ generated by field and graph automorphisms of $\G$ in such a way that $\mathcal{A}$ acts on $\wt{\G}^F$ (see Section \ref{sec:Automorphisms}). Let $\mathcal{K}$ the be group introduced in Section \ref{sec:Regular embeddings} and define the semidirect product $(\wt{\G}^F\mathcal{A})\ltimes \mathcal{K}$ as in the discussion following Hypothesis \ref{hyp:e-Harish-Chandra theory}.

Thanks to Remark \ref{rmk:Brauer-Lusztig for simple simply connected}, Theorem \ref{cor:Main bijections for nonconnected centre} follows from our next theorem.

\begin{theo}
\label{cor:Bijections for nonconnected centre}
Assume Hypothesis \ref{hyp:Brauer--Lusztig blocks} and Hypothesis \ref{hyp:e-Harish-Chandra theory}. Consider an $F$-stable Levi subgroup $\K$ of $\G$ and an $e$-cuspidal pair $(\L,\lambda)$ of $\K$. Suppose there exists a $((\wt{\G}^F\mathcal{A})_{\K,\L}\ltimes \mathcal{K})$-equivariant extension map for $\cusp{\wt{\L}^F}$ with respect to $\wt{\L}^F\unlhd \n_{\wt{\K}}(\L)^F$. Then, there exists a defect preserving $(\wt{\G}^F\mathcal{A})_{\K,(\L,\lambda)}$-equivariant bijection
\[\Omega_{(\L,\lambda)}^\K:\E\left(\K^F,(\L,\lambda)\right)\to\irr\left(\n_{\K}(\L)^F\enspace\middle|\enspace\lambda\right).\]
\end{theo}

\begin{proof}
Fix a $((\wt{\G}^F\mathcal{A})_{\K,(\L,\lambda)}\ltimes \mathcal{K})$-transversal $\wt{\mathbb{T}}_{\rm glo}$ in $\irr(\wt{\K}^F\mid \E(\K^F,(\L,\lambda)))$. By Theorem \ref{thm:Bijection for connected center} the set $\wt{\mathbb{T}}_{\rm loc}:=\{\wt{\Omega}_{(\L,\lambda)}^\K(\wt{\chi})\mid \wt{\chi}\in\wt{\mathbb{T}}_{\rm glo}\}$ is a $((\wt{\G}^F\mathcal{A})_{\K,(\L,\lambda)}\ltimes \mathcal{K})$-transversal in $\irr(\n_{\wt{\K}}(\L)^F\mid \lambda)$. For every $\wt{\chi}\in\wt{\mathbb{T}}_{\rm glo}$ fix an irreducible constituent $\chi\in\E(\G^F,(\L,\lambda))$ of $\wt{\chi}_{\G^F}$ and define the set $\mathbb{T}_{\rm glo}$ consisting of such characters $\chi$, while $\wt{\chi}$ runs over the elements of $\wt{\mathbb{T}}_{\rm glo}$. Similarly, for every $\wt{\psi}\in\wt{\mathbb{T}}_{\rm loc}$, fix an irreducible constituent $\psi\in\irr(\n_\K(\L)^F\mid \lambda)$ of $\wt{\psi}_{\n_\K(\L)^F}$ and define the set $\mathbb{T}_{\rm loc}$ consisting of such characters $\psi$, while $\wt{\psi}$ runs over the elements of $\wt{\mathbb{T}}_{\rm loc}$. Then $\mathbb{T}_{\rm glo}$ and $\mathbb{T}_{\rm loc}$ are $(\wt{\G}^F\mathcal{A})_{\K,(\L,\lambda)}$-transversals in $\E(\G^F,(\L,\lambda))$ and $\irr(\n_\K(\L)^F\mid \lambda)$ respectively. Fix $\chi\in\mathbb{T}_{\rm glo}$ and let $\wt{\chi}$ be the unique element of $\wt{\mathbb{T}}_{\rm glo}$ lying above $\chi$. Let $\wt{\psi}:=\wt{\Omega}_{(\L,\lambda)}^\K(\wt{\chi})\in\wt{\mathbb{T}}_{\rm loc}$ and consider the unique element $\psi$ of $\mathbb{T}_{\rm loc}$ lying below $\wt{\psi}$. This defines a bijection
\begin{equation}
\label{eq:bijection for transversals}
\mathbb{T}_{\rm glo}\to\mathbb{T}_{\rm loc}.
\end{equation}
Then, defining
\[\Omega_{(\L,\lambda)}^\K(\chi^x):=\psi^x\]
for every $x\in (\wt{\G}^F\mathcal{A})_{\K,(\L,\lambda)}$ and every $\chi\in\mathbb{T}_{\rm glo}$ corresponding to $\psi\in\mathbb{T}_{\rm loc}$ via \eqref{eq:bijection for transversals} we obtain the wanted bijection.
\end{proof}

The above result provides a way to extend \cite[Theorem 3.2 (ii)]{Bro-Mal-Mic93} and obtain a parametrization of $e$-Harish-Chandra series for non-unipotent $e$-cuspidal pairs of simple algebraic groups with (possibly) disconnected centre.

We now come to the proof of Theorem \ref{thm:Main iEBC follows from extensions}. As said before, this reduces the verification of Condition \ref{cond:Main iEBC} to questions on stabilizers and extendibility of characters. Before proceeding further, we give an exact definition of these conditions. The following should be compared to \cite[Definition 2.2]{Cab-Spa19}.

\begin{defin}
\label{def:Star conditions}
For every $e$-split Levi subgroup $\L$ of $\G$, we define the following condition.

There exists a $\wt{\L}^F$-transversal $\mathcal{T}$ in $\cusp{\L^F}$ such that:
\begin{enumerate} 
\item[(G)] For every $\lambda\in\mathcal{T}$ and every $\chi\in\E(\G^F,(\L,\lambda))$ there exists an $\n_{\wt{\G}}(\L)_{\lambda}^F$-conjugate $\chi_0$ of $\chi$ such that:
\begin{enumerate}
\item $\left(\wt{\G}^F\mathcal{A}\right)_{\chi_0}=\wt{\G}_{\chi_0}^F \mathcal{A}_{\chi_0}$, and
\item $\chi_0$ extends to $\G^F\mathcal{A}_{\chi_0}$.
\end{enumerate}
\item[(L)] For every $\lambda\in\mathcal{T}$ and every $\psi\in\irr(\n_\G(\L)^F\mid \lambda)$ there exists an $\n_{\wt{\G}}(\L)_{\lambda}^F$-conjugate $\psi_0$ of $\psi$ such that:
\begin{enumerate}
\item $\left(\wt{\G}^F\mathcal{A}\right)_{\L,\psi_0}=\n_{\wt{\G}}(\L)_{\psi_0}^F \left(\G^F\mathcal{A}\right)_{\L,\psi_0}$, and
\item $\psi_0$ extends to $\left(\G^F\mathcal{A}\right)_{\L,\psi_0}$.
\end{enumerate}
\end{enumerate}
\end{defin} 

We make a remark on the global condition of Definition \ref{def:Star conditions}. In fact, this condition is slightly stronger than condition $\rm{A}(\infty)$ of \cite[Definition 2.2]{Cab-Spa19}. However, these two conditions are equivalent under additional assumptions.

\begin{rmk}
\label{rmk:Global star}
Assume that Hypothesis \ref{hyp:Brauer--Lusztig blocks} holds for $(\G,F)$ and let $(\L,\lambda)$ be an $e$-cuspidal pair of $\G$. Set $B:=\bl(\lambda)^{\G^F}$ and suppose that either:
\begin{enumerate}
\item $\out(\G^F)_{\mathcal{B}}$ is abelian, where $\mathcal{B}$ denotes the $\wt{\G}^F$-orbit of $B$; or
\item $B$ is $\wt{\G}^F$-invariant and $(\L,\lambda)$ is $e$-Brauer--Lusztig-cuspidal.
\end{enumerate}
Then Definition \ref{def:Star conditions} (G) is equivalent to the following:
\begin{enumerate}
\item[(G')] For every $\lambda\in\mathcal{T}$ and every $\chi\in\E(\G^F,(\L,\lambda))$ there exists a $\wt{\G}^F$-conjugate $\chi_0$ of $\chi$ such that:
\begin{enumerate}
\item $\left(\wt{\G}^F\mathcal{A}\right)_{\chi_0}=\wt{\G}_{\chi_0}^F \mathcal{A}_{\chi_0}$, and
\item $\chi_0$ extends to $\G^F\mathcal{A}_{\chi_0}$.
\end{enumerate}
\end{enumerate}
\end{rmk}

\begin{proof}
Clearly Definition \ref{def:Star conditions} (G) implies (G') above. Conversely let $\chi\in\E(\G^F,(\L,\lambda))$ and consider a $\wt{\G}^F$-conjugate $\chi_1$ of $\chi$ satisfying the required properties. As explained in Remark \ref{rmk:Condition on block stability}, if $\out(\G^F)_\mathcal{B}$ is abelian, then $\chi$ also satisfies the required properties (see \cite[Lemma 4.7]{Bro-Spa21}) and we set $\chi_0:=\chi$. On the other hand by using the argument of Remark \ref{rmk:Condition on block stability}, if $B$ is $\wt{\G}^F$-invariant and $(\L,\lambda)$ is $e$-Brauer--Lusztig-cuspidal, then $\chi_1\in\E(\G^F,(\L,\lambda))$ and we set $\chi_0:=\chi_1$. This shows that there exists $\chi_0\in\E(\G^F,(\L,\lambda))$ and $x\in\wt{\G}^F$ such that $\chi_0=\chi^x$ satisfies the required properties. In particular $\chi_0\in\E(\G^F,(\L,\lambda))\cap \E(\G^F,(\L,\lambda)^x)$ and \cite[Proposition 4.10]{Ros-Generalized_HC_theory_for_Dade} implies that $(\L,\lambda)=(\L,\lambda)^{xy}$ for some $y\in \G^F$. It follows that $\chi_0=\chi^{xy}$ with $xy\in\n_{\wt{\G}}(\L)^F_\lambda$ as required by Definition \ref{def:Star conditions} (G).
\end{proof}

Moreover, we make another remark on the local condition in Definition \ref{def:Star conditions}. This should be compared with condition ${{\rm A}}(d)$ of \cite[Definition 2.2]{Cab-Spa19} with $d=e$.

\begin{rmk}
Let $\L$ be an $e$-split Levi subgroup of $\G$. Then condition ${{\rm A}}(d)$ of \cite[Definition 2.2]{Cab-Spa19} states that for every $\psi\in\irr(\n_\G(\L)^F)$ there exists an $\n_{\wt{\G}}(\L)^F$-conjugate $\psi_0$ such that:
\begin{enumerate}
\item $\left(\wt{\G}^F\mathcal{A}\right)_{\L,\psi_0}=\n_{\wt{\G}}(\L)_{\psi_0}^F \left(\G^F\mathcal{A}\right)_{\L,\psi_0}$, and
\item $\psi_0$ extends to $\left(\G^F\mathcal{A}\right)_{\L,\psi_0}$.
\end{enumerate}
Condition (L) of Definition \ref{def:Star conditions} gives a more precise description by saying that whenever $\lambda\in\irr(\L)^F$ is $e$-cuspidal and $\psi$ lies above $\lambda$, then $\psi_0$ can be chosen to lie above $\lambda$.

We point out that a detailed inspection of the argument used in all instances where condition ${{\rm A}}(d)$ of \cite[Definition 2.2]{Cab-Spa19} has been verified would actually provide the more precise condition (L) of Definition \ref{def:Star conditions}. This is explained in more details in Lemma \ref{rmk:Local star} below.
\end{rmk}

Before proving Theorem \ref{thm:Main iEBC follows from extensions}, we prove a similar result for the simpler Condition \ref{cond:cEBC}.

\begin{theo}
\label{thm:cEBC follows from extensions}
Assume Hypothesis \ref{hyp:e-Harish-Chandra theory} and Hypothesis \ref{hyp:Reduction}. Let $\L$ be an $e$-split Levi subgroup of $\hspace{1pt}\G$ and suppose that the following conditions hold:
\begin{enumerate}
\item maximal extendibility holds with respect to $\G^F\unlhd \wt{\G}^F$ and to $\n_{\G}(\L)^F\unlhd \n_{\wt{\G}}(\L)^F$;
\item there exists a $((\wt{\G}^F\mathcal{A})_{\L}\ltimes \mathcal{K})$-equivariant extension map for $\cusp{\wt{\L}^F}$ with respect to $\wt{\L}^F\unlhd \n_{\wt{\G}}(\L)^F$;
\item the requirement from Definition \ref{def:Star conditions} holds for $\L\leq \G$.
\end{enumerate}
Then Condition \ref{cond:cEBC} holds for every $e$-cuspidal pair $(\L,\lambda)$ of $\hspace{1pt}\G$.
\end{theo}

\begin{proof}
Fix an $e$-cuspidal pair $(\L,\lambda)$ of $\G$. We want to find a bijection $\Omega_{(\L,\lambda)}^\G$ as in Condition \ref{cond:cEBC}. Let $\mathcal{T}$ be the $\wt{\L}^F$-transversal in $\cusp{\L^F}$ given by Definition \ref{def:Star conditions}. Since $\G^F$-central isomorphisms of character triples are compatible with conjugation, it is no loss of generality to assume $\lambda\in\mathcal{T}$. Now Assumption \ref{ass:Assumption for the criterion, central} (iii) and (iv) hold by Definition \ref{def:Star conditions} (G) and (L) respectively, while under Hypothesis \ref{hyp:e-Harish-Chandra theory} the bijection from Assumption \ref{ass:Assumption for the criterion, central} (ii) exists by Theorem \ref{thm:Bijection for connected center}. Since we are assuming Hypothesis \ref{hyp:Reduction}, we can apply Theorem \ref{thm:Criterion, central} to conclude that Condition \ref{cond:cEBC} holds for $(\L,\lambda)$ and $\G$.
\end{proof}

The same proof can be used to obtain Theorem \ref{thm:Main iEBC follows from extensions}. Here, we prove a slightly more general result which allows us to consider a larger class of blocks. The additional block theoretic requirements are inspired by \cite[Theorem 4.1 (v)]{Cab-Spa15} and \cite[Theorem 2.4 (v)]{Bro-Spa20} and hold automatically for unipotent blocks, blocks with maximal defect and in general for every group not of type $\mathbf{A}$, $\mathbf{D}$ or $\mathbf{E}_6$ (see Remark \ref{rmk:Condition on block stability}).

\begin{theo}
\label{thm:iEBC follows from extensions}
Assume Hypothesis \ref{hyp:e-Harish-Chandra theory} and Hypothesis \ref{hyp:Reduction}. Let $\L$ be an $e$-split Levi subgroup of $\hspace{1pt}\G$, $B\in\Bl(\G^F)$ and suppose that the following conditions hold:
\begin{enumerate}
\item maximal extendibility holds with respect to $\G^F\unlhd \wt{\G}^F$ and to $\n_\G(\L)^F\unlhd \n_{\wt{\G}}(\L)^F$;
\item there exists a $((\wt{\G}^F\mathcal{A})_\L\ltimes\mathcal{K})$-equivariant extension map for $\cusp{\wt{\L}^F}$ with respect to $\wt{\L}^F\unlhd \n_{\wt{\G}}(\L)^F$;
\item the requirement from Definition \ref{def:Star conditions} holds for $\L\leq \G$;
\item the block $B$ satisfies either
\begin{enumerate}[label=(\alph*)]
\item ${\rm Out}(\G^F)_\mathcal{B}$ is abelian, where $\mathcal{B}$ is the $\wt{\G}^F$-orbit of $B$, or
\item for every subgroup $\G^F\leq J\leq\wt{\G}^F$, we have that every block $C$ of $J$ covering $B$ is $\wt{\G}^F$-invariant.\end{enumerate}
\end{enumerate}
Then Condition \ref{cond:Main iEBC} holds for every $e$-Brauer--Lusztig-cuspidal pair $(\L,\lambda)$ of $\G$ such that $\bl(\lambda)^{\G^F}=B$.
\end{theo}

\begin{proof}
Consider an $e$-Brauer--Lusztig-cuspidal pair $(\L,\lambda)$ of $\G$ as in the statement. Let $\mathcal{T}$ be the $\wt{\L}^F$-transversal in $\cusp{\L^F}$ given by Definition \ref{def:Star conditions}. Since $\G^F$-block isomorphisms of character triples are compatible with conjugation and our block theoretic hypothesis (iv) is preserved by $\wt{\G}^F$-conjugation, it is no loss of generality to assume $\lambda\in\mathcal{T}$. Now Assumption \ref{ass:Assumption for the criterion} (iii) and (iv) hold by Definition \ref{def:Star conditions} (G) and (L) respectively, while under Hypothesis \ref{hyp:e-Harish-Chandra theory} the bijection from Assumption \ref{ass:Assumption for the criterion} (ii) exists by Theorem \ref{thm:Bijection for connected center}. Finally notice that Assumption \ref{ass:Assumption for the criterion} (v) and (vi) hold by our hypothesis. Since we are assuming Hypothesis \ref{hyp:Reduction} we can apply Theorem \ref{thm:Criterion} to conclude that Condition \ref{cond:Main iEBC} holds for $(\L,\lambda)$ and $\G$.
\end{proof}

The extendibility conditions in Theorem \ref{thm:cEBC follows from extensions} (i)-(ii) and Theorem \ref{thm:iEBC follows from extensions} (i)-(ii) should be compared with condition ${\rm{B}}(d)$ of \cite[Definition 2.2]{Cab-Spa19} with $d=e$.

The reader should compare Theorem \ref{cor:Bijections for nonconnected centre} with Theorem \ref{thm:cEBC follows from extensions} and Theorem \ref{thm:iEBC follows from extensions}. We want to stress, when proving Condition \ref{cond:Main iEBC} and Condition \ref{cond:cEBC}, that the hardest task is to show that the associated character triples are $\G^F$-central isomorphic and $\G^F$-block isomorphic respectively.

\subsection{Results for groups of type $\mathbf{A}$ and $\mathbf{C}$}
\label{sec:Results for groups of type A and C}

Finally, by applying the main results of \cite{Bro-Spa20} and \cite{Bro22}, we obtain consequences of Theorem \ref{cor:Bijections for nonconnected centre}, Theorem \ref{thm:cEBC follows from extensions} and Theorem \ref{thm:iEBC follows from extensions} for groups of type $\mathbf{A}$ and $\mathbf{C}$. We start by considering Theorem \ref{cor:Bijections for nonconnected centre} which allows us to prove the following corollary.

\begin{cor}
\label{cor:Bijections for nonconnected centre, type A and C}
Let $\G$, $F:\G\to \G$, $q$, $\ell$ and $e$ be as in Notation \ref{notation} and suppose that $\G$ is simple, simply connected of type $\mathbf{A}$ or $\mathbf{C}$ and that $\ell\in\Gamma(\G,F)$ with $\ell\geq 5$. For every $e$-cuspidal pair $(\L,\lambda)$ of $\hspace{1pt}\G$, there exists an $\aut_\mathbb{F}(\G^F)_{(\L,\lambda)}$-equivariant bijection
\[\Omega_{(\L,\lambda)}^\G:\E\left(\G^F,(\L,\lambda)\right)\to\irr\left(\n_\G(\L)^F\enspace\middle|\enspace \lambda\right)\]
that preserves the $\ell$-defect of characters.
\end{cor}

\begin{proof}
We apply Theorem \ref{cor:Bijections for nonconnected centre} with $\K=\G$. Under our assumption, Hypothesis \ref{hyp:Brauer--Lusztig blocks} and Hypothesis \ref{hyp:e-Harish-Chandra theory} are satisfied by Remark \ref{rmk:Brauer-Lusztig for simple simply connected}. Suppose first that $\G$ is of type $\mathbf{A}$. Then \cite[Corollary 4.7 (b)]{Bro-Spa20} shows that there exists a $((\wt{\G}^F\mathcal{A})_\L\ltimes \mathcal{K})$-equivariant extension map with respect to $\wt{\L}^F\unlhd \n_{\wt{\G}}(\L)^F$ and therefore the result follows. On the other hand, if $\G$ is of type $\mathbf{C}$, then we obtain a $((\wt{\G}^F\mathcal{A})_\L\ltimes \mathcal{K})$-equivariant extension map for $\cusp{\wt{\L}^F}$ with respect to $\wt{\L}^F\unlhd \n_{\wt{\G}}(\L)^F$ by applying \cite[Theorem 4.1 (b)]{Bro-Spa20} whose conditions are satisfied thanks to \cite[Corollary 4.13, Proposition 4.18, Lemma 5.11, Proposition 5.18]{Bro22} (see the proof of \cite[Theorem 1.1]{Bro22} for a detailed explanation). This concludes the proof.
\end{proof}

Our next aim is to show that the bijections obtain in Corollary \ref{cor:Bijections for nonconnected centre, type A and C} can be chosen in such a way that the corresponding character triples are $\G^F$-central isomorphic and even $\G^F$-block isomorphic as predicted by Condition \ref{cond:cEBC} and Condition \ref{cond:Main iEBC} respectively. In order to apply Theorem \ref{thm:cEBC follows from extensions} and Theorem \ref{thm:iEBC follows from extensions}, we need to prove the requirements of Definition \ref{def:Star conditions}. In the next lemma we consider the local condition from Definition \ref{def:Star conditions} (L).

\begin{lem}
\label{rmk:Local star}
If $\G$ is simple, simply connected of type $\mathbf{A}$ or $\mathbf{C}$ and $\ell\in\Gamma(\G,F)$ with $\ell\geq 5$, then Definition \ref{def:Star conditions} (L) holds with respect to every $e$-split Levi subgroup $\L$.
\end{lem}

\begin{proof}
First, observe that \cite[Theorem 1.1]{Bro-Spa20} and \cite[Theorem 1.1]{Bro22} rely on the proof of \cite[Theorem 4.3]{Cab-Spa17II}, and therefore on the arguments introduced in \cite[Section 5]{Cab-Spa17I}, via an application of \cite[Theorem 4.1]{Bro-Spa20}. In particular, we focus on the argument used in \cite[Proposition 5.13]{Cab-Spa17I}. Consider $\psi\in\irr(\n_\G(\L)^F\mid \lambda)$ and notice that $\lambda$ has an extension $\wh{\lambda}\in\irr(\n_\G(\L)_\lambda^F)$ by \cite[Theorem 1.2 (a)]{Bro-Spa20} (if $\G$ is of type $\mathbf{A}$) and \cite[Theorem 1.2]{Bro22} (if $\G$ is of type $\mathbf{C}$). Using Gallagher's theorem and  the Clifford correspondence, we can write $\psi=(\wh{\lambda}\eta)^{\n_\G(\L)^F}$ for some $\eta\in\irr(\n_\G(\L)_\lambda^F/\L^F)$. By the argument of \cite[Proposition 5.13]{Cab-Spa17I}, there exists $\eta_0\in\irr(\n_\G(\L)_\lambda^F/\L^F)$ such that $\psi_0:=(\wh{\lambda}\eta_0)^{\n_\G(\L)^F}$ satisfies Definition \ref{def:Star conditions} (L.i)-(L.ii) and $\psi=\psi_0^x$ for some $x\in\n_{\wt{\G}}(\L)^F$. By the definition of $\psi_0$, we deduce that $\psi_0$ lies above $\lambda$ and therefore $\psi$ lies above $\lambda$ and $\lambda^x$. Then Clifford's theorem implies that $\lambda=\lambda^{xy}$ for some $y\in\n_\G(\L)^F$ and we conclude that $\psi=\psi_0^{xy}$ with $xy\in\n_{\wt{\G}}(\L)_\lambda^F$.
\end{proof}

Using Lemma \ref{rmk:Local star} we can prove Condition \ref{cond:cEBC} under suitable assumption. We start by considering groups of type $\mathbf{A}$.

\begin{cor}
\label{cor:cEBC for type A}
Consider $\G$, $F:\G\to \G$, $q$, $\ell$ and $e$ as in Notation \ref{notation} and suppose that $\G$ is simple, simply connected of type $\mathbf{A}$ and that $\ell\in\Gamma(\G,F)$ with $\ell\geq 5$. Let $(\L,\lambda)$ be an $e$-cuspidal pair of $\G$ and set $B:=\bl(\lambda)^{\G^F}$. If one of the following is satisfied:
\begin{enumerate}
\item $\out(\G^F)_\mathcal{B}$ is abelian, where $\mathcal{B}$ is the $\wt{\G}^F$-orbit of $B$; or
\item $(\L,\lambda)$ is $e$-Brauer--Lusztig-cuspidal and either $B$ is unipotent or $B$ has maximal defect,
\end{enumerate}
then Condition \ref{cond:cEBC} holds for $(\L,\lambda)$.
\end{cor}

\begin{proof}
We prove the statement via an application of Theorem \ref{thm:cEBC follows from extensions}. First, we notice that Hypothesis \ref{hyp:e-Harish-Chandra theory} and Hypothesis \ref{hyp:Reduction} are satisfied under our assumptions thanks to Remark \ref{rmk:Reduction}. Noticing that $\wt{\G}^F/\G^F\simeq \n_{\wt{\G}}(\L)^F/\n_\G(\L)^F$ is cyclic, \cite[Corollary 11.22]{Isa76} implies that maximal extendibility holds with respect to $\G^F\unlhd \wt{\G}^F$ and to $\n_\G(\L)^F\unlhd \n_{\wt{\G}}(\L)^F$. Moreover, as in the proof of Corollary \ref{cor:Bijections for nonconnected centre, type A and C}, we obtain a $((\wt{\G}^F\mathcal{A})\ltimes \mathcal{K})$-equivariant extension map with respect to $\wt{\L}^F\unlhd \n_{\wt{\G}}(\L)^F$ by applying \cite[Corollary 4.7 (b)]{Bro-Spa20}. It remains to check the requirements of Definition \ref{def:Star conditions}. We obtain Definition \ref{def:Star conditions} (L) by applying Lemma \ref{rmk:Local star}. To prove Definition \ref{def:Star conditions} (G), we observe that condition (G') in Remark \ref{rmk:Global star} is satisfied by \cite[Theorem 4.1]{Cab-Spa17I}. Then, in order to apply Remark \ref{rmk:Global star} we only need to show that if $B$ is unipotent or has maximal defect, then it is $\wt{\G}^F$-invariant. This follows by Remark \ref{rmk:Condition on block stability}.
\end{proof}

Next, we consider Condition \ref{cond:cEBC} for groups of type $\mathbf{C}$.

\begin{cor}
\label{cor:cEBC for type C}
Consider $\G$, $F:\G\to \G$, $q$, $\ell$ and $e$ as in Notation \ref{notation} and suppose that $\G$ is simple, simply connected of type $\mathbf{C}$ and that $\ell\in\Gamma(\G,F)$ with $\ell\geq 5$. Then Condition \ref{cond:cEBC} holds for every $e$-cuspidal pair $(\L,\lambda)$ of $\G$.
\end{cor}

\begin{proof}
We argue as in the proof of Corollary \ref{cor:cEBC for type A} and prove the result via an application of Theorem \ref{thm:cEBC follows from extensions}. By \ref{rmk:Reduction}, we notice that Hypothesis \ref{hyp:e-Harish-Chandra theory} and Hypothesis \ref{hyp:Reduction} are satisfied under our assumptions. Since $\wt{\G}^F/\G^F\simeq \n_{\wt{\G}}(\L)^F/\n_\G(\L)^F$ is cyclic, maximal extendibility holds with respect to $\G^F\unlhd \wt{\G}^F$ and to $\n_\G(\L)^F\unlhd \n_{\wt{\G}}(\L)^F$ by \cite[Corollary 11.22]{Isa76}. Moreover, as shown in the proof of Corollary \ref{cor:Bijections for nonconnected centre, type A and C}, we obtain a $((\wt{\G}^F\mathcal{A})\ltimes \mathcal{K})$-equivariant extension map for $\cusp{\wt{\L}^F}$ with respect to $\wt{\L}^F\unlhd \n_{\wt{\G}}(\L)^F$ by applying \cite[Theorem 4.1]{Bro-Spa20} together with \cite[Corollary 4.13, Proposition 4.18, Lemma 5.11, Proposition 5.18]{Bro22}. Finally, noticing that $\out(\G^F)$ is abelian, we obtain Definition \ref{def:Star conditions} (G) by applying Remark \ref{rmk:Global star} and noticing that condition (G') in Remark \ref{rmk:Global star} is satisfied by \cite[Theorem 4.1]{Cab-Spa17I}. 
\end{proof}

Finally, we prove Condition \ref{cond:Main iEBC} for certain $e$-Brauer--Lusztig-cuspidal pairs of groups of type $\mathbf{A}$ and $\mathbf{C}$. First, we deal with groups of type $\mathbf{A}$. Notice that the hypothesis of the following result is the same as the one of Corollary \ref{cor:cEBC for type A} with the additional assumption that $(\L,\lambda)$ is $e$-Brauer--Lusztig cuspidal even when considering the case where $\out(\G^F)_{\mathcal{B}}$ is abelian.

\begin{cor}
\label{cor:Type A}
Consider $\G$, $F:\G\to \G$, $q$, $\ell$ and $e$ as in Notation \ref{notation} and suppose that $\G$ is simple, simply connected of type $\mathbf{A}$ and that $\ell\in\Gamma(\G,F)$ with $\ell\geq 5$. Let $(\L,\lambda)$ be an $e$-Brauer--Lusztig-cuspidal pair of $\G$ and set $B:=\bl(\lambda)^{\G^F}$. If one of the following is satisfied:
\begin{enumerate}
\item $\out(\G^F)_\mathcal{B}$ is abelian, where $\mathcal{B}$ is the $\wt{\G}^F$-orbit of $B$;
\item $B$ is unipotent; or
\item $B$ has maximal defect,
\end{enumerate}
then Condition \ref{cond:Main iEBC} holds for $(\L,\lambda)$.
\end{cor}

\begin{proof}
We show that Condition \ref{cond:Main iEBC} holds for the $e$-cuspidal pair $(\L,\lambda)$ of $\G$ by applying Theorem \ref{thm:iEBC follows from extensions}. By the proof of Corollary \ref{cor:cEBC for type A} it only remains to verify Theorem \ref{thm:iEBC follows from extensions} (iv). Therefore, it is enough to show that if $B$ is unipotent or has maximal defect then Theorem \ref{thm:iEBC follows from extensions} (iv.b) is satisfied. This fact follows by using \cite[Proposition 13.20]{Dig-Mic91} and the results of \cite[Section 5]{Cab-Spa15} as explained in Remark \ref{rmk:Condition on block stability}. Now the result follows by recalling that $(\L,\lambda)$ is $e$-Brauer--Lusztig-cuspidal by assumption.
\end{proof}

To conclude, we consider groups of type $\mathbf{C}$ and prove Condition \ref{cond:Main iEBC} for all $e$-Brauer--Lusztig-cuspidal pairs under suitable assumptions on $\ell$.

\begin{cor}
\label{cor:Type C}
Consider $\G$, $F:\G\to \G$, $q$, $\ell$ and $e$ as in Notation \ref{notation} and suppose that $\G$ is simple, simply connected of type $\mathbf{C}$ and that $\ell\in\Gamma(\G,F)$ with $\ell\geq 5$. Then Condition \ref{cond:Main iEBC} holds for every $e$-Brauer--Lusztig-cuspidal pair $(\L,\lambda)$ of $\hspace{1pt}\G$.
\end{cor}

\begin{proof}
As in the proof of Corollary \ref{cor:Type A} we apply Theorem \ref{thm:iEBC follows from extensions}. Using the proof of Corollary \ref{cor:cEBC for type C} we therefore only need to check Theorem \ref{thm:iEBC follows from extensions} (iv). Since $\out(\G^F)$ is abelian, Theorem \ref{thm:iEBC follows from extensions} (iv.a) holds and then the results follows because $(\L,\lambda)$ is $e$-Brauer--Lusztig-cuspidal by hypothesis.
\end{proof}

\bibliographystyle{alpha}
\bibliography{References}
\vspace{1cm}

DEPARTMENT OF MATHEMATICS, CITY, UNIVERSITY OF LONDON, EC$1$V $0$HB, UNITED KINGDOM.
\textit{Email address:} \href{mailto:damiano.rossi@city.ac.uk}{damiano.rossi@city.ac.uk}
\end{document}